\documentclass[11pt]{amsart}
\usepackage{hyperref, amssymb,verbatim,graphicx}
\hypersetup{colorlinks}
\usepackage{color}
\definecolor{darkred}{rgb}{0.5,0,0}
\definecolor{darkgreen}{rgb}{0,0.5,0}
\definecolor{darkblue}{rgb}{0,0,0.5}
\hypersetup{ colorlinks,
linkcolor=darkblue,
filecolor=darkgreen,
urlcolor=darkred,
citecolor=darkblue }

\newtheorem{theorem}{Theorem}[section]

\newtheorem{assumption}{Assumption}[section]
\newtheorem{corollary}[theorem]{Corollary}

\newtheorem{proposition}[theorem]{Proposition}
\newtheorem{lemma}[theorem]{Lemma}

\theoremstyle{definition}
\newtheorem{definition}[theorem]{Definition}
\theoremstyle{remark}
\newtheorem{remark}[theorem]{Remark}
\newtheorem{notation}[theorem]{Notation}
\newtheorem{observation}{Observation}

\makeatletter
\@addtoreset{theorem}{section}
\makeatother
\newcommand\A{\mathcal{A}}

\newcommand\mC{\mathcal{C}}

\renewcommand\S{\mathcal{S}}
\newcommand{\K}{\mathcal{K}}

\newcommand{\V}{\mathcal{V}}

\newcommand{\U}{\mathcal{U}}
\newcommand{\X}{\mathcal{X}}
\newcommand{\F}{\mathcal{F}}

\newcommand{\R}{\mathbb{R}}

\newcommand{\C}{\mathbb{C}}

\newcommand{\Z}{\mathbb{Z}}

\newcommand\G{\mathcal{G}}
\renewcommand\H{\mathcal{H}}

\newcommand{\ddt}{\frac{d}{dt}}

\renewcommand{\P}{\mathbb{P}}

\newcommand\lie[1]{\mathfrak{#1}}

\newcommand{\g}{\lie{g}}

\newcommand{\on}{\operatorname}
\newcommand{\Aut}{ \on{Aut} }

\newcommand{\Ad}{ \on{Ad} } 
 
\newcommand{\End}{\on{End}}
\newcommand\Vect{\on{Vect}}

\newcommand\Id{\on{Id}}
\newcommand\grad{\on{grad}}
\newcommand\inj{\on{inj}}
\newcommand\interior{\on{int}}
\newcommand\loc{{\on{loc}}}
\newcommand\dvol{{\on{dvol}}}

\newcommand{\lra}{\longrightarrow}
\newcommand{\hra}{\hookrightarrow}
\newcommand{\weakto}{\rightharpoonup}
\newcommand{\tensor}{\otimes}

\newcommand{\ol}{\overline}

\newcommand{\delbar}{\ol{\partial}}
\newcommand\bs{\backslash}

\newcommand\Lam{\Lambda}
\newcommand\Sig{\Sigma}
\newcommand\sig{\sigma}
\newcommand\eps{\epsilon}
\newcommand\Om{\Omega}
\newcommand\om{\omega}

\newcommand{\lan}{\langle}
\newcommand{\ran}{\rangle}
\newcommand{\hh}{{\frac{1}{2}}}
\newcommand{\thh}{{\frac{3}{2}}}

\newcommand\Mod[1]{\lVert #1 \rVert}
\newcommand\bigMod[1]{\left \lVert #1 \right \rVert}
\newcommand\qu{/\kern-.7ex/} 

\begin{document}

\author[Sushmita Venugopalan]{Sushmita Venugopalan}

\title[Yang-Mills heat flow on gauged holomorphic maps]{Yang-Mills heat flow on gauged holomorphic maps}

\begin{abstract} We study the gradient flow lines of a Yang-Mills-type
  functional on a space of gauged holomorphic maps. These maps are defined on a
  principal $K$-bundle on a Riemann surface, possibly with boundary, where
   $K$ is a compact connected Lie group. The target space of the
  gauged holomorphic maps is a compact K\"{a}hler Hamiltonian $K$-manifold or a symplectic vector space with linear $K$-action and a proper moment map. We
  prove long time existence of the gradient flow. The flow lines
  converge to critical points of the functional, modulo sphere
  bubbling in $X$. Symplectic vortices are the zeros of the functional
  we study. When the base Riemann surface has non-empty boundary,
  similar to Donaldson's result in \cite{Do:bdry}, we show that there
  is only a single stratum; that is, any element of $\H(P,X)$ can be complex
  gauge transformed to a symplectic vortex. This is a version of
  Mundet's Hitchin-Kobayashi result \cite{Mund} on a surface with
  boundary.
\end{abstract}
\maketitle
\parskip 0in
\tableofcontents
\parskip .1in

\section{Introduction}
A {\em gauged holomorphic map} is a pair consisting of a connection on a principal bundle 
and a section  of an associated fiber bundle that is holomorphic with respect to that
connection. It is an equivariant generalization of $J$-holomorphic
maps on symplectic manifolds. {\em Vortices} are gauged holomorphic
maps that are zeros of a Yang-Mills-type functional, which we call the
vortex functional, and which involves the curvature of the connection
and a term involving the bundle section. Vortices play an important
role in gauged Gromov-Witten theory. This motivates us to study the
Morse theory of the vortex functional on the space of gauged
holomorphic maps. The base manifold is a connected Riemann surface with metric,
and we study both the case that this Riemann surface is closed, and
the case that it has non-empty boundary.  We show that the gradient
flow of the functional exists for all time and has an infinite time
limit. This limit is a vortex in the case the base manifold has
boundary. But when the base manifold is closed, there are multiple
Morse strata. However, if at the starting point of the flow line, the
value of the vortex functional is low enough, then we are guaranteed
that the flow converges to a vortex.

Our set-up is an infinite dimensional analog of the abstract setting
of Kirwan's thesis \cite{Kirwan}, which we briefly describe. Let $K$
be a compact connected Lie group. Let $(X,\om)$ be a compact K\"ahler
$K$-Hamiltonian manifold with a {\em moment map} $\Phi:X \to
\mathfrak{k}^*$. Suppose $G$ is a complex reductive group which contains $K$
as a maximal compact subgroup. Since $X$ is K\"ahler, the $K$-action
extends to a holomorphic $G$-action. Suppose $X$ has the structure of
a polarized projective variety. The Geometric Invariant Theory (git)
quotient is the quotient of the semistable locus $X^{\on{ss}}$ under the
$G$-orbit closure relation. By the Kempf-Ness theorem in \cite{KN},
the git quotient coincides with the symplectic quotient
$\Phi^{-1}(0)/K$.  On the symplectic side, we consider the gradient
flow of the function $|\Phi|^2$ under the K\"ahler metric. By results
in \cite{Kirwan}, the flow induces a $G$-invariant stratification of
the manifold $X$.  The stratum to which a point belongs is given by
the infinite time limit of its gradient flow. The largest stratum,
corresponding to points that flow to $\Phi^{-1}(0)$, is open in $X$
and coincides with the semistable locus $X^{\on{ss}}$. Further each of the
higher Morse strata can be characterized algebraically. The strata of
a point $x \in X$ is given by a one-parameter subgroup of $G$ which is
maximally destabilizing for $x$.

The ideas in the abstract setting of Kirwan have been applied to many
infinite dimensional problems and the Morse theoretic approach has
been fruitful.  The work of Atiyah and Bott \cite{AB} introduces the
above ideas on the space of connections $\A$ on a Hermitian vector
bundle over a Riemann surface. On the symplectic side there is a Morse
stratification coming from the flow of the Yang-Mills functional $\A
\ni A \mapsto \Mod{F_A}_{L^2(\Sig)}$, which coincides with a
holomorphic stratification coming from the Harder-Narasimhan type of
the holomorphic structures on the vector bundle. The analog of the
Kempf-Ness theorem in this case is the Narasimhan-Seshadri theorem in
\cite{NS}. This theorem, as re-stated by Donaldson in \cite{Do:NS},
says that every stable bundle admits a Yang Mills connection that
assumes the minimum value of the functional. 
Daskalopoulos \cite{Daska} proved that the holomorphic and Morse
stratifications agree. R{\aa}de \cite{Rade} proved the same result by
a different approach. He analytically proved the existence of gradient
flow lines of the Yang-Mills functional and examined their convergence
properties at infinite time. Similar Morse-theoretic ideas have been
applied to the space of holomorphic vector bundles equipped with some
extra data. For example, Wilkin \cite{Wilkin} studies the space of
Higgs pairs $(A,\phi)$, where $A$ is a connection on a complex vector
bundle $E$ over a Riemann surface, and $\phi$ is an $E$-valued
$(1,0)$-form such that $\delbar_A \phi=0$. With a standard choice of
symplectic structure, the action of the gauge group has moment map
$F_A+[\phi, \phi^*]$. The paper shows that the Morse stratification
obtained under the $L^2$ norm of the moment map corresponds to a
holomorphic stratification. Our set-up can be thought of as a
generalization of Higgs-bundles. Instead of a vector bundle $E$, we
consider fiber bundles whose fibers are K\"ahler manifolds.

We now describe the set-up of the paper rigorously and state the main
results.  Let $P \to \Sig$ be a principal $K$-bundle on a compact
connected Riemann surface $\Sig$ with metric, and possibly with boundary. The target
manifold $X$ is a compact K\"ahler Hamiltonian $K$-manifold described above.
A {\em gauged holomorphic map} from $P$ to $X$ is a pair $(A,u)$
consisting of a connection $A$ on the $K$-bundle $P \to \Sig$ together
with a holomorphic section of the associated fiber bundle $P(X) := (P
\times X)/K$. The complex structure on $P(X)$ is given by the complex
structure on $\Sig$ and $X$ and the connection $A$. The space of
gauged holomorphic maps $\H(P,X)$ has a formal Hamiltonian action of
the group of gauge transformations $\K(P)$. The moment map is given by
$*F_A + \Phi(u)$, where $F_A$ is the curvature of $A$. The {\em vortex
  functional} is the $L^2$-norm square of the moment map:
\begin{equation}\label{eq:functional}
  \H(P,X) \to \R, \quad (A,u) \mapsto \Mod{*F_A+\Phi(u)}_{L^2}^2.
\end{equation}
We study the long-time existence and convergence behaviour of the
downward gradient flow trajectories of the vortex functional. The gradient flow
 trajectory starting at a gauged holomorphic map
$(A_0,u_0)$ is a time-dependent pair $(A,u):[0,\infty) \mapsto \A(P) \times \Gamma(\Sig,P(X))$ which satisfies the equations
\begin{equation} \label{eq:evolve}
  \begin{split}
    \ddt A &= -*{d}_AF_{A,u}, \quad \ddt u = -J_X(F_{A,u})_u, \quad F_{A,u}|_{\partial \Sig}=0, \\
    A(0)&=A_0, \quad u(0)=u_0.
  \end{split}
\end{equation} 
In the above equations, for a gauged holomorphic map $(A,u)$, $F_{A,u}:=*F_A +
\Phi(u) \in \Gamma(\Sig,P(\mathfrak{k}))$, and for a section $\xi \in
\Gamma(\Sig,P(\mathfrak{k}))$, the vector field $\xi_u \in
\Gamma(\Sig,u^*T^{\on{vert}}(P(X)))$ is given by
$\xi_u(s):=\xi(s)_{u(s)}$ for all $s \in \Sig$. This system of
equations is a non-linear perturbation of the Yang-Mills gradient flow
equation, which in turn can be realized as a perturbation of the heat
equation. So we alternately refer to the solution of \eqref{eq:evolve}
as the {\em heat flow}. Our results are as follows.
\begin{theorem} \label{thm:flowintro} {\rm(Long time existence of
    gradient flow)} Suppose $\Sig$, $P$, $K$ and $X$ are as above. The
  gradient flow equation \eqref{eq:evolve} for the vortex functional
  has a solution for all time $(A_t,u_t) \in C^0_{loc}([0,\infty),H^1
  \times C^0)$. There is a family of gauge transformations $g_t \in
  H^2(\K)$ so that $g_t(A_t,u_t)$ is smooth on $[0,\infty) \times
  \Sig$ away from the corner $\{0\} \times \partial \Sig$.
\end{theorem}

The flow lines converge to a critical point of the functional, but the
convergence is modulo bubbling in the fibers in $P(X)$.
\begin{theorem}\label{thm:convintro} {\rm(Convergence of flow)} Suppose the gauged holomorphic map $(A_0,u_0)$ satisfies $\Phi(u_0)|_{\partial \Sig}=0$.
Let $(A_t,u_t)$ be the smooth
  gradient flow (modulo gauge) starting from the pair $(A_0,u_0)$
 calculated in Theorem
  \ref{thm:flowintro}. Then,
there exists a sequence $t_i \to \infty$, a
  sequence of unitary gauge transformations $k_i \in \K_{H^3}$, a limit pair
  $(A_\infty,u_\infty) \in \A(P)_{H^2} \times \Gamma(\Sig,P(X))_{C^1}$
  and a finite bubbling set $Z \subset \Sig$ (see Definition
  \ref{def:bubset}) such that
  \begin{enumerate}
  \item \label{part:conva} $k_i(A_{t_i}) \to A_\infty$ weakly in $H^2$.
  \item \label{part:convb} If $\Sig$ does not have boundary,
    $k_iu_{t_i}$ Gromov converges to a nodal gauged holomorphic map
    with principal component $u_\infty$. In compact subsets of $\Sig
    \bs Z$, $k_iu_{t_i} \to u_\infty$ in $C^1$.
  \item \label{part:convc} If $\Sig$ has boundary, $k_iu_{t_i} \to u_\infty$ in $C^1(\Sig)$ and
    there is no bubbling.
  \item \label{part:convd} The limit $(A_\infty,u_\infty)$ is a critical point of the functional
    \eqref{eq:functional}, i.e.\linebreak $d_{A_\infty}F_{A_\infty,u_\infty}=0$
    and $(F_{A_\infty,u_\infty})_{u_\infty}=0$.
  \end{enumerate}
\end{theorem}
In part \eqref{part:convb} above, $(A_\infty,u_\infty)$ is the
principal component of the limit. In addition, there would sphere
bubble trees in the fibers $P(X)_z$, for the points $z$ in the
bubbling set $Z$. A much stronger result can be obtained in the case
that $\Sig$ has boundary.

\begin{theorem} \label{thm:bdryintro} {\rm(Unique limit when $\Sig$
    has boundary)} Suppose $\partial \Sig \neq \emptyset$ and $p>2$.
\begin{enumerate}
\item \label{part:bdrya} The limit
  $(A_\infty,u_\infty)$  computed in Theorem \ref{thm:convintro} is a vortex and lies
  in the same complex gauge orbit as the flow line $(A_t,u_t)$. 
\item \label{part:bdryb} For a given
  flow line $(A_t,u_t)$, the limit $(A_\infty,u_\infty)$ is unique up
  to up to gauge transformations. 
\end{enumerate}
Therefore, there is a unique $\xi \in \Gamma(\Sig, P(\mathfrak{k}))_{W^{2,p}}$ such that $\xi|_{\partial \Sig}=0$ and\linebreak $e^{i\xi}(A_0,u_0)$ is a vortex.
\end{theorem}

Theorem \ref{thm:bdryintro} implies that there is a single stratum for
the flow of the vortex functional when the base manifold has boundary.
This result can be compared to Donaldson's result \cite{Do:bdry} on
Yang-Mills gradient flow. On a two-dimensional base manifold with
boundary, it says that any connection can be complex gauge transformed
to a flat connection -- there is no semi-stability condition
involved. But, the proof is different because in the case of gauged
holomorphic maps, our result overcomes the additional difficulty of
ruling out the formation of bubbles in the fibers of $P(X)$.  If
the base manifold $\Sig$ is closed there are multiple strata. However,
if in addition, we know that the flow line starts with a low enough
value of the vortex functional, then we can get a similar result to
Theorem \ref{thm:bdryintro}.

\begin{theorem}\label{thm:closedintro}
  Let $p>2$ and $\partial \Sig=\emptyset$. Assume that the $K$-action
  on $\Phi^{-1}(0)$ has finite stabilizers and the energy of the
  gauged holomorphic map $(A_0,u_0)$ is bounded by
  $c_0^2\on{vol}(\Sig)$, where the constant $c_0(X)$ is defined in
  Section \ref{subsec:uniquelimit}. Suppose $(A_\infty,u_\infty)$ is
  the limit of the heat flow trajectory starting at $(A_0,u_0)$
  produced by Theorems \ref{thm:flowintro} and
  \ref{thm:convintro}. Then, $(A_\infty,u_\infty)$ is a vortex, it is
  in the complex gauge orbit of $(A_0,u_0)$ and it is unique up to
  gauge transformations. Therefore, there is a unique element
  $\xi \in \Gamma(\Sig,P(\mathfrak{k}))_{W^{2,p}}$ such that $e^{i\xi}(A_0,u_0)$ is a
  vortex.
\end{theorem}

The results of Theorems \ref{thm:bdryintro} and \ref{thm:closedintro} are
based on the fact that the limit vortex does not have any
infinitesimal stabilizers of the gauge group action. The situation is
similar to the finite dimensional Hamiltonian manifold with the 
stable=semistable assumption, which implies that the semistable orbits
are closed. Analogously, in our case, the complex gauge orbits containing the flow line is closed, which lets us prove that the limit vortex is also contained in that orbit. As a corollary, we
then rule out bubbling.
 
All the above results continue to hold if the target manifold $X$ is a symplectic vector space with a linear group action and proper moment map.
\begin{theorem}\label{thm:vstarget} Suppose $X$
  is a symplectic vector space with linear action of the group $K$ and
  a proper moment map. If $\Sig$ has boundary we additionally assume that $u_0|_{\partial \Sig} \subset \Phi^{-1}(0)$. 
 Then the results in Theorems
  \ref{thm:flowintro}--\ref{thm:closedintro}
  hold.
\end{theorem}

Our work has important applications in the study of symplectic
vortices. The earliest literature on vortices considers the case of
vector space targets - for example Jaffe-Taubes \cite{JT},
Garcia-Prada \cite{garcia:direct}, Bradlow \cite{Brad:stability}
etc. When the target is a general symplectic manifold, Cieliebak et
al. \cite{CGMS} construct a moduli space of vortices on a compact base
manifold in the absence of sphere bubbling. Ott \cite{Ott} constructed
a compactification of the moduli space that incorporates sphere
bubbling. The phenomenon arising from allowing the base curve to vary
is studied by Mundet-Tian \cite{MundetTian}. In the large area limit,
vortices are related to $J$-holomorphic curves in the symplectic
quotient. Then, the gauged Gromov Witten invariants of $X$ are related
to the Gromov-Witten invariants on the symplectic quotient $X \qu G$
via a quantum-Kirwan morphism, which was proposed by Gaio-Salamon
\cite{GS} and defined in greater generality by Woodward
\cite{W:qkirwan}. The
quantum-Kirwan morphism is defined by counting vortices on the complex
plane $\C$. A compactification of the space of vortices on $\C$ was
constructed by Ziltener \cite{Zilt:thesis}, \cite{Zilt:QK}. An
alternate way of defining gauged Gromov-Witten theory is via quasimaps
in Ciocan-Fontanine-Kim-Maulik \cite{CKM:quasimap}.  The symplectic
version of quasimaps is given by vortices defined on Riemann surfaces that have
infinite cylindrical ends considered by the author in \cite{Venu:cyl}.

An alternate way to study the moduli space of symplectic vortices is
via Hitchin-Kobayashi (HK) correspondences, which establishes a
bijection (or homeomorphism, in some cases) between vortices and
isomorphism classes of semistable gauged maps. The moduli space of
the latter object can be studied from an algebro-geometric point of
view. Such a correspondence was first provided in the case of vortices
with vector space target, as in Jaffe-Taubes, Bradlow. Mundet provided
a HK correspondence for vortices on compact Riemann surfaces. Theorem
\ref{thm:bdryintro} stated above can be seen as a version of Mundet's
result for a Riemann surface with boundary. In that case, any gauged
holomorphic map $(A,u)$ contains a symplectic vortex in its orbit,
that is unique up to unitary gauge transformations. In fact, the
vortex is the limit of the heat flow trajectory that starts at
$(A,u)$. This shows that there is no semistability conditions when
$\Sig$ has boundary. In case $\Sig$ is closed, Theorem
\ref{thm:closedintro} provides a sufficient condition for a gauged
holomorphic map to be semistable. This condition is easier to check
than Mundet's semistability condition. We point out that Ling Lin
\cite{LingLin} has also studied heat flow on closed surfaces and used
it to give an alternate proof of Mundet's HK correspondence.

The most important application of our results is for proving HK
correspondences for vortices on Riemann surface with infinite
volume. In case $\Sig=\C$, this is done in joint work with Woodward
\cite{VW:affine}, and is proved by applying Theorem
\ref{thm:closedintro} on $\P^1$ equipped with a sequence of increasing
metrics. This result is crucially used to define the quantum-Kirwan
morphism in \cite{W:qkirwan}. In a similar way, the author has also obtained a HK
correspondence for vortices on curves with infinite cylindrical ends in \cite{Venu:cyl},
resulting in a homeomorphism to the moduli space of quasimaps.

We expect that the boundary result Theorem \ref{thm:bdryintro} can be
used to prove a classification result for vortices defined on the
complex half plane, analogous to the result for affine vortices in
\cite{VW:affine}. Analogous to the closed case, vortices on the half
plane will be required in the definition of an open quantum Kirwan
morphism proposed by Woodward \cite{W:GF}. This morphism would play an
important role in open gauged Gromov-Witten theory which has been
studied by Frauenfelder \cite{Frau:AG}, Woodward \cite{W:GF} and Xu
\cite{Xu:GF}, \cite{Xu:Lagcpt}. For the
convergence of flow in the case of a base manifold with boundary, we
need an additional hypothesis that the maps $u_t$ map the boundary of
the domain $\partial \Sig$ to the zero level set of the moment
map. This is a natural assumption in open gauged Gromov-Witten theory,
where one studies gauged maps $(A,u)$ on Riemann surfaces with
boundary that map the boundary $\partial \Sig$ to a $K$-invariant
Lagrangian $L \subset \Phi^{-1}(0) \subset X$.

Another application of heat flow in the case of a closed base manifold
is to provide a Morse stratification on the space of gauged
holomorphic maps. This goal is not achieved in this paper, and is
a subject of future research. 

The proof of the existence of heat flow uses similar techniques as
R{\aa}de \cite{Rade}. The main point of difference is that our flow
problem involves $u$ which is a map to a compact K\"{a}hler
manifold. While solving the flow equations, we assume that $u(t)$ is
in $C^0$, but in the time direction, we assume its regularity is in a
Sobolev class. We have to address some issues in defining such a mixed
space. The reason why it is necessary to have $u_t$ in $C^0$, is
because the perturbative lower order terms in the parabolic flow
equations involve composition of functions, and we need $u_t \in C^0$
to use the estimates for such cases. Donaldson \cite{Do:ASD} gives a
simpler way of obtaining flow lines, albeit modulo gauge. But this
approach does not work for us because of the non-linear moment map
term. However, after showing the existence of flow, we adapt the
technique in \cite{Do:ASD} to show that our flow is smooth in time and
space directions modulo gauge.

This paper is organized as follows: Section \ref{sec:prelim}
describes connections,\linebreak gauged holomorphic maps etc. Section
\ref{sec:heatflow} proves Theorem \ref{thm:flowintro} --- the long-time
existence of gradient flow and its regularity properties. Section
\ref{sec:conv} discusses the convergence behavior of gradient flow
trajectories and proves Theorems~\ref{thm:convintro},
\ref{thm:bdryintro} and \ref{thm:closedintro}. The results are extended to the case of vector
space target in Section \ref{sec:vstarget}.  Sections
\ref{sec:sobolevspaces} and \ref{sec:compfunc} carefully describe the
Sobolev spaces and their properties used in Section
\ref{sec:heatflow}.

\subsubsection*{Acknowledgements} This paper arose from my PhD thesis. I am
grateful to my advisor Chris Woodward for his support, encouragement
and guidance. I also acknowledge Chennai Mathematical Institute, where I was a post-doctoral Fellow at the time of writing a revised version of this paper.

\section{Preliminaries}\label{sec:prelim}
\subsection{Hamiltonian actions on K\"ahler
  manifolds}\label{subsec:kahham}
Let $K$ be a compact connected Lie group and let $(X,\om,J)$ be a
K\"ahler $K$-Hamiltonian manifold. This means the $K$-action on $X$
preserves the K\"ahler form $\om$ and the complex structure
$J$. Further, the action has a {\em moment map} $\Phi : X \to \mathfrak{k}^*$
that is equivariant and satisfies $\iota(\xi_X) \omega = {d} \lan \Phi,
\xi \ran , \ \forall \xi \in \mathfrak{k}$, where $\xi_X \in \Vect(X)$ is given
by the infinitesimal action of $\xi$ on $X$. Since $K$ is compact,
$\mathfrak{k}$ has an $Ad$-invariant metric. We fix such a metric and identify
$\mathfrak{k}$ with $\mathfrak{k}^*$ and so the moment map is a map from $X$ to $\mathfrak{k}$. The
gradient vector field of the the function $\hh |\Phi|^2$ on $X$ with
respect to the K\"ahler metric is $\grad |\Phi|^2(x)=J\Phi(x)_X$. This
can be seen as follows. For $v \in T_xX$,
$$\hh \lan \on{grad} |\Phi|^2, v\ran_g=\lan {d}\Phi(v),\Phi\ran_{\mathfrak{k}} = \iota_{\Phi(x)_X}\omega(v)=\lan v,J\Phi(x)_X\ran_g,$$
where $g$ is the Riemannian metric $\om(\cdot, J\cdot)$ on $X$.

The {\em complexified Lie group} of $K$, denoted by $G$, is a complex
reductive group that contains $K$ as a maximal compact subgroup and
whose Lie algebra is the complexification of the Lie algebra of $K$,
i.e. $\g=\mathfrak{k} \oplus i\mathfrak{k}$.  A compact Lie group has a complexification ---
see Hochs \cite{Hochs}, p205. Further, there is a diffeomorphism (see
Helgason \cite[VI.1.1]{Helg2})
\begin{equation}
  \label{eq:cartan}
  K \times \mathfrak{k} \to G, \quad (k,s) \mapsto ke^{is}.
\end{equation}
On a K\"{a}hler manifold the action of $K$ extends to a unique
holomorphic action of $G$ (see \cite{GS:quant}). Since the gradient of
$\hh \Mod{\Phi}^2$ is $J\Phi(x)_X$, the gradient flow preserves the
$G$ orbit, and therefore, the Morse strata of $\hh \Mod{\Phi}^2$ are
$G$-invariant.
\begin{lemma} \label{lem:uniquemu0} A $G$-orbit in $X$ has at most one
  $K$-orbit on which $\Phi=0$.
\end{lemma}
\begin{proof}
  Suppose $x \in X$ and $g \in G$ are such that
  $\Phi(x)=\Phi(gx)=0$. By \eqref{eq:cartan}, we can write
  $g=ke^{is}$, where $k \in K$ and $s \in \mathfrak{k}$. Since $\Phi^{-1}(0)$ is
  $K$-invariant, we can assume $g=e^{is}$. Now, for $0 \leq t \leq 1$,
  \begin{align}\label{eq:mommapconv}
    \ddt \lan \Phi(e^{its}x),s \ran &= \lan
    s_X(e^{its}x),s_X(e^{its}x) \ran \geq 0,
  \end{align}
  which implies that $s_X(e^{its}x)=0$ for all $t \in [0,1]$, and
  therefore $x=e^{its}x$.
\end{proof}

\subsection{The space of connections}\label{subsec:conn}
Let $(\Sig, j_\Sig, g_\Sig)$ be a Riemann surface with metric and $P
\to \Sig$ a principal $K$-bundle over it. A {\it connection} $A$ on
$P$ is a $\mathfrak{k}$-valued 1-form on $P$ that is $K$-equivariant and
satisfies $A(\xi_P)=\xi$. Here $K$-equivariance means that
$A_{pk}(vk)=\Ad_{k^{-1}} A_p(v)$ for all $p \in P$, $v \in T_pP$ and
$k \in K$. For $\xi \in \mathfrak{k}$, $\xi_P$ is a vector field on $P$ defined
as $\xi_P(p):=\ddt (p\exp(t\xi))$ for $p \in P$.  Let $\A(P)$ denote
the space of all connections. It is an affine space modeled on
$\Omega^1(\Sig,P(\mathfrak{k}))$, where $P(\mathfrak{k}):=(P \times \mathfrak{k})/K$ is the adjoint
bundle. The form ${d} A+\hh[A \wedge A] \in \Om^2(P,\mathfrak{k})$ is basic, so
it descends to a two-form $F_A \in \Om^2(\Sig,P(\mathfrak{k}))$, which is the
{\em curvature} of the connection $A$.  A {\em gauge transformation}
is an automorphism of $P$ -- it is an equivariant bundle map $P \to
P$. The group of gauge transformations on $P$ is denoted by $\K(P)$. A
gauge transformation can be viewed as a section $k:\Sig \to P \times_K
K$ on the associated bundle $P \times_K K$, where $K$ acts on itself
by conjugation. A gauge transformation $k \in \K(P)$ acts on the space
of connections by pullback by $k^{-1}$. The infinitesimal action of
the element $\xi \in \Gamma(\Sig,P(\mathfrak{k}))$ on a connection $A$ is $-{d}_A
\xi$.

On a trivial bundle $P=\Sig \times K$, there is a trivial connection,
denoted by ${d}$ and the adjoint bundle has a trivialization
$P(\mathfrak{k})\simeq \Sig \times \mathfrak{k}$. The space $\A$ is then equal to ${d} +
\Om^1(\Sig,\mathfrak{k})$. The formula of curvature is $F_A={d} a + \hh[a \wedge
a].$ A gauge transformation $k:\Sig \to K$ acts on the connection
$A={d} + a$ as
\begin{equation}
  \label{eq:gatran}
  k(A)= {d}+({d} k k^{-1} + \Ad_ka).
\end{equation}

Given a $K$-manifold $M$, there is an associated bundle $P(M):=(P \times
M)/K$, which is a fiber bundle on $\Sig$ with fibers diffeomorphic to
$M$.  A connection $A$ defines a covariant derivative ${d}_A$ on the
space of sections $\Gamma(\Sig,P(M))$. On a local trivialization of
$P$, $d_A$ is given by
$$\Gamma(\Sig,P(M)) \ni u \mapsto {d}_A u:={d} u+ a_u\in \Om^1(\Sig,u^*T^{\on{vert}}P(M)).$$
At a point $z \in \Sig$, $a_u(z)$ is the infinitesimal action of
$a(z)$ at $u(z)$. In the particular case when $M$ is $\mathfrak{k}$ with the
adjoint $K$-action, the covariant derivative of a section $\xi \in
\Gamma(\Sig,\mathfrak{k})$ is ${d}_A\xi:={d}\xi + [a,\xi]$.

In the next 2 paragraphs, we show that the space of connections is an
infinite-dimensional K\"ahler manifold with a formal Hamiltonian
action of the group of gauge transformations. First, we show that a
connection determines a holomorphic structure on associated fiber
bundles, through which we get a complex structure on $\A(P)$. If $M$
is a complex $K$-manifold, the associated fiber bundle $P(M)$ has a
holomorphic structure given by the operator $\delbar_A$ corresponding
to a connection $A \in \A(P)$, where
$$\delbar_A:\Gamma(\Sig,P(M)) \to \Om^{0,1}(\Sig,
u^*T^{\on{vert}}P(M)), \quad u \mapsto ({d}_Au)^{0,1}.$$ The almost
complex structure corresponding to the operator $\delbar_A$ is indeed a
holomorphic structure on $P(M)$ because: by the Newlander-Nirenberg
theorem, the condition for the above almost complex structure to be
integrable is that $\delbar_A^2=0$ (see \cite{AB} p.555 or \cite{DoKr}
Theorem 2.1.53), which is same as saying the $(0,2)$-part of the
curvature $F_A$ vanishes, but this is vacuously true in our case since
the base manifold is a Riemann surface. In particular, taking $M=G$ on
which $K$ acts by left multiplication, produces a holomorphic
$G$-bundle $P_\C$. Let $\mC(P)$ denote the space of holomorphic
structures on $P_\C$.  The above construction yields a map from
$\A(P)$ to $\mC(P)$, which we claim is a bijection. Given a
holomorphic bundle $P_\C$, a choice of a section $\sig: \Sig \to
P_\C/K$ gives a principal $K$-bundle $P$ by pullback of the bundle
$P_\C \to P_\C/K$, so that $P$ is naturally a submanifold of $P_\C$.
The intersection $TP \cap J(TP)$ defines a connection in $TP$
\cite{Singer}. The correspondence between connections and holomorphic
structures gives an infinitesimal isomorphism
$$ T_A \A =\Om^1(\Sig, P(\mathfrak{k})) \to T_C\mC=\Om^{0,1}(\Sig, P(\mathfrak{k})), \quad
a \mapsto a^{0,1}.$$
The complex structure on $\mC$ pulls back to a complex structure on
$\A$ given by $J_\A a= a \circ j_\Sig$.

The space of connections $\A(P)$ has a symplectic form on it: for $A
\in \A$ and $a$, $b \in T_A\A=\Omega^1(\Sig,P(\mathfrak{k}))$, $\om_\A(a,b):=
\int_\Sig \lan a \wedge b \ran_{\mathfrak{k}}$. The symplectic form is compatible
with the complex structure $J_\A$. Let $\K(P)_\partial$ be the
subgroup of $\K(P)$ consisting of gauge transformations $k$ that are
identity on the boundary, i.e. $k|_{\partial \Sig}=\Id$.  The action
of $\K(P)_\partial$ on $\A(P)$ is formally Hamiltonian and has a
moment map $A \mapsto *F_A \in \Gamma(\Sig,P(\mathfrak{k}))$. We verify the
moment map condition: for any $a \in T_A\A:=\Om^1(\Sig,P(\mathfrak{k}))$ and
$\xi \in \on{Lie}(\K(P)_\partial)$,
\begin{equation*}
  \iota_{\xi_\A}\om_\A(a)=\int_\Sig\lan -{d}_A\xi,a\ran =\int_\Sig \lan *{d}_Aa,\xi\ran ={d}_\A\lan *F_A,\xi\ran(a).
\end{equation*}
The second equality comes from integration by parts and the fact that
$\xi|_{\partial \Sig}=0$.  Observe that the $L^2$-norm square of the
moment map is the Yang-Mills functional $A \mapsto
\Mod{F_A}_{L^2(\Sig)}^2$. Analogous to the finite dimensional case in
Section \ref{subsec:kahham}, the gradient of this functional is
\begin{align*}
  \on{grad}(A \mapsto \Mod{F_A}_{L^2(\Sig)}^2) = J_\A(-
  d_A(*F_A))=d_A^*F_A.
\end{align*}

The {\em group of complex gauge transformations} $\G(P)$ consists of
automorphisms of the associated bundle $g:P_\C \to P_\C$, which can be
viewed as sections of the bundle $P \times_K G$, where $K$ acts on $G$
by left multiplication. By the Cartan diffeomorphism
\eqref{eq:cartan}, a complex gauge transformation $g$ can be written
as $ke^{i\xi}$, where $k \in \K(P)$ and $\xi \in
\Gamma(\Sig,P(\mathfrak{k}))$. The group $\G(P)$ acts on the space of
holomorphic structures on $P_\C$ via pullback. Through the isomorphism
$\A(P) \to \mC(P)$ the action of $\G(P)$ on $\mC(P)$ pulls back to an
action on $\A(P)$. The isomorphism $\A(P) \to \mC(P)$ is
$\K(P)$-equivariant and so the complexified gauge group action on
$\A(P)$ extends the action of the unitary gauge group $\K(P)$. For any
$\xi \in \Gamma(\Sig,P(\mathfrak{k}))$, the infinitesimal action of $i\xi$ on a
connection $A$ is $-{d}_A \xi \circ j_\Sig$, which is equal to
$J_\A(-{d}_A \xi)$ and can also be re-written as $*{d}_A\xi$.

\subsection {Gauged holomorphic maps}\label{subsec:ghm}
The concept of $J$-holomorphic curves extends to the equivariant
setting by introducing a connection into the picture.  A gauged
holomorphic map from $P$ to $X$ is a pair $(A,u)$ consisting of a
connection $A$ on $P$ and a section $u:\Sig \to P(X)$ that is
holomorphic with respect to $\delbar_A$. The space of gauged
holomorphic maps from the principal bundle $P \to \Sig$ to target $X$
is called $\H(P,X)$. The group of complex gauge transformations
$\G(P)$ acts on gauged holomorphic pairs diagonally:
$g(A,u)=(g(A),gu)$. We remark that this action preserves
holomorphicity because $\delbar_{g(A)}(gu)=(g \circ \delbar_A \circ
g^{-1})(gu)=g(\delbar_A u)$.

Analogous to the space of connections, the action of unitary gauge
transformations on the space of pairs $(A,u)$ in $\A(P) \times
\Gamma(\Sig,P(X))$ is formally Hamiltonian. We say `formal' because we
have not given this space a manifold structure. The symplectic form is
\begin{align*}
  \lan(a_1,\xi_1), (a_2,\xi_2)\ran \mapsto \int_\Sig (a_1 \wedge a_2 + \om_X(\xi_1,\xi_2) \dvol_\Sig),
\end{align*}
  where $(a_i,\xi_i) \!\in\! T_{(A,u)}(\A(P) \!\times\! \Gamma(P(X)))
  \!=\! \Om^1(\Sig,P(\mathfrak{k})) \!\times\! \Gamma(\Sig,u^*T^{\on{vert}}X)$.
The formal moment map for the $\K(P)_\partial$-action is
\begin{align*}
  \A(P) \times \Gamma(\Sig,P(X)) \to \Gamma(\Sig,P(\mathfrak{k})), \quad (A,u)
  \mapsto F_{A,u}:=*F_A+\Phi(u).
\end{align*}
Since $\Phi$ is $K$-equivariant, it induces a map $P(X) \to P(\mathfrak{k})$,
which is also denoted $\Phi$, so that $\Phi(u)$ in the above
definition is a section of $P(\mathfrak{k}) \to \Sig$.

Analogous to the finite-dimensional case and the case of the space of
connections $\A(P)$, we consider the norm square of the moment map
$(A,u) \mapsto \Mod{F_{A,u}}_{L^2(\Sig)}^2$. The gradient at $(A,u)$
is
\begin{align}\label{eq:gradgh}
  J_{(A,u)}(F_{A,u})_{\A(P) \times
    \Gamma(P(X))}=(*d_AF_{A,u},J_X(F_{A,u})_u).
\end{align}
This justifies the right hand side of the heat flow equation
\eqref{eq:evolve}.  We recall the notation that given $\xi \in
\Gamma(P(\mathfrak{k}))$, $\xi_u \in u^*T^{vert}P(X)$ denotes the action of
$\xi$ on the image of $u$, i.e. for $z \in \Sig$,
$\xi_u(z)=\xi(z)_{u(z)}$. The gradient flow preserves $\G$ orbits and
so it preserves $\H(P,X)$.

\section{Heat flow}\label{sec:heatflow}
\subsection{Existence of trajectories on compact target
  manifold}\label{subsec:flowexist}
In this section, we prove the long-term existence of the gradient flow
of the vortex functional, given by \eqref{eq:evolve}, when the target
is a compact K\"ahler Hamiltonian manifold.
\subsubsection{Setting up the system of equations for gradient flow}
If $t \mapsto (A_t,u_t)$ is a solution of the system
\eqref{eq:evolve}, then the time-dependent $P(\mathfrak{k})$-valued $0$-form $F_{A_t,u_t}:=*F_{A_t} + \Phi(u_t)$ satisfies
\begin{align*}
  \ddt F_{A_t,u_t}&=*{d}_{A_t}\left(\ddt A_t\right)+u_t^*{d}\Phi\left(\ddt u_t\right)\\
  &=-{d}_{A_t}^*{d}_{A_t}F_{A_t,u_t}+u_t^*{d}\Phi(-J(F_{A_t,u_t})_{u_t}),
\end{align*}
where $d_{A_t}^*:=-*d_{A_t}*$ is the formal adjoint of ${d}_{A_t}$. For a connection $A$, the Hodge
Laplacian is defined as $\Delta_A=d_A^*d_A+d_Ad_A^*$. For
$0$-forms, $d_A^*d_A=\Delta_A$, which is an elliptic
operator. Writing $F_t:=F_{A_t,u_t}$, the above equation is
equivalent to
\begin{equation}\label{eq:gradF}
  \frac{dF_t}{dt}=-\Delta_{A_t}F_t - u_t^*{d}\Phi(J(F_t)_{u_t}).
\end{equation}
Except for the non-linear term $u_t^*{d}\Phi(J(F_t)_{u_t})$, \eqref{eq:gradF} is
parabolic. Roughly speaking, once we solve this equation in $F$, $A_t$
is given by $A_0-\int_0^t*d_AF_{A,u}$ and $u_t$ is obtained by
integrating the vector field $(F_{A_t,u_t})_{u_t}$. But unfortunately,
$A_t$ occurs in the term $\int_0^t*d_{A_t}F_{A_t,u_t}$ and $A_t$, $u_t$ occur
in the equation in $F_t$. Hence, we need to solve the three equations
(\eqref{eq:evolve} and \eqref{eq:gradF}) as a coupled system.

We know, if $(A_t,u_t)$ is a solution of \eqref{eq:evolve}, then
$(A_t,F_{A_t,u_t},u_t)$ is a solution of
\begin{equation}\label{eq:evolve2}
  \ddt A_t = *{d}_AF, \quad \ddt F_t = - {d}_{A_t}^* {d}_{A_t} F_t -  u_t^* {d}\Phi J(F_t)_{u_t}, \quad \ddt u_t = J(F_t)_{u_t}.
\end{equation}
with initial data at $t=0$ given by $(A_0,F_{A_0,u_0},u_0)$. Note
that $F_t \in \Gamma(\Sig,P(\mathfrak{k}))$ is an independent variable in this
system, whereas $F_{A_t}$ denotes the curvature of the connection $A_t$ and
$F_{A_t,u_t}=*F_{A_t}+u_t^*\Phi \in \Gamma(\Sig,P(\mathfrak{k}))$.
\begin{remark}
  A solution $(A_t,F_t,u_t)$ of \eqref{eq:evolve2} whose initial value
  $(A_0,F_0,u_0)$ satisfies $F_0=F_{A_0,u_0}$ will satisfy
  $F_t=F_{A_t,u_t}$ for all time $t$.
\end{remark}

We use $A_0$ as the base connection, and write the connection $A_t$
on $P$ as $A_0+a_t$, where $a_t \in \Omega^1(\Sig,P(\mathfrak{k}))$. In an analogous way, write $u_t =
\exp_{u_0} \xi_t$, where $\xi_t \in \Gamma(\Sig,
u_0^*T^{\on{vert}}P(X))$. Then, the system \eqref{eq:evolve2} becomes
\begin{align}\label{eq:evolve3}
    \ddt a_t - *{d}_{A_0} F_t &= *[a_t,F_t]\\
 \nonumber   \ddt F_t + \Delta_{A_0} F_t &=-u_t^*  {d}\Phi J(F_t)_{u_t} - *[a_t\wedge *{d}_{A_0} F_t] - [{d}_{A_0}^*a_t,F_t] \\
\nonumber &\quad  - *[a_t\wedge *[a_t,F_t]]\\
 \nonumber   \ddt \xi_t &= -({d}\exp_{u_0})(\xi_t)^{-1}(J(F_t)_{u_t})
\end{align}
with initial conditions $a_0=0$, $F_0=*F_{A_0,u_0}$, $\xi_0=0$. The
advantage of writing the system this way is that now, $a$, $F$ and
$\xi$ are just time-dependent sections of vector bundles over $\Sig$. We next explain
the expression $({d}\exp)(\xi)^{-1}$ in the last equation.
\begin{remark}\label{rem:injrad} (\rm{Injectivity radius})
  The exponential map on $X$ defines for every $s \in \Sig$, a map
  $\exp_{u_0(s)}:T_{u_0(s)}X \to X$, and its derivative
  ${d}\exp_{u_0(s)}(\xi):T_{u_0(s)}X \to T_{\exp_{u_0}\xi(s)}X$ for
  every tangent vector $\xi \in T_{u_0(s)}X$. The inverse of the
  derivative ${d}\exp_{u_0(s)}(\xi)^{-1}$ is well-defined if the
  derivative is injective. For that to be the case, we ensure
$$\Mod{\xi}_{C^0} < \on{inj}_X,$$
where $\on{inj}_X$ is the injectivity radius of $X$, which we now
define. The injectivity radius at a point $x \in X$, denoted by
$\on{inj}_X(x)$, is the radius of the largest ball around the origin
in $T_xX$ on which the exponential map is a diffeomorphism. Taking
infimum over all of $X$, we obtain $\on{inj}_X:=\inf_{x \in X}
\on{inj}_X(x)$. For a compact manifold $\on{inj}_X >0$.
\end{remark}
\subsubsection{Description of Sobolev spaces}
To show the existence of a solution, we work in Sobolev spaces of
sections of vector bundles. The vector bundles are of the type
$E:=\wedge^k(T^*\Sig) \tensor P(\mathfrak{k})$. A covariant derivative on $E$ is
determined by a choice of connection $A$ on $P$ and the Levi-Civita
connection on $\Sig$, and we denote it by $\nabla_A$. For a
non-negative integer $s$, we denote by $H^s(\Gamma(\Sig,E))$ or
$H^s(\Sig,E)$ the completion of the space of sections $\Gamma(\Sig,E)$
under the Sobolev norm
\begin{equation}
  \label{eq:sob1}
\Mod{\sig}_{H^s(\Sig,E)}^A:=\left(\sum_{i=0}^s\int_\Sig|\nabla^i_A \sig|^2\right)^{1/2}.  
\end{equation}
To solve the evolution equations, we introduce mixed Sobolev
completions of time-dependent sections.  For any real $r$, $s$, $T>0$,
$H^{r,s}([0,T] \times \Sig,E)$, also referred to as $H^{r,s}$ or
$H^r(H^s)$, is the space of (equivalence classes of) time-dependent
sections that are in Sobolev class $H^r$ in time and $H^s$ in
space. When $r$ and $s$ are non-negative integers, $H^{r,s}$ is the
completion of $C^\infty([0,T] \times \Sig,E)$ under the norm
$$\Mod{\sig}_{r,s}^A:=\left (\sum_{i=0}^r \sum_{j=0}^s\Mod{T^{-(r-i)}\frac{d^i}{dt^i}\nabla_A^j\sig}^2_{L^2(\Sig \times [0,T]}\right )^{1/2}.$$
For other exponents, the spaces $H^{r,s}$ are defined by interpolation
and duality. For negative Sobolev exponents, the elements of
$H^{r,s}$, need not be almost-everywhere defined sections, they are
just distributions. The norm $\Mod{\cdot}_{r,s}$ depends on the base
connection $A$ but is equivalent for any choice of connection, so that
the space $H^{r,s}$ is well-defined independent of the connection. The
base connection need not be smooth. A connection $A$ is said to be
$H^1$ if for any smooth connection $A'$, the difference $A-A'$ is in
$\Om^1(\Sig,P(\mathfrak{k}))_{H^1}$.  If the base connection $A$ is a $
H^1$-connection, then the spaces $H^{r,s}$ can be defined for $s \in
[-2,2]$. Detailed definitions and properties of these spaces are given
in Section \ref{sec:sobolevspaces}. A crucial property is that
although the operator norms depend on the choice of base connection,
if the curvature satisfies $\Mod{F(A)}_{L^2}<\kappa$, the operator
norms are bounded by constants dependent only on $\kappa$ and
independent of $A$. To prove the existence of the flow line starting
at $(A_0,u_0)$, we fix $A_0$ as the base connection for the Sobolev
norms.

Another type of Sobolev space we use is $H^r([0,T],C^0(\Sig,E))$ - it
is the space of (equivalence classes of) sections that are in
Sobolev-class $r$ in time and are $C^0$ in space. This space has norm
$$\Mod{\sig}_{r,C^0}:=\sup_{z \in \Sig}\Mod{\sig_z}_{H^r([0,T],E_z)}.$$
The way this is defined, it is more appropriate to call it
$C^0(\Sig,H^r([0,T],E))$, but we call it $H^r(C^0)$ to preserve our
convention of having the time-index outside. This space satisfies the
expected embedding properties, for example $H^{r,s} \hra H^r(C^0)$ for
$s>1$, but that is not obvious because the spaces
$H^{r,s}=H^r([0,T],H^s(E))$ are defined with the time and space
co-ordinates in a different order. These details will be presented in
Section \ref{subsec:interchange}. This space is used, for example,
when the bundle $E$ is $u_0^*T^{\on{vert}}P(X)$, where it is useful to
have operator norms be independent of the derivatives of the map
$u_0$.

We next define the Banach space in which we solve the system of
equations \eqref{eq:evolve3}.  With initial value $F(0) \in
L^2(\Sig,P(\mathfrak{k}))$, we expect to solve for $F$ in spaces of the type
$H^{\hh + r,-2r}([0,T] \times \Sig, P(\mathfrak{k}))$ (see Lemma
\ref{lem:heateqnbaseg}). We fix a small constant $\eps \in (0,1/16)$. The system \eqref{eq:evolve3} is solved in the Banach space
\begin{align*}
  U(T)=\{(a,F,\xi)&|a \in H^{1/2+\eps} (H^{1-2\eps}), \\
  &F \in H^{1/2+\eps}(H^{-2\eps})\cap
  H^{-1/2+\eps}(H_\partial^{2-2\eps}),\xi \in H^{1/2+\eps}(C^0)\}.
\end{align*}
For $s>1$, $H^s_\partial(\Sig,E)$ is defined as the subspace of
$H^s(\Sig,E)$ consisting of sections that vanish on the boundary of
$\Sig$. From this point onward, for brevity of notation, we drop the subscript $t$ from the time-dependent sections $a$, $F$ and $\xi$.
We will prove:
\begin{proposition} \label{prop:flowexist} Let $(A_0,u_0) \in
  \A(P)_{H^1} \times \Gamma(\Sig,P(X))_{C^0}$ be a gauged map. Then
  for any $\kappa> 0$ there exists $t_0(\kappa) > 0$ such that if
  $\Mod{F_{A_0}}_{L^2} < \kappa$ then the initial value problem
  \eqref{eq:evolve3} has a unique solution $(a,F,\xi) \in
  C^0([0,t_0],H^1 \times L^2\times C^0)$.
\end{proposition}
With this Proposition, we can prove the existence of a unique solution
for the flow equation for all time.
\begin{corollary}\label{cor:heat_nonreg} Suppose the target manifold
  $X$ is compact. Let $(A_0,u_0) \in \A(P)_{H^1} \times
  \Gamma(\Sig,P(X))_{C^0}$ be a gauged map. Then, the initial value
  problem \eqref{eq:evolve} has a unique solution for all time
  $(A_t,u_t) \in C^0_\loc([0,\infty),\A(P)_{H^1} \times
  \Gamma(\Sig,P(X))_{C^0})$.
\end{corollary}
\begin{proof} By compactness of $\Sig$, for any gauged map $(A,u)$,
  there is a uniform bound on the moment map term:
  $\Mod{u^*\Phi}_{L^2} \leq \Mod{\Phi}_{C^0}\on{Vol}(\Sig) \leq
  c$. Therefore, the norm of the curvature differs from the vortex
  functional by a constant at most: $\Mod{F_A}_{L^2} \leq
  \Mod{F_{A,u}}_{L^2} + c$. Applying Proposition \ref{prop:flowexist}
  with $\kappa=\Mod{F_{A_0,u_0}}_{L^2}+c$, we get the flow for a time
  interval $[0,t_0]$, with $(A_{t_0},u_{t_0}) \in H^1 \times C^0$. The
  vortex functional $\Mod{F_{A_t,u_t}}_{L^2}^2$ decreases along the
  flow line $(A_t,u_t)$, and so, $\Mod{F_{A_{t_0},u_{t_0}}}_{L^2} <
  \Mod{F_{A_0,u_0}}_{L^2}$. Since $\Mod{F_{A_{t_0}}}_{L^2} < \kappa$,
  we can get flow for the time interval $[t_0,2t_0]$ starting from the
  pair $(A_{t_0},u_{t_0})$. The process is repeated to get the flow
  line for all time $t \in [0,\infty)$.
\end{proof}

To prove Proposition \ref{prop:flowexist}, we define certain Banach
spaces needed to state intermediate results. The first one $U_P(T)$ is
a subspace of $U(T)$ consisting of sections that vanish at $t=0$.
\begin{align*}
  U_P(T)=\{ &(a,F,\xi)| a \in H_P^{1/2+\eps}(H^{1-2\eps}), \\
  &  F \in H_P^{1/2+\eps}(H^{-2\eps})\cap H^{-1/2+\eps}(H_\partial^{2-2\eps}), \xi \in H_P^{1/2+\eps}(C^0) \}\\
  W(T)=\{&(a,F,\xi)| a \in H^{-1/2+\eps}(H^{1-2\eps}), F \in H^{-1/2+\eps}(H^{-2\eps}),\\ 
  &  \xi \in H^{-1/2+\eps}(C^0) \}\\
  \X= \{& (a_0,F_0,\xi_0)|a_0 \in H^1, F_0 \in H^0, \xi_0 \in C^0 \}
\end{align*}
\begin{notation} We call $x:=(a,F,\xi)$ and $x_i:=(a_i,F_i,\xi_i).$
\end{notation}

\subsubsection{Outline of proof of Proposition \ref{prop:flowexist}}
The terms in the system \eqref{eq:evolve3} can be broken into 2 parts
- the leading order terms and the rest. The leading order terms form
an operator
\begin{align*}
  L: U(T) & \to W(T)\\
  (a,F,\xi) & \mapsto \left(\ddt a - *{d}_{A_0}F, \left(\ddt + \Delta_{A_0}\right)F,
  \ddt\xi\right).
\end{align*}

When restricted to $U_P(T)$, this operator is invertible (see Lemma
\ref{lem:Lbd}).
The terms in the right hand side of \eqref{eq:evolve3} form a
non-linear operator $Q:U(T) \to W(T)$. We break up the solution into 2
parts $x=x_1+x_2$, the first is an approximate solution and the second
is a correction.  The approximate solution $x_1$ is in $U(T)$ and
satisfies $Lx_1=0$, $x_1(0)=x_0$ and can be found uniquely (see Lemma
\ref{lem:Mbd}). The correction $x_2$ is in $U_P(T)$ and satisfies
$$Lx_2 = Q(x_1+x_2).$$
The existence of a unique value of $x_2$ is proved for small $T$ using
implicit function theorem.

Let $M$ denote the operator whose input is the initial value
$x_0=(a_0,F_0,\allowbreak \xi_0)$ and output is the approximate solution
$x_1=(a_1,F_1,\xi_1)$. That is,
\begin{align*}
  M: \X \to U(T), \quad (a_0,F_0,\xi_0)\mapsto (a_1,F_1,\xi_1)
\end{align*}
where $x_1(0)=x_0$ and $L(a_1,F_1,\xi_1)=0.$

The terms in $Q$ are split into $Q_1$, $Q_2$, $Q_3$ in a way that they
have a linear, quadratic and cubic bound on them respectively. (See
Lemma \ref{lem:Qbd}.)
\begin{align*}
  &Q:U(T) \to W(T) \quad Q=Q_1+Q_2+Q_3\\
  &Q_1 : (a,F,\xi) \mapsto (0,-u_0^*{d}\Phi(JF_{u_0}),JF_{u_0})\\
  &Q_2 :(a,F,\xi) \mapsto (*[a,F],-*[a\wedge *{d}_{A_0}F]\\
  &\hspace{7em} -[{d}_{A_0}^*a,F]-((\exp_{u_0}\xi)^*{d}\Phi (JF_{\exp_{u_0} \xi})-u_0^*{d}\Phi (JF_{u_0})),\\
  &\hspace{7em} -(({d}\exp_{u_0}\xi)^{-1}(JF_{\exp_{u_0}\xi}) - JF_{u_0})\\
 & Q_3 : (a,F,\xi) \mapsto (0,-*[a\wedge *[a,F]],0)
\end{align*}

\subsubsection{Bounds on $L$, $M$ and $Q$}
The next 3 Lemmas prove that $L$, $M$, $Q$ are well-defined operators
and that they satisfy certain bounds, given $\Mod{F(A_0)}_{L^2}
\leq \kappa$. The constants in these bounds, denoted by $c_\kappa$ are
independent of $(A_0,u_0)$, $T$ and depend only on $\kappa$.
\begin{lemma}\label{lem:Lbd} The operator $L:U_P(T) \to W(T)$ is
  invertible. For any $\kappa$, there exists a constant $c_\kappa$
  such that if $\Mod{F(A_0)}_{L^2} \leq \kappa$ then $\Mod{L^{-1}}
  \leq c_\kappa.$
\end{lemma}
\begin{proof}
  In matrix form,
  \[ L= \begin{pmatrix} \ddt &-*{d}_{A_0} &0 \\
    0 &\ddt+\Delta_{A_0} &0\\
    0 &0 &\ddt
  \end{pmatrix}. \] The operators
  \begin{equation*}
    \ddt : H_P^{1/2+\eps,1-2\eps} \to H_P^{-1/2+\eps, 1-2\eps}, \quad 
    \ddt : H_P^{1/2+\eps}(C^0) \to H_P^{-1/2+\eps}(C^0)
  \end{equation*}
  have as their inverse the integration operator $\int_0$ defined on
  the respective spaces, which is bounded by $c_\kappa$ using Lemma
  \ref{lem:integrationtime}. The operator $\ddt + \Delta_{A_0}:H^{1/2+\eps, -2\eps}_{P,}\cap
  H^{-1/2+\eps, 2-2\eps}_{,\partial}\to H^{-1/2+\eps, -2\eps}$ has an
  inverse with norm $\leq c_\kappa$ by Lemma \ref{lem:heateqnbasef}.

  Last, we look at $*{d}_{A_0}$. The operator
  $\nabla_{A_0}:H^{2-2\eps}\to H^{1-2\eps}$ has norm bounded by
  $c_\kappa$ for all $t \in [0,T]$ (using \eqref{eq:Da}). This induces
  $$\nabla_{A_0}:H^{-1/2+\eps, 2-2\eps}\to H^{-1/2+\eps,1-2\eps}$$ with
  the same bound on the norm (see \eqref{eq:bundlemap}). On 0-forms,
  $\nabla_{A_0}={d}_{A_0}$ and so $\Mod{*{d}_{A_0}} \leq c_\kappa$.
\end{proof}
\begin{lemma}\label{lem:Mbd} The operator $M$ is well-defined. For any
  $\kappa>0$, there exists a constant $c_\kappa$ such that if
  $\Mod{F(A_0)}_{L^2}<\kappa$, $\Mod{M} \leq c_\kappa T^{-\eps}.$
\end{lemma} 
\begin{proof}
  By Lemma \ref{lem:heateqnbaseg}, given $F_0 \in L^2$, the system
  \begin{equation*}
    \ddt F_1 + {d}_{A_0}^*{d}_{A_0}F_1=0, \quad F_1(0)=F_0
  \end{equation*}  
  has a unique solution $F_1 \in H^{1/2+\eps, -2\eps}\cap
  H_{,\partial}^{-1/2+\eps, 2-2\eps}$ satisfying
$$\Mod{F_1}_{H^{1/2+\eps,-2\eps}\cap H^{-1/2+\eps,2-2\eps}} \leq c_\kappa \Mod{F_0}_{H^0}.$$
Define $a_1(t):=a_0 + \int_0^t *{d}_{A_0} F_1$. Then,
$$\bigMod{\int_0^t* {d}_{A_0}F_1}_{\hh+\eps,1-2\eps} \leq c_\kappa\Mod{F_1}_{-\hh+\eps,2-2\eps}$$
because $\Mod{{d}_{A_0}} \leq 2\Mod{\nabla_{A_0}} \leq c_\kappa$ by
\eqref{eq:Da} and $\int_0^t$ has norm $\leq c$ by Lemma
\ref{lem:integrationtime}. So,
$$\Mod{a_1}_{\hh+\eps,1-2\eps} \leq c_\kappa(\Mod{a_0}_{H^1}+\Mod{F_0}_{H^0}).$$
Finally, since $\frac{d \xi_1}{dt}=0$, we set $\xi_1(t)=\xi_0$, and
$\Mod{\xi_1}_{\hh+\eps,C^0} \leq c\Mod{\xi_0}_{C^0}$ for some constant
$c$.
\end{proof}
\begin{lemma}\label{lem:Qbd}
  Let $x=(a,F,\xi)$. Assume $\Mod \xi_{C^0} \leq \on{inj}_X$ (see
  Remark \ref{rem:injrad}). Then, $Q:U(T) \to W(T)$ is a well-defined
  map. It is differentiable so that ${d} Q(x):U(T) \to W(T)$ is a
  linear map for each $x \in U(T)$. If $\Mod{F(A_0)}_{L^2}<\kappa$,
  there exist constants $c_\kappa$ so that
  \begin{align*}
    \Mod{Q_1x}_W &\leq c_\kappa T^{\hh-2\eps}\Mod{x}_U, \qquad
    \Mod{Q_2(x)}_W \leq c_\kappa T^{\hh-2\eps}\Mod{x}^2_U, \\
    \Mod{Q_3(x)}_W &\leq c_\kappa T^{\hh-2\eps}\Mod{x}^3_U.
  \end{align*}
  The derivatives satisfy
  \begin{align*}
    \Mod{{d} Q_1(x)} &\leq c_\kappa T^{\hh-2\eps},\qquad\quad  \Mod{dQ_2(x)}
    \leq c_\kappa T^{\hh-2\eps}(1+\Mod{x}_U),\\
    \Mod{{d} Q_3(x)} &\leq c_\kappa T^{\hh-2\eps}\Mod{x}^2_U.
  \end{align*}
  Putting them together,
  \begin{align*}
    \Mod{Q(x)}_W \leq c_\kappa T^{\hh-2\eps}(1+\Mod x_U^3), && \Mod{d Q(x)} \leq c_\kappa T^{\hh-2\eps}(1+\Mod{x}_U^2).
  \end{align*}
\end{lemma}
\begin{proof}
  The terms $[a,F]$, $[{d}_{A_0}a,F]$, $[a,{d}_{A_0}F]$ and $[a,[a,F]]$
  are polynomials of $a$, $F$ and their derivatives. Consider $[a,F]$
  The term $a$ is in $H^{\hh+\eps,1-2\eps}$ and by interpolation, $F$
  is in $H^{0,1}$ (see Corollary \ref{cor:interpolmixed}). By the multiplication theorem
  \eqref{eq:multiplicationmixed}, $[a,F]$ is in $H^{-\eps,-2\eps}$
  which embeds into $H^{-\hh+\eps,-2\eps}$. The embedding has norm
  $c_\kappa T^{\hh-2\eps}$ \eqref{eq:includetime}. We have
  \begin{align*}
    \Mod{[a,F]}_{-\hh+\eps,-2\eps} &\leq c_\kappa T^{\hh-2\eps}\Mod{[a,F]}_{-\eps,-2\eps} \leq c_\kappa T^{\hh-2\eps}\Mod a_{\hh+\eps,1-2\eps} \Mod F_{0,1} \\
    &\leq c_\kappa T^{\hh-2\eps}\Mod x_U^2.
  \end{align*}
  That the constants depend only on $\kappa$ follows from Proposition
  \ref{prop:curvbound}. The other polynomial terms are bounded the
  same way. A bound on the derivatives for these terms is obvious :
  for example,
$$\Mod{d[a,F]}_{-\hh+\eps,-2\eps} \leq c_\kappa T^{\hh-2\eps}(\Mod{a}_{\hh+\eps,1-2\eps} + \Mod{F}_{0,1}) \leq c_\kappa T^{\hh-2\eps}\Mod x_U.$$

To discuss the other terms, which are not polynomial, we define an
operator: for any map $u \in C^0(\Sig,P(X))$, let
$X_u:\Gamma(\Sig,P(\mathfrak{k})) \to \Gamma(\Sig, u^*T^{\on{vert}}P(X)) $ be
given by $\xi \mapsto J\xi_u$. The terms $u_0^*{d}\Phi (JF_{u_0})$ and
$JF_{u_0}$ are obtained by the action of linear bundle maps
${d}\Phi_{u_0} \circ X_{u_0}$ and $X_{u_0}$ respectively on $F$. For
example, the first of these terms $u_0^*{d}\Phi F_{u_0}$ can be seen as
the tensor product of the sections ${d}\Phi_{u_0} \circ X_{u_0}$ and
$F$.  The first factor ${d}\Phi \circ X_{u_0}$ is in $L^2(\Sig,P(\End
\mathfrak{k}))$ and the norm is independent of $u_0$. Further, since it is time
independent ${d}\Phi \circ X_{u_0} \in H^{1,0}$.  As earlier $F \in
H^{0,1}$.  By the multiplication theorem
\eqref{eq:multiplicationmixed}, $u_0^*{d}\Phi F_{u_0} \in
H^{-\eps,-2\eps} \hra H^{-\hh +\eps,-2\eps}$. The operator norm picks
up a factor of $c_\kappa T^{\hh-2\eps}$ from the last inclusion.

The remaining two terms
$$((\exp_{u_0}\xi)^*{d}\Phi (JF_{\exp_{u_0}
  \xi})-u_0^*{d}\Phi (JF_{u_0})),(({d}\exp_{u_0}\xi)^{-1}
 (JF_{\exp_{u_0}\xi}) - JF_{u_0})$$ 
require Corollary
\ref{cor:comptbd}, which is a result on composition of functions in
the space $H^{\hh+\eps}(C^0)$. Section \ref{sec:compfunc} explains
this result in detail. Consider the first of these terms. The bundle
map $\xi \mapsto ({d}\Phi \circ X_{\exp_{u_0}\xi} - {d}\Phi \circ
X_{u_0})$ is continuous and by Corollary \ref{cor:comptbd}, it induces
a map 
$$H^{\hh+\eps}(C^0)(\Sig,P(\mathfrak{k})) \to H^{\hh+\eps}(C^0)(\Sig,P(\End
\mathfrak{k})).$$
So $({d}\Phi \circ X_{\exp_{u_0}\xi} - {d}\Phi \circ X_{u_0})\in
H^{\hh+\eps}(C^0)(P(\End \mathfrak{k}))$ and
$$\Mod{{d}\Phi \circ X_{\exp_{u_0}\xi} - {d}\Phi  \circ X_{u_0}}_{\hh+\eps,C^0} \leq c_\kappa\Mod \xi_{\hh+\eps,C^0},$$
where $c_\kappa$ is independent of $u_0$. By compactness of $\Sig$,
${d}\Phi \circ X_{\exp_{u_0}\xi} - {d}\Phi \circ X_{u_0}$ is in
$H^{\hh+\eps,0}$. Multiplying by $F \in H^{0,1}$, we get the result
$$((\exp_{u_0}\xi)^*{d}\Phi (JF_{\exp_{u_0} \xi})-u_0^*{d}\Phi
(JF_{u_0})) \in H^{-\eps,-2\eps} \hra H^{-\hh+\eps,-2\eps}$$ and
\begin{align*}
  &\quad\ \Mod{((\exp_{u_0}\xi)^*{d}\Phi (JF_{\exp_{u_0} \xi})-u_0^*{d}\Phi (JF_{u_0}))}_{-\hh+\eps,-2\eps} \\
  &\leq c_\kappa T^{\hh-2\eps}\Mod{((\exp_{u_0}\xi)^*{d}\Phi (JF_{\exp_{u_0} \xi})-u_0^*{d}\Phi (JF_{u_0}))}_{-\eps,-2\eps}\\
  &\leq  c_\kappa T^{\hh-2\eps}\Mod{{d}\Phi \circ X_{\exp_{u_0}\xi} - {d}\Phi  \circ X_{u_0}}_{\hh+\eps,C^0}\Mod{F}_{0,1}\\
  &\leq c_\kappa T^{\hh-2\eps}\Mod{\xi}_{\hh+\eps,C^0}\Mod{F}_{0,1}
  \leq c_\kappa T^{\hh-2\eps}\Mod{x}^2_U.
\end{align*}
Similarly, for the second term
$(({d}\exp_{u_0}\xi)^{-1}(JF_{\exp_{u_0}\xi}) - JF_{u_0})$,
$$\Mod{({d}\exp_{u_0} \xi)^{-1}\circ X_{\exp_{u_0}\xi}  - X_{u_0}}_{\hh+\eps,C^0} \leq c_\kappa \Mod \xi_{\hh+\eps,C^0}.$$
Applying interpolation (Corollary \ref{cor:interpolmixed}), followed by Sobolev embedding
\eqref{eq:embedspace}, $F \in H^{-\eps,1+2\eps} \hra
H^{-\eps}(C^0)$. By multiplication Theorem
\eqref{eq:c0mult}, $$(({d}\exp(\xi))^{-1}\circ X_{\exp_{u_0}\xi} -
X_{u_0})F \in H^{-\eps}(C^0) \hra H^{-\hh+\eps}(C^0).$$ Corollary
\ref{cor:comptbd} also gives differentiability and a bound on the
derivative for these terms.
\end{proof}

\subsubsection{Existence of flow in a uniform time interval $[0,t_0(\kappa)]$}
We now prove Proposition \ref{prop:flowexist}, which is the main
result of Section \ref{subsec:flowexist}.

\begin{proof}[Proof of Proposition \ref{prop:flowexist}] 
  The proof of the existence statement in Proposition
  \ref{prop:flowexist} is by an application of implicit function
  theorem (Proposition \ref{prop:impfnMS}) and is similar to R{\aa}de's
  proof in \cite{Rade}.  To show the existence of gradient flow in a
  time interval $[0,T]$, we need to solve
  $$L(a_2,F_2,\xi_2) = Q(M(a_0,F_0,\xi_0) + (a_2, F_2, \xi_2))$$ 
  for $(a_2,F_2,\xi_2) \in U_P(T)$. As earlier, we use the notation
  $x_i=(a_i,F_i,\xi_i)$ for $i=0, 1, 2$. Recall that
  $x_0=(a_0,F_0,\xi_0)$ is the initial data for the system
  \eqref{eq:evolve3} and $x_1:=Mx_0$ is an approximate solution of
  \eqref{eq:evolve3}.  In order for the map $Q$ to be well-defined, we
  need to ensure that the $L^\infty$ norm of $\xi=\xi_1+ \xi_2$ is
  less than the injectivity radius of $X$.  The construction of $M$
  gives $\xi_1(t)\equiv \xi_0=0$. We next define the various
  quantities in the hypothesis of Proposition \ref{prop:impfnMS}. Let
  $Y_1:=U_P(T)$, $Y_2:=W(T)$, $\S(T):=\{(a_2,F_2,\xi_2) \in U_P(T) :
  \Mod{\xi_2}_{C^0} < \inj_X\}$, $\F(x) := -Lx + (Q_1+Q_2+Q_3)(Mx_0+x)$
  and $C:=\linebreak \Mod{L^{-1}} = c_\kappa$. By the definition of $\S(T)$, $Q$ is
  well-defined on $\S(T)$. The set $\S(T)$ is indeed an open
  neighborhood of the origin in $U_P(T)$ because the map $\pi : U_P(T)
  \to \Gamma(\Sig, P(\mathfrak{k}))_{C^0}$ that takes $(a,F,\xi)$ to $\xi$ is
  continuous (by Sobolev embedding \eqref{eq:embedtime}). Choose a
  constant $\delta>0$ such that $B_\delta \subset \S(T)$. For $x_2 \in
  B_\delta$,
  \begin{align*}
\Mod{{d} \F(x_2)-{d} \F(0)}&=\Mod{{d} Q(Mx_0+x_2) - {d} Q(Mx_0)}\\
& \leq c_\kappa t_0^{\hh-2\eps}(1+ \Mod{Mx_0}^2+\Mod{Mx_0+x_2}^2)\\
&\leq c_\kappa t_0^{\hh-4\eps}(1+\Mod{x_0}^2).    
  \end{align*}
Proposition \ref{prop:impfnMS} can be applied if
\begin{itemize}
\item $\Mod{{d} Q(Mx_0\!+\!x_2) \!-\! {d} Q(Mx_0)}\!<\! 1/2C$, i.e. $T^{\hh-4\eps}(\Mod{x_0}i\!+\!1)^2
  \leq 1/c_\kappa$ and
\item $\Mod{F(0)} = \Mod{Q(Mx_0)}=c_\kappa T^{\hh
    -3\eps}(\Mod{x_0}+1) < \delta/4c_\kappa$.
\end{itemize}
Both these conditions can be met by a small enough value of $T$, that
is dependent only on $\kappa$. We call this value $t_0(\kappa)$ and it
proves the existence part of Proposition \ref{prop:flowexist}.

Next, we prove that there exists a solution with extra regularity $a
\in C^0(H^1)$, $F \in C^0(L^2)\cap L^2(H^1)$ and $\xi \in C^0$ as
required by the statement of Proposition \ref{prop:flowexist}.  First,
we look at $(a_1,F_1,\xi_1)$. By Remark \ref{rem:C0heat}, $F_1
\in C^0(L^2) \cap H^{0,1}$. Since $F_1$ satisfies $(\ddt +
\nabla_{A_0}^* \nabla_{A_0})F_1=0$, we get $F_1(t)-F_0=\nabla_{A_0}^*
\nabla_{A_0}\int_0F_1 \in C^0(L^2)$. By elliptic regularity (see
Proposition \ref{prop:laplaceinv}), $\int_0 F_1 \in C^0(H^2)$. So,
$a_1(t)=a_0+\int_0 {d}_{A_0}^* F_1 \in C^0(H^1)$. It can be checked
that Lemmas \ref{lem:Lbd}, \ref{lem:Mbd} and \ref{lem:Qbd} hold with
the following stronger spaces
\begin{align*}
  \tilde{U}(t_0)=\{& (a,F,\xi)|a \in H^{\hh+\eps, 1-2\eps}\cap H^{\hh,1}, F \in H^{\hh+\eps, -2\eps}\cap H^{-\hh, 2},\\
&  \xi \in H^{\hh+\eps}(C^0)\}\\
  \tilde U_P(t_0) =\{& (a,F,\xi)| a \in H_P^{\hh+\eps, 1-2\eps}\cap H_P^{\hh,1}, F \in H_P^{\hh+\eps, -2\eps}\cap H_P^{-\hh, 2},\\
&  \xi \in H_P^{\hh+\eps}(C^0) \}\\
  \tilde W(t_0) =\{& (a,F,\xi)| a \in H^{-\hh+\eps, 1-2\eps}\cap H_P^{-\hh,1}, F \in H^{-\hh+\eps, -2\eps}\cap H_P^{-\hh,0}, \\
  &  \xi \in H^{-\hh+\eps}(C^0) \}.
\end{align*}
So, there exists a solution of \eqref{eq:evolve3} in $\tilde
U(t_0)$. Using this, we can get improved estimates for the right hand
side of \eqref{eq:evolve3}. For example, $a \in H^{\hh,1} \implies
\nabla_{A_0}a \in H^{\hh,0}$. By interpolation $F \in H^{\frac{1}{4}
  -\eps, \frac{3}{2} +2\eps}$. By the multiplication theorem,
$[\nabla_{A_0}a,F] \in H^{-\frac{1}{4}-2\eps,0}$. Similarly we
estimate all terms in the right hand side of \eqref{eq:evolve3} to get
\begin{equation*}
  \left\{ \begin{aligned}
      \ddt a_2 + {d}_{A_0}^* F_2 \in H_P^{-\frac{1}{4} -2\eps,1}\\
      \ddt F_2 + \nabla_{A_0}^*\nabla_{A_0} F_2 \in H_P^{-\frac{1}{4} -2\eps,0}
    \end{aligned} \right.
\end{equation*}
By parabolic regularity (Lemma \ref{lem:heateqnbasef}), we get $F_2 \in H_P^{\frac{3}{4}
  -2\eps,0} \cap H_P^{-\frac{1}{4} -2\eps,2} \hra C^0(L^2)$. So,
$F=F_1+F_2 \in C^0(L^2)$. Also, since ${d}_{A_0}^*F_2 \in
H_P^{-\frac{1}{4} -2\eps,1}$, $a_2 \in\linebreak  H_P^{\frac{3}{4} -2\eps,1} \hra
C^0(H^1)$. Therefore, $a=a_1+a_2 \in C^0(H^1)$.

Finally, we prove there is a unique solution of \eqref{eq:evolve3} in
the time interval $[0,t_0]$ by contradiction. Suppose $x=(a,F,\xi)$
and $x'=(a',F',\xi')$ are two solutions to \eqref{eq:evolve3} for some
$t_0 > 0$ with the same initial data $(a_0,F_0,\xi_0)$, $a_0 \in L^2$,
$F_0 \in H^1$, $\xi_0 \in C^0$.
\begin{gather*}
a\!\in\! H^{\hh+\eps,1-2\eps} \!\cap\! C^0(H^1),\ F \!\in\! H^{\hh+\eps,-2\eps} \!\cap\! H^{-\hh+\eps,2-2\eps}\!\cap\! C^0(L^2),\ \xi \!\in\! H^{\hh\!+\!\eps}(C^0)\\
a'\in H^{\hh+\eps,1-2\eps},\ F' \in H^{\hh+\eps,-2\eps} \cap H^{-\hh+\eps,2-2\eps},\ \xi' \in H^{\hh+\eps}(C^0)
\end{gather*}
Assume $x \neq x'$. Let $t_1$ be the largest number such that the
restrictions of $x$ and $x'$ to $\Sig \times [0,t_1]$ are
identical. Since the solutions are in $C^0([0,t_1], H^1 \times L^2
\times C^0)$, $(a(t_1),F(t_1),\xi(t_1))$ is well-defined. Then,
$(a,F,\xi)$ and $(a',F',\xi')$ solve the initial value problem
\eqref{eq:evolve3} on $\Sig \times [t_1,t_0]$ with initial data
$(a(t_1),\linebreak F(t_1),\xi(t_1))$. Therefore, without loss of generality, we
may assume that $t_1=0$. We can split $x=x_1+x_2$, where $Lx_1=0$,
$x_1(0)=(a_0,F_0,\xi_0)$ and $x_2 \in U_P(t_0)$. Similarly
$x'=x_1'+x_2'$. Since $M$ is uniquely defined $x_1=x_1'$. So, now both
$x_2$ and $x_2'$ are solutions of $Lx=Q(x_1+x)$ in $U_P(t_0)$. By
Sobolev embedding, both $\xi_2, \xi_2' \in H_P^{\hh+\eps}(C^0) \hra
C^0_P([0,t_0],C^0)$. There exists $0<t < t_0$ such that
$\Mod{\xi_2'}_{C^0_P([0,t],C^0)} < \on{inj}_X$. So, the restrictions
of $x_2$ and $x_2'$ to $U_P(t)$ are in $\S(t)$. The set $\S(t) \subset U_P(t)$ is convex and so, by Lemma \ref{lem:inj_convex}, $x \mapsto -Lx+Q(x_1+x)$ is
injective on $\S(t)$. Therefore, $x_2=x_2'$ in $U_P(t)$ and this leads to a
contradiction.
\end{proof}

\subsection{Smooth flow modulo gauge}\label{subsec:smoothflow}
We recall from Section \ref{subsec:ghm} that gradient flow of the
vortex functional preserves the complex gauge orbit of the gauged
holomorphic map. Hence, there is a family of time-dependent complex
gauge transformations $g_t$ such that
$$(A_t,u_t)=g_t(A_0,u_0), \quad g_t \in \G, t \in [0,\infty).$$
Then the system of equations \eqref{eq:evolve} generating the gradient flow $(A_t,u_t)$ can be
written as a single equation in $g_t$:
\begin{equation}\label{eq:floweqng}
  \frac{dg_t}{dt}g_t^{-1}=-iF_{A_t,u_t}, \quad g_0=\Id.
\end{equation}
To write $F_{A_t,u_t}$ in terms of $g_t$, we need some
preliminaries. We follow Donaldson \cite{Do:ASD}.
\subsubsection{How curvature transforms under complex gauge
  transformations}\label{subsubsec:gcap}
The transformation relation is derived by working on an associated
vector bundle.  First we assume that $K=U(n)$, and define a complex
vector bundle $E:=P \times_K \C^n$. The standard Hermitian metric on
$\C^n$ is preserved by the $K$-action, and hence the vector bundle $E$
has a Hermitian metric. A connection $A$ on the principal bundle $P$
induces a unitary connection on $E$, which is also denoted by $A$.

The transformation relation of the curvature is obtained via a
corresponding relation on the covariant derivative, which in turn is
obtained by relations on the $(0,1)$ and $(1,0)$ parts of the
covariant derivative.  We recall from Section \ref{subsec:conn} that
there is a canonical isomorphism between $\A(P)$, the space of
connections and $\mC(P)$, the space of holomorphic structures on
$P_\C$. The action of the complexified gauge group on $\A(P)$ is
defined by pulling back the action of $\G(P)$ on $\mC(P)$ via that
isomorphism. Further, the holomorphic structure on $P_\C$ corresponding
to the Dolbeault operator $\delbar_A$ also induces a Dolbeault
operator on the vector bundle $E$. Therefore, on the space of sections
$\Gamma(\Sig,E)$,
\begin{align}\label{eq:gcactc}
  \delbar_{g(A)}&=g\circ\delbar_A\circ g^{-1}, \quad g \in \G.
\end{align}
The Hermitian metric on $E$ together with the Riemannian metric on
$\Sig$ give a metric on the spaces $\Om^k(\Sig,E)$. For any connection
$A$, let $(\delbar_A)^*$ denote the formal adjoint of $\delbar_A$
under this metric. It satisfies $\partial_A=*(\delbar_A)^**$. Applying
this identity to the the connection $g(A)$, we get
\begin{align}\label{eq:gcpartial}
  \partial_{g(A)}=(g^*)^{-1}\circ \partial_A \circ g^*, \quad \text{$A
    \in \A(P)$, $g \in \G$,}
\end{align}
where $g^*$ is the adjoint of $g$ under the metric fixed on
$\Gamma(\Sig,E)$. By viewing $K$ and $G$ as matrix groups and using
the fact that $k^*k=\Id$ for $k \in K$, we see that the element $g^*$
lies in $\G(P)$. The covariant derivative corresponding to the
connection $g(A)$ is ${d}_{g(A)}=\delbar_{g(A)}+\partial_{g(A)}$.  For
$g \in \G$, define $h(g):=g^*g \in \G(P)$.  Adding \eqref{eq:gcactc}
and \eqref{eq:gcpartial}, we get $g^{-1}\circ {d}_{g(A)} \circ
g=\delbar_A+h^{-1}\circ \partial_A \circ h$. On vector bundles, the
curvature $F_A$ is equal to $d_A^2$. It transforms as
\begin{align}\label{eq:hcurv}
  g^{-1}\circ F_{g(A)} \circ g&=F_A+\delbar_A(h^{-1}\partial_A
  h)\\
\nonumber &=F_A+h^{-1}(\delbar_A \partial_A h -
  (\delbar_Ah)h^{-1}\partial_Ah).
\end{align}
Since $K$ is isomorphic to the structure group of $E$, $F_{g(A)}$ is
also the curvature of the connection $g(A)$ on $P$.

For a general compact Lie group $K$, there is a $U(n)$ into which it
can be mapped injectively. So, we work on the bundle $P\times_K
\C^n$. We look upon the space of $K$-connections on $E$ as a subset of
$U(n)$-connections. The group action preserves $K$-connections,
because the infinitesimal action $-\delbar_A \xi$ is in
$\Omega^{0,1}(\g)$. All the relations above carry over to the general
case.
\subsubsection{Gauge-invariant version of the flow equations}
Using the transformation relation \eqref{eq:hcurv}, the evolution
equation \eqref{eq:floweqng} can be re-written as
\begin{equation}\label{eq:floweqnh}
  \frac{dh_t}{dt}=-2ig_t^*F_tg_t=-2ih_t(*F_{A_0}+*\delbar_0(h_t^{-1}(\partial_0h_t))+g_t^{-1}u_t^*\Phi
  g_t),
\end{equation}
where $h_t\!=\!g_t^*g_t$. For $0$-forms $\Delta_{A_0}\!=\!
{d}_{A_0}^*{d}_{A_0}\!=\!*\delbar_0 \partial_0$, and so the equation~\eqref{eq:floweqnh} can be modified to
\begin{align}\label{eq:heqn}
  \frac{dh_t}{dt} + \Delta_{A_0}h_t&= -2ih_t\{*F_{A_0}+*(\delbar_0 h_t) h_t^{-1}(\partial_0 h_t) + g_t^{-1}u_t^*\Phi g_t\}\\
  \nonumber h(0)&=\Id,\quad h|_{\partial \Sig}=\Id.
\end{align}
We make a few comments about \eqref{eq:heqn} to show that it is indeed
the {\em gauge-invariant version of the flow equation}. Firstly,
replacing $g_t$ by $k_tg_t$, $k_t \in \K$ does not alter
$h_t$. Secondly, the term $g_t^{-1}u_t^*\Phi g_t$ is $\K$-invariant ---
it is unchanged if $u_t$ and $g_t$ are replaced by $k_tu_t$ and
$k_tg_t$ respectively. Lastly, any time-dependent complex gauge
transformation $g_t'$ that satisfies $(g_t')^*g_t'=h_t$ differs from
$g_t$ by unitary gauge transformations, i.e. $g_t'g_t^{-1} \in \K$. In
particular, $h_t$ is pointwise positive and self-adjoint. So,
$g_t':=\sqrt{h_t} \in \G(P)$ is well-defined and satisfies
$(g_t')^*g_t'=h_t$.  Therefore, given the initial pair $(A_0,u_0)$,
$g_t^{-1}u_t^*\Phi g_t$ is a function of $h_t$ and is given by the
composition
$$h_t \mapsto g_t':=\sqrt{h_t} \mapsto (g_t')^{-1}((g_t'u_0)^*\Phi)g_t'.$$
A solution $h_t$ of the equation \eqref{eq:floweqnh} gives a solution
$\sqrt{h_t}(A_0,u_0)$ of the gradient flow equation modulo gauge.
\begin{proposition}\label{prop:smooth} Suppose the gauged map $(A_0,u_0)$
  is smooth, then the solution of \eqref{eq:heqn} $h_t:[0,\infty) \to \G$
  is smooth except at the corner $\{0\} \times \partial \Sig$. Hence the gradient flow $(A_t,u_t)$ computed in
  Corollary \ref{cor:heat_nonreg} is smooth modulo gauge away from the corner $\{0\} \times \partial \Sig$.
\end{proposition}
\begin{proof}[Proof of Proposition \ref{prop:smooth}]
   We use the following Banach
  space:
$$\mathcal{L}^{k,p}:=L^{p/2}([0,t_0], W^{k,p}(\Sig, P(G))) \cap W^{1,p/2}([0,t_0], L^p(\Sig, P(G))),$$
where $P(G)$ is the associated bundle $P \times_KG$ and $K$ acts on $G$ by conjugation. The groups $K$ and $G$ are viewed as matrix groups so that $P \times_KG$ is a subbundle of a vector bundle. The $W^{k,p}$-completion of the space of sections is defined as in \eqref{eq:wspnorm} below.
The proof of the Proposition is by a bootstrapping argument using the following two observations.

\begin{observation}\label{obs:1}
Let $s \geq 2$, $p>2$ and $\delta \in (0,\hh - \frac 1 p)$ be such that there
  is an integer in the interval $[\frac s 2 -1 +\delta , \frac s 2]$
  (and hence also in the interval $[s-2+2\delta,s]$). Suppose $h \in
  \mathcal{L}^{s,p}([0,T]\times \Sig)$, then the rhs of \eqref{eq:heqn} is in
  $\mathcal{L}^{s-2+2\delta,p}([0,T]\times \Sig)$. The reason is as follows. By
  the presence of an integer in the interval $[\frac s 2 -1 +\delta ,
  \frac s 2]$, the smooth term $g_t^{-1}u_t^*\Phi g_t$ is in
  $\mathcal{L}^{s-2+2\delta,p}$ by Corollary \ref{cor:bundsob_time}. Next
  consider the term $*(\delbar_0 h_t) h_t^{-1}(\partial_0 h_t)$. By
  interpolation, $h_t$ is in $W^{s/2-1+\delta,p}([0,T], W^{2-2\delta,p})$,
  therefore the derivatives $\partial_0 h_t$ and $\delbar_0 h_t$ are
  continuous (or above Sobolev borderline) and can be multiplied using
  Proposition \ref{prop:sobmult}. This implies $h_t \in
  W^{s/2-1+\delta,p}([0,T],L^p)$. The other part that $h_t$ is in
  $L^p([0,T],W^{s-2+2\delta,p})$ is similar and easier.  
\end{observation}
\begin{observation}\label{obs:2}
Suppose $\frac 1 p < k<s+2$, $h \in \mathcal{L}^{k,p}([0,T]\times
  \Sig)$ and the r.h.s. of the equation \eqref{eq:heqn} is in
  $\mathcal{L}^{s,p}([0,T]\times \Sig)$. Then for any $0<\alpha<T$, $h \in
  L^{s+2,p}([0,\alpha]\times \Sig)$ by parabolic regularity (Theorem,
  p.96, \cite{Ham:bdry}).
  \end{observation}

We first show that $h \in \mathcal{L}^{2,p}$ for some $p>3$.
From the proof of Proposition~\ref{prop:flowexist}, $F_{A_t,u_t} \in H^{-\hh+\eps}(H^{2-2\eps})$. Then, by Lemma
\ref{lem:flowong}, there is a solution of the equation
\eqref{eq:floweqng} $g_t \in H^{\hh+\eps}([0,t_0],\G(P)_{H^{2-2\eps}})$. Recall that $\eps \in (0,1/16)$ is a fixed number. By the
multiplication Theorem (see \eqref{eq:multiplicationmixed}),
\begin{equation}
  \label{eq:hreg}
  h_t:=g_t^*g_t \in H^{\hh+\eps}([0,t_0],H^{2-2\eps}(\Sig, P \times_K G))
\end{equation}
and $h_t$ is the solution of the gauge-invariant equation
\eqref{eq:heqn}. 
We first show that the r.h.s. of equation
\eqref{eq:heqn} is in $\mathcal{L}^{0,p}$. The term $g_t^{-1}u_t^*\Phi
g_t$, which is produced by the action of a smooth bundle map on $h_t$, is in $\mathcal{L}^{0,p}$ using
Corollary \ref{cor:bundsob}. By Sobolev multiplication $*(\delbar_0
h_t) h_t^{-1}(\partial_0 h_t)$ is in $H^{\hh+\eps ,1-4\eps}$, and hence it is in $\mathcal{L}^{0,p}$ by Sobolev embedding. Further, by Sobolev embedding, $h_t$ is in $\mathcal{L}^{s,p}$ for some $p>3$ and $s \in (\frac 1 p, \frac 2 p)$. Now Observation \ref{obs:2} from above is applicable, and we have $h \in \mathcal{L}^{2,p}$.

We use induction to show that the solution $h$ is smooth away from the
corner $\{0\} \times \partial \Sig$.  Suppose $k \geq 2$ is an integer
multiple of $\frac 2 7$. If $h \in \mathcal{L}^{k,p}([\alpha,T]\times \Sig)$,
where $p>3$, then the the r.h.s. of \eqref{eq:heqn} is in $h \in\linebreak
\mathcal{L}^{k-12/7,p}([0,T]\times \Sig)$ and by parabolic bootstrapping $h$ is
in $\mathcal{L}^{k+2/7,p}([\alpha',T]\times \Sig)$ for any
$\alpha<\alpha'<T$. From Step 1, we know that $h \in
\mathcal{L}^{2,p}$. Therefore by induction $h$ is smooth on $(0,\infty) \times
\Sig$. Smoothness on $\{0\} \times \interior(\Sig)$ can be shown as in
Hamilton's proof (Theorem, p.119 \cite{Ham:bdry}).
\end{proof}

\begin{remark} For the solution of the heat equation to be smooth at the corner $\{0\} \times \Sig$, an infinite number of compatibility conditions have to be satisfied at that corner (see Remark, p.365  in Evans \cite{Evans}). In our case, without any condition $h$ is continuous at the corner. If $F_{A_0,u_0}|_{\partial \Sig}=0$, then $h$ would be differentiable. For higher differentiability, we would require an appropriate number of normal derivatives of $F_{A_0,u_0}$ to vanish at the boundary.
\end{remark}
The following Lemma is used in the proof of Proposition
\ref{prop:smooth}.
\begin{lemma}\label{lem:flowong} Suppose $s>1$, $r \in (-\hh,\hh)$, $T>0$ and $F_t \in
  H^r([0,T],\linebreak H^s(\Sig,P(\mathfrak{k})))$ be a time-dependent section. Then,
  $g_t$ defined by \eqref{eq:floweqng} is in $H^{r+1}([0,T],
  \G(P)_{H^s})$.
\end{lemma}
\begin{proof} 

  We first prove the result when $F$ is small. In particular, we show
  that there is a constant $\delta$ such that if
  $\Mod{F}_{H^r(H^s)}<\delta$, then $g \in H^{r+1}(H^s)$. We view $K$ and
  $G$ as matrix groups, so that $g_t$ can be viewed as a section of a
  vector bundle.  
The proof is by an implicit function theorem
  argument on the operator
  \begin{equation}
    \label{eq:diffeq}
    \F:H^{r+1}_P([0,T],H^s(\G(P))) \to H^r([0,T],H^s(\Sig,P(\g))), \quad g \mapsto \frac {dg} {dt} g^{-1}.
  \end{equation}
  As in Section \ref{subsec:flowexist}, $H^{r+1}_P$ is the subspace of
  $H^{r+1}$ that consists of sections vanishing at $t=0$. Both
  differentiation and multiplication are smooth operations between
  appropriate Sobolev spaces, so \eqref{eq:diffeq} is differentiable.
  Its linearization at $g=g_0$ is
  \begin{align*}
    D\F_{g_0}: H^{r+1}_P([0,T],H^s(\Sig,P(\g))) &\to H^r([0,T],H^s(\Sig,P(\g))).\\
    \xi &\mapsto [\xi',\frac {dg_0}{dt} g_0^{-1}] + \frac {d\xi'}{dt}, \quad \text{where }\xi'=\xi g_0^{-1}.
  \end{align*}
  The time differentiation operator $\ddt : H^{r+1}_P \to H^r$ is
  invertible, whose inverse is the integral $\int_0$ by Lemma
  \ref{lem:integrationtime}. Suppose the norm of the integral
  operator is bounded by a constant $C$. So, we have $D\F_{\Id}$ is
  invertible and has inverse bounded by $C$. There is a constant
  $\delta>0$ such that if $\Mod{g_0-\Id}_{H^{r+1}(H^s)}<4C\delta$, then,
  \begin{align*}
    \Mod{D\F_{g_0}(\xi) - D\F_{\Id}(\xi)}_{H^r(H^s)} &\leq \Mod{\xi \frac {d(g_0^{-1})} {dt} + [(\xi g_0^{-1}), \frac {dg_0} {dt} g_0^{-1}]}_{H^r(H^s)}\\
    & \leq \frac 1 {2C}\Mod{\xi}_{H^{r+1}(H^s)}.
  \end{align*}
  Then, by the implicit function theorem (Proposition \ref{prop:impfnMS}) if
  $\Mod{F}_{H^r(H^s)}<\delta$, there is a unique $g$ such that $g(0)=\Id$, $\Mod{g-\Id}_{H^{r+1}(H^s)}<4C\delta$ and $\F(g)=iF$. By using the
  norm of time-dependent sections as in Definition \ref{def:normsum},
  the norms of the differentiation, integration and multiplication
  operators are independent of the length $T$ of the time interval,
  and therefore the constants $\delta$ and $C$ are independent of $T$.

  Finally, we prove the result for any $F \in H^r(H^s(\Sig,\g))$.  The
  interval $[0,T]$ can be split up into a finite number of
  sub-intervals $0=t_0\leq t_1 \leq \dots \leq t_n=T$, on each of
  which the norm of $\Mod{F}_{H^r(H^s)}$ is less than $\delta$. On
  each sub-interval, we can find $g_i\in H^{r+1}_P([t_{i-1},t_i],H^s)$
  such that $\frac {dg_i} {dt} g_i^{-1}=(iF)|_{[t_{i-1},t_i]}$ and
  $g_i(0)=\Id$. Then, we define
  $g|_{[t_{i-1},t_i]}=g_ig_{i-1}(t_{i-1})\cdots g_1(t_1)$.
\end{proof}

\begin{proof}[Proof of Theorem \ref{thm:flowintro}]
The existence of the flow line $(A_t,u_t)$ is proved by Corollary \ref{cor:heat_nonreg}. Regularity modulo gauge for this flow line is proved by Proposition \ref{prop:smooth}.
\end{proof}

\section{Convergence}\label{sec:conv}
In this section, we prove all the results regarding convergence of
heat flow trajectories. Theorem \ref{thm:convintro} is proved in Section \ref{subsec:convmodgauge}, Theorems \ref{thm:bdryintro} and \ref{thm:closedintro} are proved in
Section \ref{subsec:uniquelimit} .

\subsection{Some results about gauge transformations}
We recall from Section \ref{sec:prelim} that complex gauge
transformations act on the space of connections and holomorphic
maps. If these spaces have a Banach manifold structure, the actions
are smooth. Suppose $P$ is a principal bundle. In order to define
Sobolev norms, we fix a smooth connection $A \in \A(P)$ for the rest
of Section \ref{sec:conv}. Given a representation $K \hra SO(n)$, we
recall that a completion of the space of sections $\Gamma(\Sig,P \times_K \R^n)$ can be defined under the $W^{s,p}$-norm:
\begin{equation}
  \label{eq:wspnorm}
  \Mod{\sig}_{W^{s,p}}^p:=\sum_{i=0}^s\Mod{\nabla^i_{A}\sig}^p_{L^p}, \quad \sig \in \Gamma(\Sig,P \times_K \R^n).
\end{equation}
Different choices of the connection $A$ would produce equivalent
norms. For $k \in \Z_{\geq 0}$ and $p>1$ the space of
$W^{k,p}$-connections, called $\A^{k,p}(P)$, is the affine space $A_0
+ \Om^1(\Sig, P(\mathfrak{k}))_{W^{k,p}}$, where $A_0$ is any smooth
connection. For spaces above Sobolev borderline, that is, if $kp>2$, the
Sobolev spaces of maps $\Gamma(\Sig,P(X))_{W^{k,p}}$ and unitary
(resp. complex) gauge transformations $\K^{k,p}(P)$
(resp. $\G^{k,p}(P)$) are Banach manifolds. These are modelled on the
Banach spaces $\Gamma(\Sig,u^*T^{\on{vert}}P(X))_{W^{k,p}}$ and
$\Gamma(\Sig,P(\mathfrak{k}))_{W^{k,p}}$ (resp. $\Gamma(\Sig,P(\g))_{W^{k,p}}$).
See, for example, Appendix B in Wehrheim's book \cite{Weh:Uh} for
details. The action of $\K^{k,p}$ on $\A^{k-1,p}$ is smooth. The
following Lemma from \cite{VW:affine} says that the same can also be
said about the action of $\G^{k,p}$ on $\A^{k-1,p}$, along with some
uniform bounds for the action.

\begin{lemma}{\rm (Action of $\G$ on $\A$, \cite[Lemma
    6.4]{VW:affine})} \label{lem:gcactona} Let $\Sig$ be a
  compact Riemann surface with metric, possibly with a smooth boundary
  and $P$ be a principal $K$-bundle. Suppose $k \in \Z_{\geq 0}$ and
  $p>1$ satisfy $kp>2$. Complex gauge transformations in $\G^{k,p}(P)$
  act smoothly on the space of connections $\A^{k-1,p}(P)$.

  Let $A_0$ be a smooth connection on $P$. For any $\eps>0$, there is
  a constant $C$ so that the following is satisfied. For any
  $W^{k-1,p}$ connection $A=A_0+a$ which satisfies
  $\Mod{a}_{W^{k-1,p}(\Sig)}<\eps$ and any $\xi \in W^{k,p}(\Sig,
  P(\mathfrak{k}))$ that satisfies $\Mod{\xi}_{W^{k,p}}<1$,
  \begin{equation}
    \label{eq:gcactbd}
    \Mod{(\exp i\xi)A - A}_{W^{k-1,p}(\Sig)} \leq C\Mod{\xi}_{W^{k,p}(\Sig)}.
  \end{equation}
\end{lemma}

The next Lemma, which is a standard result, says that the action of
complex gauge transformations $\G^{k,p}$ is smooth on the space of
sections $\Gamma(\Sig,P(X))_{k,p}$.
\begin{lemma}\label{lem:gcactonmap} Suppose $k \in \Z_{\geq 0}$ and
  $p>1$ satisfy $kp>2$. The action of complex gauge
  transformations $\G^{k,p}$ is smooth on the space of sections\linebreak
  $\Gamma(\Sig,P(X))_{k,p}$.
\end{lemma}
\begin{proof} 
  For a smooth complex gauge transformation $g_0$ and a small
  enough constant $\delta_1$, $\{e^{i\xi}g_0: \xi \in
  W^{k,p}(\Sig,P(\g)), \Mod{\xi}_{k,p}<\delta_1\}$ is an open set in
  $\G^{k,p}$ and $e^{i\xi}g_0 \mapsto \xi$ is a chart of the Banach
  manifold.  Similarly, for a smooth map $u_0:U \to X$ and a small
  constant $\delta_2$,
$$\{\exp_{u_0}\zeta:\Mod{\zeta}_{k,p}<\delta_2\} \mapsto \zeta \in \Gamma(\Sig, u^*T^{\on{vert}}P(X))_{W^{k,p}}$$
is a chart of the manifold $\Gamma(\Sig,P(X))_{k,p}$. The map
\begin{align}
  \label{eq:actsec}
&\quad\  \Gamma(\Sig,P(\g)) \times \Gamma(\Sig,u_0^*T^{\on{vert}}P(X)))\\
\notag &\ni (\xi,\zeta) \mapsto \zeta_1 \in \Gamma(\Sig,(g_0u_0)^*T^{\on{vert}}P(X))),
\end{align} 
where $e^\xi g_0 \exp_{u_0}\zeta = \exp_{g_0 u_0}\zeta_1$, is a smooth
map of sections, that vanishes on the zero section. Therefore, by
Corollary \ref{cor:bundsob} below, \eqref{eq:actsec} extends to a smooth map
between $W^{k,p}$-sections.  This proves that the map $\G^{k,p} \times
\Gamma(\Sig,P(X))_{k,p} \ni (g,u) \mapsto gu$ is smooth in the
neighborhood of any smooth pair $(g_0,u_0)$. Since smooth elements
form a dense subset of $W^{k,p}$ elements and the constants $\delta_1$, $\delta_2$ can be chosen uniformly, the result is proved.
\end{proof}

The next two results regarding complex gauge transformations are used
in the proof of Proposition \ref{prop:locconv} and Lemma \ref{lem:ubd_bdry}.
\begin{lemma}\label{lem:YMbdry}{\rm (Transforming to flat connections,
    Donaldson \cite[Theorem 1]{Do:bdry})}
  Let $\Sig$ be a compact Riemann surface with metric and with a non-empty boundary. Let $k \in \Z_{\geq 0}$ and
  $p>1$ be such that $(k+1)p>2$.
Let $A$
  be a $W^{k,p}$-connection on the trivial bundle $\Sig \times K$. There
  is a unique $s \in W^{k+1,p}(\Sig,\mathfrak{k})$ satisfying $s|_{\partial \Sig}
  \equiv 0$ such that $e^{is}A$ is a flat connection.
\end{lemma}

\begin{lemma}
  \label{lem:toflat} {\rm (Transforming to flat connections with small
    complex gauge transformations, Lemma 4.3 and Remark 4.4 in
    \cite{VW:affine})} Let $\Sig$, $k$ and $p$ be as in Lemma \ref{lem:YMbdry}. Let $P:=\Sig \times K$ be the
  trivial principal $K$-bundle on
  $\Sig$. 
  There are constants $c_1$, $c_2$ and $c_2'$ so that the following
  holds. Let $A={d} +a$ be a connection on $P$ so that $a \in
  \Om^1(\Sig,\mathfrak{k})_{W^{k,p}}$. If $\Mod{a}_{W^{k,p}(\Sig)}<c_1$, there
  is a unique $\xi \in W^{k+1,p}(\Sig,\mathfrak{k})$ satisfying $\xi|_{\partial
    \Sig}=0$, $F_{e^{i\xi}A}=0$ and $\Mod{\xi}_{W^{k+1,p}} \leq
  c_2\Mod{F_A}_{W^{k-1,p}} \leq c_2'\Mod{a}_{W^{k,p}}$.
  
  Further, on any contractible closed set $\Sig' \subset \on{int}
  \Sig$, there is a gauge transformation $k \in W^{k+1,p}(\Sig',K)$ so
  that $ke^{i\xi}A={d}$ on $\Sig'$.  The gauge transformation $k$ is
  unique up to left multiplication by a constant element in $K$.
\end{lemma}

The following result, which is used in Section \ref{subsec:uniquelimit}, 
shows that the convergence of a sequence of gauge equivalent connections
implies convergence of the corresponding gauge transformations.
\begin{lemma}\label{lem:hausquot} Let $P \to \Sig$ be a principal
  $K$-bundle over a compact Riemann surface with metric and let
  $p>2$. Suppose $\{A_i\}_i$ is a sequence of connections converging to
  $A_\infty$ in $W^{1,p}$. Further, assume that the connections lie in
  the same gauge orbit, i.e. there exists a sequence of gauge transformations $k_i
  \in \K^{2,p}(P)$ satisfying $k_i(A_0)=A_i$.  Then, the sequence $\{k_i\}_i$
  is bounded in $W^{2,p}$ and has a weak limit $k_\infty \in
  W^{2,p}(\K)$. Further, $A_\infty=k_\infty(A_0)$ and so, it is in the
  same gauge orbit as the sequence.
\end{lemma}
\begin{proof} Denote $\theta_i:=k_i(A_0)-A_\infty$. We are given
  $\theta_i \to 0$ in $\Om^1(\Sig,P(\mathfrak{k}))_{W^{1,p}}$ as $i \to \infty$. By embedding $K$ in $U(n)$ for some $n$, we view $k_i$ as sections of vector bundles whose fibers are $n \times n$-matrices.
The relation $A_\infty + \theta_i=k_i(A_0)=(dk_i)k_i^{-1} + k_iA_0k_i^{-1}$ can be re-written as 
\begin{align}\label{eq:gtb}
  dk_i=-k_iA_0+A_\infty k_i +\theta_ik_i.
\end{align}
The terms $A_0$, $A_\infty$ and $\theta_i$ are bounded in $L^p$. Since $K$ is compact, $\Mod{k_i}_{L^\infty}$ is bounded and so, the right hand side is bounded in
$L^p$. This implies that $\Mod{k_i}_{W^{1,p}}$ is uniformly bounded. Next, we show that the right hand-side of~\eqref{eq:gtb} is
bounded in $W^{1,p}$. By Sobolev multiplication (Proposition \ref{prop:sobmult}),
\begin{align*}
  \Mod{\nabla(\theta_ik_i)}_{L^p}\leq
  \Mod{\theta_i}_{W^{1,p}}\Mod{\nabla k_i}_{L^p}+\Mod{\nabla
    \theta_i}_{L^p}\Mod{k_i}_{W^{1,p}}.
\end{align*}
Hence,
$\Mod{k_i\theta_i}_{W^{1,p}}$ is uniformly bounded. Similarly $k_iA_0$
and $A_\infty k_i$ are also uniformly bounded in $W^{1,p}$ and therefore by \eqref{eq:gtb},
$\Mod{k_i}_{W^{2,p}}$ is uniformly bounded. After passing to a
subsequence, $k_i$ weakly converges to a limit $k_\infty$ in $W^{2,p}$ and the convergence
is strong in $C^1$. This implies $k_i(A_0) \to k_\infty(A_0)$ in $L^p$ and
so, $k_\infty(A_0)=A_\infty$.
\end{proof}

\subsection{Convergence of a subsequence modulo
  gauge}\label{subsec:convmodgauge}
The convergence behavior of gauged holomorphic maps is similar to the Gromov convergence of $J$-holomorphic maps, and relies on a uniform energy bound. 
The {\it energy of a gauged map $(A,u)$} on a Riemann
surface $\Sig$ is defined as
$$E(A,u):= \hh\int_\Sig |F(A)|^2+|\Phi \circ u|^2 + |d_A u|^2 \dvol_\Sig.$$ 
\begin{lemma}{\rm(An energy identity, \cite[Theorem 3.1]{CGS})}\label{lem:vortenergy}
  Let $\Sig$ be a compact Riemann surface and $P$ a principal
  $K$-bundle on it. A pair $(A,u) \in \A(P) \times \Gamma(\Sig,P(X))$
  satisfies
  \begin{equation}\label{eq:vortenergy}
      E(A,u)=\int_\Sig |\delbar_A u|^2 +\hh|*F_A + \Phi(u)|^2 \dvol_\Sig +
      \lan \om_X-\Phi,u\ran,
  \end{equation}
  where $\lan \om_X-\Phi,u\ran = \int_\Sig u^*\om - d\lan
  \Phi(u),A\ran.$
\end{lemma}
The last term $\lan \om_X-\Phi,u\ran$ in \eqref{eq:vortenergy} denotes
the pairing of equivariant cohomology and homology. For a closed
$\Sig$, the quantity is an invariant of the homotopy class of $(A,u)$
(see \cite{CGS}). Therefore the quantity is also independent of the
choice of $A$. This topological term is well-defined because $u^*\om -
d\lan \Phi(u),A\ran \in \Om^2(P,\mathfrak{k})$ is equivariant and horizontal, so
it descends to a 2-form on $\Sig$.

\begin{proposition}\label{prop:endec} {\rm(Energy decreases along gradient flow line)} Suppose\linebreak $(A_t,u_t)$ is the
  smooth gradient flow modulo gauge computed by Theorem~\ref{thm:flowintro}. In addition, if the base manifold $\Sig$ has
  boundary, we assume $u_0(\partial \Sig) \subset \Phi^{-1}(0)$. Then,
  the energy $E(A_t,u_t)$ decreases with time $t$.
\end{proposition}
\begin{proof} The Proposition is proved using the energy identity
  \eqref{eq:vortenergy}. The energy of a holomorphic pair $(A,u)$ is
  the sum of the vortex functional and the term $\int_\Sig u^*\om -
  d\lan \Phi(u),A\ran$. The vortex functional decreases along the flow
  line, so the result is proved by showing that the latter term is
  constant along the flow line. For a closed $\Sig$, the quantity is
  an invariant of the homotopy class of $(A,u)$, and is therefore
  constant (see \cite{CGS}). If $\Sig$ has boundary, we fix a 
  trivialization of the bundle $P \to \Sig$. Under this
  trivialization, the family of maps is $u_t:\Sig \to X$, and there is
  a family of complex gauge transformations $g_t:\Sig \to G$ such that
  $g_tu_0=u_t$ and $g_t|_{\partial \Sig}=\Id$. Therefore,
  $u_t|_{\partial \Sig}$ is time-dependent and so, $\int_\Sig
  u_t^*\om$ is time independent. The remaining part of the term
  $\int_\Sig d\lan \Phi(u),A\ran$, which is equal to $\int_{\partial
    \Sig} \lan \Phi(u),A\ran$ vanishes, because $\Phi(u_t)|_{\partial
    \Sig}=0$ for all $t$.
\end{proof}

The convergence result, Theorem \ref{thm:convintro} \eqref{part:convb}, is proved using
the following local result on convergence of gauged
holomorphic maps away from the bubbling set.

\begin{definition}\label{def:bubset}{\rm(Bubbling set)} Suppose
  $(A_i,u_i)$ is a sequence of gauged holomorphic maps defined on a
  Riemann surface $\Sig$. A point $z \in \interior \Sig$ is in the bubbling set of the sequence $(A_i,u_i)$
  if and only if there is a sequence of points $z_i$ in $\Sig$
  converging to $z$ which satisfy $|d_{A_i}u_i(z_i)| \to \infty$ as $i
  \to \infty$.
\end{definition}

\begin{proposition}\label{prop:locconv}
\hspace{-1ex}{\rm(A local convergence result for gauged holomorphic maps)}
  Suppose $U$ is a contractible compact connected Riemann surface with
  metric and $P:=U \times K$ is the trivial principal
  $K$-bundle. Suppose $(A_i,u_i) \in H^2 \times H^3$ is a sequence of
  gauged holomorphic maps with a uniform energy bound $E(A_i,u_i) \leq
  k$. Further, we are given that $A_i$ converges weakly to a limit
  connection $A_\infty$ in $H^2(U)$. Then, there is a subsequence of $(A_i,u_i)$ (also denoted by $(A_i,u_i)$), for which the bubbling set $Z \subset \interior(U)$ is finite and there is a limit map $u_\infty:\interior(U) \to X$ that
  satisfies $\delbar_{A_\infty}u_\infty=0$ and such that $u_i$
  converges to $u_\infty$ weakly in $H^3(S)$ for any compact subset $S
  \subset U \bs (Z \cup \partial U)$.
\end{proposition}
\begin{proof}[Proof of Proposition \ref{prop:locconv}] 
  First, by a sequence of converging complex gauge transformations, we
  transform the connections $A_i$ to trivial connections.  We observe
  that after passing to a subsequence, the weak convergence $A_i
  \weakto A_\infty$ in $H^2(U)$ implies strong convergence in the norm
  $W^{1,p}(U)$ for any $p>2$.  By Donaldson's result (Lemma
  \ref{lem:YMbdry}), there is a complex gauge transformation $e^{is}$
  on $U$, where $s \in W^{2,p}(U,\mathfrak{k})$, $s|_{\partial U}=0$ and
  $F_{e^{is}A_\infty}=0$. Since the action of complex gauge
  transformations on the space of connections is smooth (Lemma
  \ref{lem:gcactona}), the sequence $e^{is}A_i$ converges to
  $e^{is}A_\infty$ strongly in $W^{1,p}(U)$. Since the curvature map
  $A \mapsto F_A$ is a continuous map between the spaces $W^{1,p} \to
  L^p$, the sequence of curvatures $F_{e^{is}A_i}$ converges to zero
  in $L^p(U)$. Next, by Lemma \ref{lem:toflat}, there exists a
  sequence $\xi_i$ that converges to zero in $W^{2,p}(U)$ that
  satisfies $\xi_i|_{\partial U}=0$ and such that
  $F_{e^{i\xi_i}e^{is}A_i}=0$. Using the convergence $\xi_i \to 0$ in
  $W^{2,p}(U)$, we continue to have $e^{i\xi_i}e^{is}A_i \to
  e^{is}A_\infty$ in $W^{1,p}(U)$. Fix a point $p \in U$. Again using
  Lemma \ref{lem:toflat}, there is a unique sequence of unitary gauge
  transformations $k_i \in W^{2,p}(\Sig,K)$, for all $i$ including
  $i=\infty$, such that $k_ie^{i\xi_i}e^{is}A_i$ is the trivial
  connection and $k_i(p)=\Id$. By a bootstrapping procedure using the
  transformation relation \eqref{eq:gatran}, we can see that $k_i$
  converges to $k_\infty$ in $W^{2,p}(U)$. We denote the cumulative
  complex gauge transformations as $g_i:=k_ie^{i\xi_i}e^{is}$ and
  $g_\infty:=k_\infty e^{is}$. The sequence $g_i$ converges to
  $g_\infty$ in $W^{2,p}(U)$.

The transformed gauged holomorphic maps are just $J$-holomorphic maps to the target on which Gromov convergence results apply.
Since the connections $g_iA_i$ are
  trivial, the holomorphicity condition reduces to
  $\delbar(g_iu_i)=0$, i.e. $u_i:U \to X$ are holomorphic maps. By
  elliptic regularity, $g_iu_i$ is smooth on $\interior(U)$.  Energy of
  gauged holomorphic maps also transforms continuously under $W^{2,p}$
  complex gauge transformations. Therefore we have a uniform bound on
  $E(g_iA_i,g_iu_i)$, which implies a uniform bound on the quantity
  $\Mod{du_i}^2_{L^2(U)}$, which is the energy of a holomorphic map as
  in McDuff-Salamon \cite{MS}. The convergence modulo bubbling for
  holomorphic maps \cite[Lemma 4.6.1]{MS} is now
  applicable. Therefore, there is a subsequence of $g_iu_i$, a finite
  set $Z' \subset \interior (U)$ and a limit map $u_\infty'$ such that
  \begin{enumerate}
  \item a point $z \in Z'$ if and only if there is a sequence $z_i \to
    z$ such that $|{d} (g_iu_i)(z_i)| \to \infty$ as $i \to \infty$.
  \item On any compact subset $S \subset U \bs (Z' \cup \partial U)$,
    the sequence $g_iu_i$ converges smoothly to $u_\infty'$.
  \end{enumerate}

  Finally, we reverse the complex gauge transformations $g_i$. By the
  continuous action of complex gauge transformations on holomorphic
  maps, we can conclude that the bubbling set $Z$ of the sequence $(A_i,u_i)$ coincides with $Z'$ and that the conclusions of the
  Lemma hold with $u_\infty:=g_\infty^{-1}u_\infty'$. We explain in
  more detail why $Z$ is equal to $Z'$. The connections $A_i$ are
  bounded in $H^2$ and hence in $C^0$, by the Sobolev embedding
  $H^2(U) \hra C^0(U)$. By the compactness of the target space, the
  term $(A_i)_{u_i}$ is uniformly bounded in $C^0$. Then, by the
  relation ${d}_{A_i}u_i={d} u_i + (A_i)_{u_i}$, we can say that for any
  sequence of points $z_i$ in $U$, the sequence $|{d} u_i(z_i)|$ is
  unbounded exactly when $|{d}_{A_i} u_i(z_i)|$ is unbounded. Next, we
  have the relation
$$d(g_i u_i)=({d} g_i g_i^{-1})_{u_i}+ g_i {d} u_i.$$
The first term in the r.h.s. is uniformly bounded in $C^0$ because of
a bound on $\Mod{g_i}_{W^{2,p}}$ and by compactness of the target
$X$. The uniform $C^0$ bound on $g_i$ also implies that at any point
$z \in U$, the point-wise norms $|g_i {d} u_i(z)|$ and $|{d} u_i(z)|$
differ by a multiplicative factor that is uniformly bounded above and
below. This proves that $Z=Z'$.
\end{proof}

We now prove parts \eqref{part:conva}, \eqref{part:convb} and \eqref{part:convd} of Theorem
\ref{thm:convintro}. Part \eqref{part:convc} is proved along with Theorem
\ref{thm:bdryintro} in Section \ref{subsec:uniquelimit}.

\begin{proof}[Proof of Theorem \ref{thm:convintro} \eqref{part:conva}, \eqref{part:convb}, \eqref{part:convd}]
  Part \eqref{part:conva} is a convergence result for
  connections which is proved using Uhlenbeck compactness theorem and an energy bound on the heat flow.  Let $(\tilde A_t, \tilde u_t) \in
  C^0_{loc}([0,\infty),H^1 \times C^0)$ be the solution of the
  gradient flow equation \eqref{eq:evolve}. Further, let $(A_t,u_t)$
  be the smooth solution of the flow equation modulo gauge. Recall
  from the proof of Theorem \ref{thm:flowintro} that $(A_t,u_t)$ is
  related to $(\tilde A_t, \tilde u_t)$ by a family of $H^2$ gauge
  transformations. Let $\tilde F_t=*F(\tilde A_t)+\Phi(\tilde u_t)$.
  By the heat flow equation \eqref{eq:evolve2},
  \begin{align*}
    \ddt \Mod{\tilde F_t}^2_{L^2}=\int_\Sig\langle \tilde F_t,
    {d}_{\tilde A_t}^*{d}_{\tilde A_t}\tilde F_t + \tilde
    u_t^*{d}\Phi(J\tilde F_t)_{\tilde u_t}\ran \dvol_\Sig.
  \end{align*}
  The quantities involved are gauge-invariant, so we can write
  \begin{align*}
    \ddt \Mod{F_t}^2_{L^2}&=\int_\Sig\langle  F_t, {d}_{ A_t}^*{d}_{ A_t} F_t +  u_t^*{d}\Phi(J F_t)_{ u_t}\ran \dvol_\Sig\\
    &=\Mod{{d}_{ A(t)} F_t}^2_{L^2}+\int_\Sig g_X((F_t)_{ u_t},(
    F_t)_{u_t})\dvol_\Sig.
  \end{align*}
  In the above calculation, the boundary term $\int_{\partial \Sig}
  \lan F_t,d_{A_t}F_t\ran$ vanishes. The quantity $\Mod{ F_t}_{L^2}$
  is positive and decreasing. So, one can choose a sequence $\{t_i\}_i$
  ($t_i \to \infty$ as $i \to \infty$), such that
  \begin{equation}
    \label{eq:decay}
    \Mod{{d}_{A_{t_i}}F_{t_i}}_{L^2}, \Mod{(F_{t_i})_{u_{t_i}}}_{L^2} \to 0 \text{ as }i \to \infty.    
  \end{equation}
  We replace the subscripts $t_i$ by $i$. By Proposition
  \ref{prop:endec}, the energy of the sequence $E(A_i,u_i)$ is
  bounded, which implies $\Mod{{d}_{A_i}u_i}_{L^2}<c$ for all $i$. The
  $L^2$-bound on $d_{A_i}F_i$ from \eqref{eq:decay} implies that
  $\sup_i\Mod{{d}_{A_i}*F_{A_i}}_{L^2}<\infty$. Then, we can apply
  Uhlenbeck compactness (Proposition \ref{prop:uhhigh}) and obtain a sequence of
  gauge transformations $\{k_i\}_i$ in $H^3(\K)$, such that $k_i(A_i)$
  converges weakly to $A_\infty$ in $H^2$ and strongly in $W^{1,p}$,
  because of the compact embedding $H^2 \hra W^{1,p}$.

  We now show convergence of the maps $k_iu_i$. We know that the
  connections $k_iA_i$ converge to a limit connection weakly in
  $H^2(\Sig)$ and the energy of the gauged maps $(A_i,u_i)$ is
  bounded. We choose a finite cover of $\Sig$ by contractible compact
  sets $\{U_\alpha\}_{\alpha \in \A}$ with smooth boundary. On any of
  these subsets $U_\alpha$, we can apply the local convergence result
  Proposition \ref{prop:locconv} on the sequence
  $k_i(A_i,u_i)|_{U_\alpha}$. The conclusion of Proposition
  \ref{prop:locconv} is that after passing to a subsequence, the
  bubbling set $Z \subset \interior(U_\alpha)$ of $k_iu_i$ is finite
  and there is a limit map $u_\infty:\interior(U_\alpha) \to X$ such
  that $k_iu_i$ converges to $u_\infty$ weakly in $H^3(S)$ for all
  compact subsets $S \subset U_\alpha \bs (Z \cup \partial
  U_\alpha)$. By applying on $U_\alpha$ for all $\alpha \in \A$ and
  successively passing to subsequences of $k_i(A_i,u_i)$, we obtain a
  finite bubbling set $Z \subset \interior(\Sig)$ and a limit
  $u_\infty:\interior(\Sig) \to X$ such that $k_iu_i$ converges to
  $u_\infty$ weakly in $H^3(S)$ for all compact subsets $S \subset
  \Sig \bs (Z \cup \partial \Sig)$. This proves part \eqref{part:convb} of the
  Theorem.

  We now prove part \eqref{part:convd} of the Theorem. We remark that for the case
  when $\Sig$ has boundary, we have defined $u_\infty$ only in the
  interior of $\Sig$. While proving part \eqref{part:convc} later, we will show that
  $u_\infty$ extends to the boundary.  But for now
  $u_\infty|_{\interior(\Sig)}$ is enough to prove part \eqref{part:convd}, as we
  will only prove the relations $d_{A_\infty}F_\infty=0$ and
  $(F_\infty)_{u_\infty}=0$ weakly.  Denote
  $(A_i',u_i'):=k_i(A_i,u_i)$ and
  $F_i':=\Ad_{k_i}F_i=*F(A'_i)+\Phi(u_i')$. Consider the sequence
  $\Phi(u_i')$ in $L^p$
  \begin{align*}
    \Mod{\Phi(u_i')-\Phi(u_\infty)}_{L^p(\Sig)} &\leq \Mod{\Phi(u_i')-\Phi(u_\infty)}_{L^p(\Sig\bs B_\eps(Z\cup \partial \Sig))}\\
    &\quad +\Mod{\Phi(u_i')}_{L^p(B_\eps(Z\cup \partial\Sig))}+\Mod{\Phi(u_\infty)}_{L^p(B_\eps(Z\cup \partial\Sig))}.
  \end{align*}
  Since, $\Mod{\Phi}_{L^\infty}$ is bounded, the second and third
  terms can be made small by taking small enough $\eps$. From \eqref{part:convb}, we
  have $u'_i \to u_\infty$ in $L^p(\Sig \bs B_\eps(Z))$. Therefore
  $\Phi(u_i') \to \Phi(u_\infty)$ in $L^p$. The sequence $F(A_i')$
  converges to $F(A_\infty)$ in $L^p$ as $A_i' \to A_\infty$ in
  $W^{1,p}$. Adding, we get $F_i' \to F_\infty$ in $L^p$. We know from
  earlier in the proof, that $\Mod{F_i'}_{H^1}$ is uniformly
  bounded. Hence, after passing to a subsequence, $F_i'$ converges to
  a limit weakly in $H^1$ and strongly in $L^p$. This limit must be
  $F_\infty$ and therefore ${d}_{A_i'}F_i' \weakto
  {d}_{A_\infty}F_\infty$ in $L^2$. Further, by the convergence in
  \eqref{eq:decay}, we have ${d}_{A_\infty}F_\infty=0$.

  It remains to show $(*F_\infty)_{u_\infty}=0$, for which we work on
  local trivializations of the bundle $P$. Let $S \subseteq \Sig \bs (Z
  \cup \partial \Sig)$ be a compact set on which $P$ is trivializable. Then, $u_i'$ is a map from $S$ to $X$. Since
  $u_i' \to u_\infty$ in $C^0(S)$, we can take $S$ small enough that
  for any $i$, $u_i'(S) \subseteq V$ and $V$ is a chart of $X$ that is
  bi-holomorphic to a subset of $\C^n$. Therefore, we may view $u_i'$ as a
  map from $S$ to $\C^n$.  Define a map
  \begin{align}\label{eq:defl}
    L:S \to \on{Hom}(\mathfrak{k}, \C^n) \qquad x \mapsto (\xi \mapsto \xi_x).
  \end{align}
  The map $L$ is smooth and $(F_i')_{u_i'}$ can be written as $(L
  \circ u'_i)F_i'$. Since $L \circ u'_i $ converges to $L \circ
  u_\infty$ in $C^0$ and $F_i'$ converges to $F_\infty$ in $L^p$, we
  get $(F_i')_{u'_i} \to (F_\infty)_{u_\infty}$ in $L^p(S)$. By the
  convergence in \eqref{eq:decay}, $(F_\infty)_{u_\infty}=0$ on $S$
  and hence on $\Sig \bs (Z \cup \partial \Sig)$. That means
  $(F_\infty)_{u_\infty}=0$ almost everywhere on $\Sig$ and this
  proves the result.
\end{proof}

In the process of proving part \eqref{part:convb} of Theorem \ref{thm:convintro} for closed base manifolds, we have proved the following statement for the case when $\Sig$ has boundary.
\begin{proposition}\label{prop:bdrylim} Assume the setting of Theorem \ref{thm:convintro}. Suppose $\{t_i\}_i$ be the sequence of increasing time points and $k_i$ the sequence of gauge transformations from the conclusion of Theorem \ref{thm:convintro} part \eqref{part:conva}. Then, after passing to a subsequence, there is a finite bubbling set $Z \subset \interior(\Sig)$ and a limit map $u_\infty:\interior(\Sig) \to P(X)$ such that $k_iu_{t_i}$ converges to $u_\infty$ weakly in $H^3(S)$ for any compact set $S \subset \Sig \bs (\partial \Sig \cup Z)$. Further, $F_\infty:=*F_{A_\infty} + \Phi(u_\infty)=0$.
\end{proposition}
\begin{proof} All the statements except $F_\infty=0$ have been proved in the process of proving Theorem \ref{thm:convintro}, part \eqref{part:convb}. For this, we recall from the proof of part \eqref{part:convd} of Theorem \ref{thm:convintro} that $F_\infty$ is the weak $H^1(\Sig)$-limit of the sequence $F_i$. The boundary trace map
$$H^1(\Sig,E) \to H^{1/2}(\partial \Sig, E|_{\partial \Sig}) \quad \sig \mapsto \sig|_{\partial \Sig}$$
is continuous. Since $F_i'=0$ on the boundary for all $i$,
$F_\infty$ is also zero on the boundary. Further, by Theorem \ref{thm:convintro}\eqref{part:convd}, we know that $d_{A_\infty}F_\infty=0$, therefore $F_\infty=0$ on $\Sig$. 
\end{proof}

\subsection{Unique limit of heat flow} \label{subsec:uniquelimit}

In this Section, we prove Theorems \ref{thm:bdryintro} and
\ref{thm:closedintro}, which say that if the base manifold has
boundary, or if the limit map takes the generic point to the
semistable locus, then the heat flow trajectory has a unique limit up
to gauge transformations. Furthermore, the limit lies in the complex
gauge orbit containing the flow line. 
We first discuss the additional assumptions required in the case when
$\Sig$ does not have boundary.
\begin{assumption}\label{ass:freeaction} If $\partial \Sig \neq
  \emptyset$, the action of $K$ on $\Phi^{-1}(0)$ has finite
  stabilizers.
\end{assumption}
This assumption implies that the $G$-action on
$G\Phi^{-1}(0)$ has finite stabilizers. In fact this is the open stratum of the gradient flow
of $|\Phi|^2$ studied by Kirwan \cite{Kirwan}. In case, the target $X$
has the structure of a projective variety, then $G\Phi^{-1}(0)$
coincides with the semistable locus $X^{\on{ss}}$, so in any case we denote
$G\Phi^{-1}(0)$ by $X^{\on{ss}}$. The complement $X\bs X^{\on{ss}}$ is a union of
complex submanifolds of codimension at least 2. Therefore, for a gauged
holomorphic map $(A,u)$, if $u^{-1}(P(X^{\on{ss}}))$ is non-empty, then it
must be the complement of a finite subset of $\Sig$. In that case, we
say that the gauged map $(A,u)$ is {\it generic}.  Under the
Assumption \ref{ass:freeaction}, the constant
\begin{equation}\label{eq:c0def}
  c_0:=\inf\{|\Phi(x)|: \text{stabilizer of $x$ is infinite}\}.
\end{equation}
is positive. We will show that the hypothesis $E(A_0,u_0) \leq
c_0^2\on{vol}(\Sig)$ in Theorem \ref{thm:closedintro} ensures that for
flow line starting at $(A_0,u_0)$, the limit modulo bubbling
$(A_\infty,u_\infty)$ is generic.

We now outline the proof of Theorem \ref{thm:bdryintro} and
\ref{thm:closedintro}. In both cases, the limit of the heat flow
modulo bubbling $(A_\infty,u_\infty)$ is a vortex. Proposition
\ref{prop:orbitvortex} below says that if a gauged holomorphic map is
`close' to a vortex away from bubbling points, then it can be complex
gauge transformed to a vortex, via a small complex gauge
transformation. This Proposition, when applied to a converging sequence
of points on the gradient flow, would produce a sequence of
vortices. But a complex gauge orbit has at most one unitary gauge
orbit of vortices. Therefore the sequence of vortices are actually
unitary gauge equivalent. The proof of the Theorems is finished by
showing that, modulo unitary gauge equivalence, the sequence of
vortices is the limit of the heat flow.
\begin{proposition} \label{prop:orbitvortex}{\rm(Pairs close to a
    vortex are complex gauge equivalent to a vortex)} Let $k=0$ or $1$ and
  $p>1$ be such that $(k+1)p>2$. Suppose $\Sig$ is a compact Riemann
  surface possibly with a smooth boundary. Let $(A_i,u_i)$ be a
  sequence of gauged holomorphic maps on $\Sig$. Suppose $A_i \to
  A_\infty$ in $W^{k,p}$ and there is a finite set $Z \subseteq \Sig$
  so that $u_i \to u_\infty$ in $C^0$ on compact subsets of $\Sig \bs
  (Z \cup \partial \Sig)$. Also, $F_i:=*F(A_i)+u_i^*\Phi \to 0$ in
  $W^{k-1,p}$. Further,
  \begin{equation}
    \label{eq:condbd}
    \text{$\Sig$ has boundary or $u_\infty$ is generic}.
  \end{equation}

  Then, there exist constants $C$ and $i_0$ so that for $i>i_0$, there
  exists $\xi_i \in W^{k+1,p}(\Sig,P(\mathfrak{k}))$ such that $\xi_i|_{\partial
    \Sig}=0$ and $(\exp i\xi_i)(A_i,u_i)$ is a vortex and satisfies
  $\Mod{\xi_i}_{W^{k+1,p}}<8C\Mod{F_i}_{W^{k-1,p}}$.
\end{proposition}
\begin{remark}
  Proposition \ref{prop:orbitvortex} is only used in the case $k=1$, but the lower regularity
result for $k=0$ has applications in the article \cite{VW:affine} and
does not involve extra work.  For the $k=0$ case, we remark that if $p>2$, the curvature
of an $L^p$-connection is in $W^{-1,p}$ by Sobolev multiplication
(Proposition~\ref{prop:sobmult}). 
\end{remark}
\begin{proof}[Proof of Proposition \ref{prop:orbitvortex}]
  The proof of the Proposition is by applying the implicit function
  Theorem (in the form of Proposition \ref{prop:impfnMS}) to the
  functions $\F_i$ defined below.  For every $(A_i,u_i)$, define
  \begin{align*}
    \F_i:\Gamma(\Sig,P(\mathfrak{k}))_\partial \to \Gamma(\Sig,P(\mathfrak{k})), \quad \xi \mapsto
    F_{(\exp i\xi)A_i, (\exp i\xi)u_i}.
  \end{align*}
  Here $\Gamma(\Sig,P(\mathfrak{k}))_\partial:=\{\xi \in \Gamma(\Sig,P(\mathfrak{k})):\xi|_{\partial
    \Sig}=0\}$. We recall that for any $s>\frac 1 p$,
  $W^{s,p}_\partial(\Sig,P(\mathfrak{k}))$ is the subspace of
  $W^{s,p}(\Sig,P(\mathfrak{k}))$ consisting of sections $\sig$ whose boundary
  trace $\sig|_{\partial \Sig}$ vanishes.  The map $\F_i$ extends to a
  smooth map between Sobolev completions $\F_i:W^{k+1,p}_\partial \to
  W^{k-1,p}$. This is because the action of $W^{k+1,p}$-complex gauge
  transformations on gauged holomorphic pairs in $\A^{k,p} \times
  W^{k+1,p}(\Sig,P(X))$ is smooth (see Lemmas \ref{lem:gcactona} and
  \ref{lem:gcactonmap}), and the map $(A,u) \mapsto *F_A + \Phi(u)$ is
  a smooth map from $\A^{k,p} \times W^{k+1,p}(\Sig,P(X))$ to
  $W^{k-1,p}(\Sig,P(\mathfrak{k}))$. The differential of $\F_i$ at a point $\xi
  \in W^{k+1,p}_\partial$ is given by
$$D\F_i(\xi)\xi_1={d}_{e^{i\xi}A_i}^*{d}_{e^{i\xi}A_i}\xi_1+(e^{i\xi}u_i)^*{d}\Phi(J(\xi_1)_{e^{i\xi}u_i}):W^{k+1,p}_\partial \to W^{k-1,p}.$$
This is because for any $\zeta \in \Gamma(\Sig,P(\mathfrak{k}))$, the action of
the infinitesimal complex gauge transformation $i\zeta$ on a
connection $A$ is given by $*d_A\zeta$ and the curvature varies with
the connection as
 $$F_{A+ta}=F_A+td_Aa+\frac{t^2}2[a \wedge a].$$

 {\sc Step 1}: {\em The operator  $D\F_i(0)$ is invertible for large $i$ including $i=\infty$.}\\
 The operator $\Id + {d}_{A_i}^*{d}_{A_i}:W^{k+1,p}_\partial(\Sig,
 P(\mathfrak{k})) \to W^{k-1,p}(\Sig, P(\mathfrak{k}))$ is invertible because the
 Dirichlet problem has a unique solution (see Appendix D in
 \cite{Weh:Uh}). Therefore, the Fredholm index of $\Id + {d}_{A_i}^*{d}_{A_i}$ is $0$. The difference\linebreak
 $u_i^*{d}\Phi(J(\cdot)_{u_i})-\Id$ is a compact perturbation, so
 $D\F_i(0)$ also has Fredholm index $0$. For any non-zero $\xi_1 \in
 W^{k+1,p}_\partial$,
 \begin{align}\label{eq:oneone}
  &\quad\ \langle
   {d}_{A_i}^*{d}_{A_i}\xi_1+u_i^*{d}\Phi(J(\xi_1)_{u_i}),\xi_1\rangle_{\mathfrak{k}}\\
   \notag &=\Mod{{d}_{A_i}\xi_1}_{L^2}^2
   + \int_X\omega_{u_i}((\xi_1)_{u_i},J(\xi_1)_{u_i})>0,
 \end{align}
 For $\Sig$ with boundary, this inequality follows easily. Otherwise,
 by hypothesis \eqref{eq:condbd}, the set $u_\infty^{-1}(P(X^{\on{ss}}))
 \cap \{\xi \neq 0\}$ is a non-empty open set, and by Assumption
   \ref{ass:freeaction}, for any point $x$ in this set $\xi(x)_X \neq
   0$. Therefore $D\F_\infty(0)$ is injective.  Since $u_i$ converges to
   $u_\infty$ away from a finite set, the maps $u_i$ are generic for
   large $i$ and so, $D\F_i(0)$ is also injective. Since the Fredholm
   index of $D\F_i(0)$ is $0$, the operators are onto for large $i$.

   {\sc Step 2}: {\em For large $i$, $\Mod{D\F_i(0)^{-1}}<C$ and $C$ is independent of $i$.}\\
   Let $Q_i$ and $Q_\infty$ denote the inverses of $D\F_i(0)$ and
   $D\F_\infty(0)$ respectively. We will proceed by showing that the
   difference between $D\F_\infty(0)$ and $D\F_i(0)$ is small and so
   $\Mod{Q_i}$ can be bounded in terms of $\Mod{Q_\infty}$. For
   notational convenience, we define an operator $L_x$ for every $x
   \in X$,
   \begin{align*}
     L_x : \mathfrak{k} \to \mathfrak{k}, \quad \xi \mapsto {d}\Phi_x(J\xi_x).
   \end{align*}
   For any $\xi_1 \in W^{k+1,p}_\partial(\Sig,P(\mathfrak{k}))$ and a small
   constant $\eps>0$,
   \begin{align}\label{eq:split3bd}
    &\quad\     \Mod{(D\F_\infty(0)-D\F_i(0))\xi_1}_{W^{k-1,p}(\Sig)}\\
    \notag &\leq \Mod{{d}_{A_\infty}^*{d}_{A_\infty}\xi_1 - {d}_{A_i}^*{d}_{A_i}\xi_1}_{W^{k-1,p}(\Sig)} \\
     \notag &\quad + \Mod{(L_{u_\infty}-L_{u_i})\xi_1}_{W^{k-1,p}(\Sig \bs B_\eps(Z\cup \partial \Sig))} \\
    \notag  &\quad +\Mod{(L_{u_\infty}-L_{u_i})\xi_1}_{W^{k-1,p}(B_\eps(Z\cup \partial
       \Sig))}.
   \end{align}
   Here $B_\eps(Z \cup \partial \Sig)$ denotes an $\eps$-neighborhood
   of $Z \cup \partial \Sig$. The third term in~\eqref{eq:split3bd} satisfies
\begin{align*}
&\quad\ \Mod{(L_{u_\infty}-L_{u_i})\xi_1}_{W^{k-1,p}(B_\eps(Z \cup \partial \Sig))}\\
&\leq 2\Mod{L}_{C^0(X)}\Mod{\xi_1}_{W^{k+1,p}(\Sig)}\cdot vol(B_\eps(Z \cup \partial \Sig)),
\end{align*}
and is bounded by fixing a small enough value of $\eps$ so that
\begin{align*}
  2\Mod{L}_{C^0(X)}\cdot vol(B_\eps(Z \cup \partial \Sig)) \leq \frac
  1 {4\Mod{Q_\infty}}.
\end{align*}
To bound the second term, we use the fact that $u_i \to u_\infty$ in
$C^0(\Sig \bs B_\eps(Z \cup \partial \Sig))$. For large enough $i$,
\begin{align*}
 &\quad\ \Mod{(L_{u_\infty}-L_{u_i})\xi_1}_{W^{k-1,p}(\Sig \bs B_\eps(Z \cup \partial \Sig))} \\
&\leq \Mod{L_{u_\infty}-L_{u_i}}_{C^0(\Sig \bs B_\eps(Z \cup \partial \Sig))}\Mod{\xi_1}_{W^{k+1,p}(\Sig)} \leq \frac 1 {8\Mod{Q_\infty}} \Mod{\xi_1}_{W^{k+1,p}}.
\end{align*}
The first term is bounded similarly. For connections $A$ and $A+a$, we
have the following expansion:
\begin{equation}
  \label{eq:condiff}
  ({d}_{A+a}^*{d}_{A+a}-{d}_A^*{d}_A)\xi_1=*[a \wedge *{d}_A\xi_1]+{d}_A^*[a\wedge \xi_1]+*[a  \wedge *[a\wedge\xi_1]].  
\end{equation}
Therefore, for large enough $i$,
\begin{align*}
&\quad\  \Mod{({d}_{A_\infty}^*{d}_{A_\infty}-{d}_{A_i}^*{d}_{A_i})\xi_1}_{W^{k-1,p}(\Sig)} \\
&\leq c(\Mod{A_\infty-A_i}_{W^{k,p}(\Sig)}+ \Mod{A_\infty-A_i}_{W^{k,p}(\Sig)}^2)\Mod{\xi_1}_{W^{k+1,p}(\Sig)}\\
  &\leq \frac 1 {4\Mod{Q_\infty}}\Mod{\xi_1}_{W^{k+1,p}(\Sig)}.
\end{align*}
Then, $\Mod{D\F_\infty -D\F_i} \leq \frac 1 {2\Mod{Q_\infty}}$ and so,
$\Mod{Q_i} \leq 2\Mod{Q_\infty}$.

{\sc Step 3}: {\em For large $i$ and a section $\xi\in W^{k+1,p}_\partial$ satisfying $\Mod{\xi}_{W^{k+1,p}}<1$, there is
  a constant $c_1$ independent of $i$ such that}
$$\Mod{D\F_i(\xi)-D\F_i(0)} \leq c_1\Mod{\xi}_{W^{k+1,p}}.$$
Proceeding in a similar way as Step 2,
\begin{align}\label{eq:diffderivative}
  \Mod{D\F_i(\xi)-D\F_i(0)} &\leq \Mod{{d}^*_{(\exp i\xi)A_i}{d}_{(\exp
      i\xi)A_i}-{d}^*_{A_i}{d}_{A_i}}\\
\nonumber &\quad + \Mod{L_{(\exp i\xi)u_i}-L_{u_i}}.
\end{align}
Consider the first term. Choose a small constant $\eps>0$. By the convergence of the sequence $A_i$,
for large $i$,
$\Mod{A_i-A_\infty}_{W^{1,p}}<\eps/2$. Choose a smooth
connection $A_0$ such that $\Mod{A_0-A_\infty}_{W^{1,p}}<\eps/2$ so
that we have $\Mod{A_i-A_0}_{W^{1,p}}<\eps$ for large $i$. Using the
base connection $A_0$, we apply Lemma \ref{lem:gcactona}. We can then
conclude that for any $\xi \in W^{k,p}(\Sig,P(\mathfrak{k}))$ satisfying
$\Mod{\xi}_{W^{k+1,p}}<1$, the connection $e^{i\xi}A_i$ is in
$W^{1,p}$. Further, there is a constant $C(\eps)$ independent of $i$
such that
   $$\Mod{(\exp i\xi)A_i - A_i}_{W^{k,p}(\Om)} \leq C\Mod{\xi}_{W^{k+1,p}(\Om)}.$$

   Using the expansion \eqref{eq:condiff} and the multiplication
   theorem (Prop.~\ref{prop:sobmult}), we get
$$\Mod{({d}_{A_i+a}^*{d}_{A_i+a}-{d}_{A_i}^*{d}_{A_i})\xi_1}_{W^{k-1,p}} \leq c(\Mod{a}_{k,p}+\Mod{a}_{k,p}^2)\Mod{\xi_1}_{k+1,p}.$$
Therefore,
\begin{align}\label{eq:adiffbd}
&\quad\  \Mod{({d}^*_{(\exp i\xi)A_i}{d}_{(\exp i\xi)A_i}-{d}_{A_i}^*{d}_{A_i})\xi_1}_{W^{k-1,p}}\\
\notag &\leq c\Mod{(\exp i\xi)A_i - A_i}_{k,p}\Mod{\xi_1}_{k+1,p}\\
\notag  & \leq c\Mod{\xi}_{k+1,p}\Mod{\xi_1}_{k+1,p}.
\end{align}

To bound the second term in \eqref{eq:diffderivative}, we observe that
$\xi \mapsto (L_{(\exp i\xi)u_i}-L_{u_i})$ is a continuous map. So,
$\Mod{L_{(\exp
    i\xi)u_i}-L_{u_i}}_{C^0}<c\Mod{\xi}_{C^0}<c\Mod{\xi}_{k+1,p}$. The
constants are independent of $i$, because by the compactness of $X$,
there is a constant $c$ for which $d_X(e^{i\xi}u_i(z),u_i(z))<c|\xi|$
for any $z \in \Sig$ and $\xi \in \mathfrak{k}$. Now, since $k-1 \leq 0$,
\begin{align}\label{eq:udiffbd}
\Mod{(L_{(\exp i\xi)u_i}-L_{u_i})\xi_1}_{W^{k-1,p}}&\leq c\Mod{L_{(\exp i\xi)u_i}-L_{u_i}}_{C^0}\Mod{\xi_1}_{W^{k-1,p}}\\
\nonumber &\leq c\Mod{\xi}_{k+1,p}\Mod{\xi_1}_{k+1,p}.  
\end{align}

Therefore, by \eqref{eq:adiffbd} and \eqref{eq:udiffbd}, there is a constant $c_1$ independent of $i$, such that for
large enough $i$ and $\Mod{\xi}_{k+1,p}<1$,
$$\Mod{D\F_i(\xi)-D\F_i(0)} \leq c_1\Mod{\xi}_{W^{k+1,p}}.$$

{\sc Step 4}: {\em Finishing the proof of Proposition \ref{prop:orbitvortex}.}\\
Let $\delta_{max}:=1/2Cc_1$. We assume $i$ is large enough that the
results in Steps 1-3 hold, and $\Mod{F_i}_{W^{k-1,p}}<\frac
{\delta_{max}}{8C}$. We can restate the result in Step 3 as: if $\xi$
satisfies $\Mod{\xi}_{W^{k+1,p}}<\delta_{max}$, then,
$\Mod{D\F_i(\xi)-D\F_i(0)} \leq \frac 1 {2C}$.  We apply the implicit
function Theorem (Proposition \ref{prop:impfnMS}) on the function
$F_i$ with $\delta:=8C\Mod{F_i}_{W^{k-1,p}}$. Then, we get $\xi_i \in
W^{k+1,p}_\partial$ so that $\F_i(\xi_i)=0$ and
$\Mod{\xi_i}_{k+1,p}<\delta=8C\Mod{F_i}_{W^{k-1,p}}$.
\end{proof}
The following definition is handy in stating and proving the next few results.
\begin{definition}[$\G_{\partial,K}(P)$] Given a Riemann surface
  $\Sig$ possibly with boundary, and a principal bundle $P \to \Sig$,
  $\G_{\partial,K}$ is the subgroup of the complexified gauge group
  $\G(P)$ consisting of $g \in \G(P)$ for which $g(\partial \Sig)
  \subset K$. Alternately, we can say $\G_{\partial,K}(P)$ consists of
  elements $ke^{i\xi}$, where $(k,\xi) \in \K(P) \times
  \Gamma(\Sig,P(\mathfrak{k}))_\partial$.
\end{definition}

The following is the next result required in the proof of Theorems \ref{thm:bdryintro} and \ref{thm:closedintro}.
\begin{proposition} \label{prop:atmost1vortex}{\rm(At most one vortex
    in a complex gauge orbit)} Let $p>1$ and $k \in \Z_{\geq 0}$ be
  such that $(k+1)p>2$. Let $\Sig$ be a compact connected Riemann
  surface possibly with a smooth boundary. Let $(A_0,u_0)$, $(A_1,u_1)
  \in W^{k,p} \times W^{k+1,p}$ be vortices on a principal bundle $P
  \to \Sig$ that are related by a complex gauge transformation $g \in
  \G^{k+1,p}_{\partial,K}(P)$, i.e. $(A_1,u_1)=g(A_0,u_0)$. Further,
  \begin{equation*}
    \text{$\Sig$ has boundary or $u_\infty$ is generic}.
  \end{equation*}
  Then, $(A_0,u_0)$ and $(A_1,u_1)$ are gauge-equivalent, i.e. $g \in
  \K^{k+1,p}(P)$.
\end{proposition}
\begin{proof}
  A vortex is a zero of the moment map $*F_{A,u}$, so the proof is
  similar to the finite-dimensional case - Lemma
  \ref{lem:uniquemu0}. The Cartan diffeomorphism \eqref{eq:cartan}
  induces a smooth bijection
  \begin{align*}
    \G^{k+1,p} \to \K^{k+1,p} \times W^{k+1,p}(\Sig,P(\mathfrak{k})) && g
    \mapsto (k,\xi) \text{ so that $g=ke^{i\xi}$}.
  \end{align*}
  So, if $(A,u)$ and $(A',u')$ are vortices that are related by a complex
  gauge transformation in $\G_{\partial,K}$, after a gauge transformation, we may assume\linebreak
  $(A',u')=e^{i\xi}(A,u)$ where $\xi \in W^{k+1,p}(\Sig,P(\mathfrak{k}))$ and
  $\xi|_{\partial \Sig}=0$. Let $(A_t,u_t):=e^{it\xi}(A,u)$. For
  $\xi|_{\partial \Sig}=0$,
  \begin{align*}
    \ddt \int_\Sig \lan *F_{A_t,u_t},\xi \ran &= \langle {d}_{A_t}^*{d}_{A_t}\xi+u_t^*{d}\Phi(J(\xi)_{u_t}),\xi\rangle_{\mathfrak{k}}\\
    &=\Mod{{d}_{A_t}\xi}_{L^2}^2 +
    \int_X\omega_{u_t}((\xi)_u,J(\xi)_u)\geq 0.
  \end{align*}
  It is easily seen that the inequality is strict for non-zero $\xi$, in
  case $\Sig$ has non-empty boundary. Otherwise, if $u$ is generic, there is a finite set $Y$ such that $u(\Sig \bs Y) \subset P(X^{\on{ss}})$. This condition is invariant under complex gauge transformations, therefore, for all $t \in [0,1]$, $e^{it\xi}u(\Sig \bs Y) \subset P(X^{\on{ss}})$. By Assumption \ref{ass:freeaction}, for non-zero $\xi$, $\xi_{u_t} \neq 0$ and the above inequality is strict.
  Since $F_{A_0,u_0}=F_{A_1,u_1}=0$, we can conclude $\xi=0$ and that
  the vortices $(A,u)$ and $(A',u')$ are gauge-equivalent.
\end{proof}

The following Lemma proves the technical part of Theorems \ref{thm:bdryintro} and Theorem \ref{thm:closedintro}.
\begin{lemma}
  \label{lem:glim} Suppose  $\Sig$ is a Riemann surface, possibly with smooth
  boundary and $(A_i,u_i) \in W^{1,p} \times W^{2,p}$ is a sequence of
  gauged holomorphic maps on $\Sig$ that satisfies the following.
  \begin{enumerate}
  \item There are complex gauge transformations $g_i \in
    \G^{2,p}_{\partial,K}$ such that $g_i(A_0,u_0)\allowbreak =(A_i,u_i)$.
  \item There is a vortex $(A_\infty,u_\infty)$ such that $A_\infty
    \in W^{1,p}(\Sig)$ and $u_\infty \in\linebreak
    W^{2,p}_{\loc}(\interior(\Sig))$ such that the sequence
    $(A_i,u_i)$ converges to $(A_\infty,u_\infty)$ in the following
    sense:
    \begin{align}\label{eq:convsense}
      A_i &\xrightarrow{W^{1,p}(\Sig)} A_\infty,& u_i
      &\xrightarrow{W^{1,p}(S)} u_\infty 
    \end{align}
for all compact
        subsets $S \subset \Sig\bs(\partial \Sig \cup Z)$.
  \end{enumerate}
  Then, the map $u_\infty$ extends to $\Sig$ as a $W^{2,p}$-map and
  there is a complex gauge transformation $g_\infty\in
  \G_{\partial,K}^{2,p}$ such that
  $(A_\infty,u_\infty)=g_\infty(A_0,u_0)$ and the sequence $g_i$
  converges to $g_\infty$ weakly in $W^{2,p}(\Sig)$.
\end{lemma}
\begin{proof} We first transform the sequence of gauged holomorphic
  maps to vortices via a sequence of small complex gauge
  transformations in a way that the new sequence still has the same
  limit. For this, we observe that the sequence $(A_i,u_i)$ satisfies
  the hypothesis of Proposition \ref{prop:orbitvortex}. This is
  because, by arguments similar to the proof of Theorem
  \ref{thm:convintro} \eqref{part:convd}, the sequence $F_{A_i,u_i}$ converges in
  $L^p(\Sig)$ to $F_{A_\infty,u_\infty}$, which is zero.  By dropping
  a tail of the sequence, we may assume $i_0=0$. Therefore, for all
  $i$, there exist $\xi_i \in W_\partial^{2,p}$ such that
  $(A_i',u_i'):=(\exp i\xi_i)(A_i,u_i)$ is a vortex and $\xi_i \to 0$
  in $W^{2,p}$. The action of complex gauge transformations on gauged
  maps is continuous, and so, the sequence $(A_i',u_i')$ converges to
  $(A_\infty,u_\infty)$ in the sense of \eqref{eq:convsense}.

  We next show that the sequence of vortices are related to each other
  by a sequence of weakly converging unitary gauge
  transformations. The complex gauge transformations of the previous paragraph can be represented schematically as
  \begin{equation*}
    (A_0',u_0') \xleftarrow{\exp(i\xi_0)} (A_0,u_0) \xrightarrow{g_i} (A_i,u_i) \xrightarrow{\exp(i\xi_i)} (A_i',u_i').
  \end{equation*}
  From this diagram, we can conclude that the sequence of vortices
  satisfies $(A_i',u_i')=e^{i\xi_i}g_ie^{-i\xi_0}(A_0',u_0')$, and the
  complex gauge transformations $e^{i\xi_i}g_ie^{-i\xi_0}$ are in
  $\G^{2,p}_{\partial,K}$. By Proposition \ref{prop:atmost1vortex}, up
  to unitary gauge equivalence, there is a unique vortex in a
  $\G^{2,p}_{\partial,K}$-orbit. Therefore, $e^{i\xi_i}g_ie^{i\xi_0}$
  is actually a gauge transformation. We denote
  $k_i:=e^{i\xi_i}g_ie^{-i\xi_0} \in \K^{2,p}$. Since the sequence of
  connections $k_iA_0'$ converges to $A_\infty$ in $W^{1,p}(\Sig)$, by
  Lemma \ref{lem:hausquot}, we can conclude that, after passing to a
  subsequence, the gauge transformations $k_i$ converge weakly to a
  limit $k_\infty$ in $W^{2,p}$.

  This indeed proves the Proposition. The complex gauge
  transformations $g_i$, which are equal to $e^{-i\xi_i}k_ie^{i\xi_0}$,
  converge weakly in $W^{2,p}$ and strongly in $C^1$. The limit is
  $g_\infty:=k_\infty e^{i\xi_0}$. Therefore, the sequence $u_i$ also
  converges weakly in $W^{2,p}(\Sig)$ to a limit, which must agree
  with $u_\infty$ on $\interior(\Sig)$.
\end{proof}

\begin{proof}[Proof of Theorem \ref{thm:bdryintro}]
  The Theorem can be proved in a straightforward manner by applying
  Lemma \ref{lem:glim}. We start with the setting in Theorem
  \ref{thm:convintro}.  We recall that by Theorem \ref{thm:convintro}
  part \eqref{part:conva}, there is an increasing sequence $\{t_i\}_i$ of time points
  and a sequence of gauge transformations $k_i \in \K(P)_{H^3}$ such
  that the sequence of gauged maps $k_i(A_{t_i},u_{t_i})$ converges to
  a limit $(A_\infty,u_\infty)$ in the sense of
  \eqref{eq:convsense}. That is, the convergence of the maps
  $k_iu_{t_i}$ is away from the bubbling points and the boundary.  By
  Proposition \ref{prop:bdrylim}, the limit $(A_\infty,u_\infty)$ is a
  vortex.  We further recall that the gauged maps
  $k_i(A_{t_i},u_{t_i})$ are related to the starting point of the flow
  $(A_0,u_0)$ by a complex gauge transformation in
  $\G_{\partial,K}$. Therefore, Lemma \ref{lem:glim} is applicable on
  the sequence $k_i(A_{t_i},u_{t_i})$, and we can conclude that
  $u_\infty$ extends to the boundary $\partial \Sig$ and is in
  $W^{2,p}(\Sig)$. Further, there is a complex gauge transformation
  $g_\infty \in \G_{\partial,K}^{2,p}$ such that
  $g_\infty(A_0,u_0)=(A_\infty,u_\infty)$, which completes the proof
  of Theorem \ref{thm:bdryintro} \eqref{part:bdrya}. By Proposition
  \ref{prop:atmost1vortex}, there is a unique vortex up to gauge in
  the $\G_{\partial,K}$-orbit of $(A_0,u_0)$, which proves part \eqref{part:bdryb} of
  the Theorem. By the Cartan map, the limit complex gauge
  transformation $g_\infty$ can be written as $g_\infty=k_\infty
  e^{i\xi_\infty}$, where $k_\infty \in \K^{2,p}(P)$, $\xi_\infty \in
  W^{2,p}(\Sig,P(\mathfrak{k}))$ and $\xi|_{\partial \Sig}=0$. The element
  $\xi_\infty$ is uniquely determined.
\end{proof}

The proof of Theorem \ref{thm:convintro} part \eqref{part:convc} follows from the conclusions of Theorem \ref{thm:bdryintro}.
\begin{proof}[Proof of Theorem \ref{thm:convintro} \eqref{part:convc}] Suppose $(A_t,u_t)$ is the smooth gradient flow trajectory modulo gauge. Then, there is a family of complex gauge transformations in $g_t \in \G_{\partial,K}(P)$ such that $g_t(A_0,u_0)=(A_t,u_t)$. By Theorem \ref{thm:convintro} part \eqref{part:conva}, there is a sequence $t_i$ of increasing time points and a sequence of gauge transformations $k_i$ such that the sequence $k_i(A_{t_i},u_{t_i})$ converges modulo bubbling to a limit $(A_\infty,u_\infty)$. The proof of Theorem \ref{thm:bdryintro} proceeded by applying Lemma \ref{lem:glim} to the sequence $k_ig_{t_i}(A_{t_i},u_{t_i})$. One of the conclusions of the Lemma is that the sequence of complex gauge transformations $k_ig_{t_i}$ has a weak limit $g_\infty$ in $W^{2,p}(\Sig)$. By the compact inclusion $W^{2,p} \hra C^1$, the sequence $k_ig_{t_i}$ strongly converges to $g_\infty$ in $C^1(\Sig)$. Therefore, by Lemma \ref{lem:gcactona}, the sequence $k_iu_{t_i}$, which is same as $k_ig_{t_i}u_0$, converges to a limit $g_\infty u_0$ in $C^1(\Sig)$. Since the sequence $k_iu_{t_i}$ converges modulo bubbling to $u_\infty$, $u_\infty$ agrees with  $g_\infty u_0$ away from the bubbling set. Hence, $g_\infty u_0=u_\infty$ on $\Sig$.
\end{proof}

\begin{proof}[Proof of Theorem \ref{thm:closedintro}]
 We first show that the limit of the heat flow modulo bubbling is a vortex.
  The limit $(A_\infty, u_\infty)$ computed in Theorem \ref{thm:convintro} is a critical
  point of the flow and it satisfies ${d}_{A_\infty}F_\infty=0$ and
  $(F_\infty)_{u_\infty}=0$.
  The first equation implies that the conjugacy class of $F_\infty$ is
  constant. If $F_\infty \neq 0$, then all points in the image of $u$ have an infinite stabilizer group, and hence $u$ maps to $X \bs X^{\on{ss}}$ (see Assumption \ref{ass:freeaction}). This implies
  that $|\Phi(u_\infty(x))|>c_0>0$ and hence,
  $E(A_\infty,u_\infty)>c_0^2\on{vol}(\Sig)$. By Proposition
  \ref{prop:endec}, energy of gauged holomorphic maps decreases along the flow line, which means that $E(A_0,u_0) \geq E(A_\infty,u_\infty) \geq
  c_0^2\on{vol}(\Sig)$ which contradicts the hypothesis of the Theorem.

  We have proved that $F_\infty=0$ and the limit  $(A_\infty,u_\infty)$ is generic. Now, Proposition
  \ref{prop:orbitvortex} is applicable and the rest of the
  proof proceeds in the same way as the proof of Theorem
  \ref{thm:bdryintro}.
\end{proof}

\section{Heat flow on vector space target}\label{sec:vstarget}
In this section, we prove that if the target $X$ is a symplectic
vector space with a linear action of the group $K$ and a proper moment
map, we can apriori say that the image of the heat flow is contained
in a compact subset of the target. Then the results about heat flow
with a compact target are applicable. The following Lemma proves the
result in the case the base manifold $\Sig$ does not have boundary.

\begin{lemma}\label{lem:ubd} Suppose $\Sig$ is a compact Riemann
  surface without boundary and $X=\C^n$ with a linear $K$-action and
  proper moment map $\Phi$. Let $P \to \Sig$ be a principal
  $K$-bundle.  Given a constant $k$ there is a compact set $S \subset
  X$ such that if $(A,u)$ is a gauged holomorphic curve on $P$ with
  the energy bound $E(A,u) \leq k$, then $u(\Sig) \subset S$.
\end{lemma}

For the proof, we define the following operator. For any $x \in X$,
define
\begin{equation}\label{eq:defL}
  L_x : \mathfrak{k} \to T_XX,\quad \xi \mapsto \xi_X(x).
\end{equation}
Given a section $u:\Sig \to P(X)$, $L_u \in \Gamma(\Sig,
P(\End(\mathfrak{k},u^*TX)))$ is a section of a vector bundle on $\Sig$.
 
\begin{proof} 
  The proof uses elliptic regularity to produce a $C^0$ bound on
  $u$. We first produce a local $H^1$ bound on the connection $A$ and
  a preliminary $L^2$ bound on the map $u$. This is done using the
  energy bound, which implies $\Mod{F(A)}_{L^2}<k$ and
  $\Mod{\Phi(u)}_{L^2} < k$. By Uhlenbeck's local theorem
  (\cite{Uh:compactness}), we can find a cover of $\Sig$,
  $\cup_\alpha\U_\alpha$ and local trivializations under which the
  connection $A$ is $d+a_\alpha$ on $\U_\alpha$ and
  $\Mod{a_\alpha}_{H^1(\U_\alpha)}<c_k$. Here $c_k$ is a constant
  depending only on $k$. Suppose, under this trivialization $u$ is
  given by $u_\alpha:\U_\alpha \to \C^n$. Since $\Phi$ is a quadratic
  function on $X$ and is proper, we get
  $\Mod{u}_{L^2}<c(1+\Mod{\Phi(u)}_{L^2})<c_k$.

  Now, we apply elliptic regularity. By holomorphicity of $(A,u)$, we
  have $\delbar u_\alpha = (a_\alpha)_{u_\alpha}^{0,1}$. The term
  $(a_\alpha)_{u_\alpha}$ can be seen as the product of two sections
  $L_{u_\alpha} \in \Gamma(\U_\alpha, \End(\mathfrak{k},u_\alpha^*TX))$ and
  $a_\alpha \in \Om^1(\U_\alpha, \mathfrak{k})$. To bound the first term $L_u$,
  we observe that $|L_x|$ grows linearly with $x$, so $|L_x| \approx
  c|\Phi(x)|^{1/2}$. Since $\Mod{\Phi(u_\alpha)}_{L^2}<k$,
  $\Mod{L_{u_\alpha}}_{L^4}<c_k$. Then, by the multiplication theorem,
  $\Mod{(a_\alpha)_{u_\alpha}}_{L^{2+\eps}(\U_\alpha)}<c_k$.  Let
  $\U_\alpha'' \subset \U_\alpha' \subset \U_\alpha$, be such that
  $\{\U_\alpha'\}_\alpha$ and $\{\U_\alpha''\}_\alpha$ still cover
  $\Sig$. We apply interior elliptic regularity twice. First,
  \begin{align*}
    \Mod{u_\alpha}_{W^{1,2}(\U_\alpha')} \leq c(\Mod{\delbar
      u_\alpha}_{L^2(\U_\alpha)}+\Mod{u_\alpha}_{L^2(\U_\alpha)}) \leq
    c_k
  \end{align*}
  By Sobolev embedding, $W^{1,2} \hra L^{2+\eps}$ and so, the
  $L^{2+\eps}$ norms of $u_\alpha$ are bounded. Next,
  \begin{align*}
    \Mod{u_\alpha}_{W^{1,2+\eps}(\U_\alpha'')} \leq c(\Mod{\delbar
      u_\alpha}_{L^{2+\eps}(\U_\alpha')}+\Mod{u_\alpha}_{L^{2+\eps}(\U_\alpha')})
    \leq c_k
  \end{align*}
  By the inclusion $W^{1,2+\eps} \hra C^0$, $u_\alpha(\U_\alpha)$ is
  contained in a compact set $S_\alpha \subseteq \C^n$. The image of
  $u$ is contained in the compact set $\cup_\alpha KS_\alpha$.
\end{proof}

When the base manifold has boundary, we can prove the existence of
heat flow on a larger class of non-compact manifolds. We require that
the target manifold $X$ is {\em equivariantly convex} (as defined by
Cieliebak et al \cite{CGMS}) and has proper moment map. An important
example of equivariantly convex spaces are vector spaces with a linear
group action and a proper moment map. The notion of convexity for
symplectic manifolds first arose in work by Eliashberg and Gromov
\cite{Eliashberg:convexity}.
\begin{definition}\label{def:convexinfty}
  A K\"ahler manifold $(X,\om,J)$ with a Hamiltonian $K$-action is
  {\em equivariantly convex at infinity} if there is a $K$-invariant
  proper function $f:X \to \R_{\geq 0}$ and a value $c_0$ such that if
  $f(x)>c_0$, then
  \begin{align}\label{eq:convexinfty}
    \lan \nabla_\xi \nabla f, \xi \ran \geq 0 \quad \forall \xi \in
    T_xX, \quad df(J\Phi(x)_X) &\geq 0.
  \end{align}
  Here $\nabla f \in \on{Vect}(X)$ is the gradient vector field of $f$
  with respect to the metric $\om(\cdot,J\cdot)$.
\end{definition}
The above definition is equivalent to the definition in \cite{CGMS},
where there is an additional term $\lan \nabla_{J\xi} \nabla f, J\xi
\ran$ in the left hand side of the first equation above. But, when $X$
is K\"ahler $\nabla_{J\xi}=J\nabla \xi$ so that term is equal to $\lan
\nabla_\xi \nabla f, \xi \ran$.  This condition implies that $f$ is
sub-harmonic on holomorphic curves mapping to $f^{-1}(c_0,\infty)$. To
see this, consider a holomorphic curve $u:B_r \subset \C \to
f^{-1}(c_0,\infty) \subset X$. On $B_r$, re-write the holomorphic
coordinate $z$ as $z=s+it$. Then,
\begin{align}\label{eq:subharmonic}
    \Delta (f \circ u) &= \partial_s\lan \nabla f(u), \partial_su \ran +  \partial_t\lan \nabla f(u), \partial_tu \ran\\
 \notag   &=\lan \partial_s\nabla f(u), \partial_su \ran +
    \lan \partial_t\nabla f(u), \partial_tu \ran + \lan \nabla
    f(u), \partial_s^2u+\partial_t^2u\ran\\
    \notag &\geq 0
\end{align}
using \eqref{eq:convexinfty} and the fact that the last term is zero.

\begin{lemma}\label{lem:ubd_bdry} Suppose $\Sig$ is a compact Riemann surface with boundary, and $X$ is a Hamiltonian K\"ahler manifold that is equivariantly convex at infinity and has a proper moment map $\Phi$. Let $P=\Sig \times K$ be the trivial principal $K$-bundle.
  Given a constant $k$ and a $K$-invariant compact set $S_\partial
  \subset X$, there is a compact set $S \subset X$ such that if
  $(A,u)$ is a gauged holomorphic curve on $\Sig$ with the energy
  bound $E(A,u) \leq k$, then $u(\Sig) \subset S$.
\end{lemma}
\begin{proof}
  The proof is by contradiction.  Suppose the lemma is not true. Then
  there is a sequence of gauged holomorphic maps $(A_i,u_i)$
  satisfying the conditions of the Lemma and the union of whose images
  are unbounded in~$X$.

  We first produce a sequence of converging complex gauge
  transformations that make the sequence of connections $A_i$ flat and
  such that the union of images of the sequence of maps $u_i$ is still
  unbounded.  Since $\Mod{F(A_i)}_{L^2(\Sig)}<k$, by Uhlenbeck
  compactness, after passing to a subsequence, there exist gauge
  transformations $k_i \in H^2(\Sig,K)$ so that $k_iA_i$ converges to
  $A_\infty$ weakly in $H^1(\Sig)$. For any $p>2$, after passing to a
  subsequence we have $k_iA_i$ converges to $A_\infty$ strongly in
  $L^p(\Sig)$.  By Lemma \ref{lem:YMbdry} above, there is a complex
  gauge transformation $g \in H^2(\Sig, G)$, with $g|_{\partial
    \Sig}\equiv \Id$ and such that $F_{gA_\infty}=0$. The complex
  gauge transformation $g \in H^2 \hra W^{1,p}$ acts continuously on
  the space of $L^p$ connections by Lemma \ref{lem:gcactona}. So, the
  sequence of connections $gk_iA_i$ converges to $gA_\infty$ in
  $L^p(\Sig)$. By Lemma \ref{lem:toflat} above, for large $i$, there
  exists a sequence $\xi_i \to 0$ in $W^{1,p}(\Sig,\mathfrak{k})$ satisfying
  $\xi_i|_{\partial \Sig}=0$ and so that $e^{i\xi_i}gk_iA_i$ is a flat
  connection weakly.  By the Sobolev embedding theorem, there is a
  $C^0$ bound on $e^{i\xi_i}$ and hence also on
  $e^{i\xi_i}gk_i$. Therefore $\cup_ie^{i\xi_i}gk_iu_i(\Sig)$ is not
  bounded.

  Now, we derive a contradiction using the subharmonicity of $f$,
  where $f:X \to [0,\infty)$ is a proper function satisfying the
  condition equivariant convexity condition \eqref{eq:convexinfty}.
  By the unboundedness of the sequence $e^{i\xi_i}gk_iu_i$, the sequence of maps $f \circ
  e^{i\xi_i}gk_iu_i$ is also unbounded.  For any $i$, suppose $f \circ
  (e^{i\xi_i}gk_iu_i)|_{\Sig}$ assumes its maximum value at a point
  $x_i$. For large $i$, the maximum value is necessarily attained in
  the interior of $\Sig$ because $e^{i\xi_i}gk_iu_i(\partial \Sig)$,
  which is equal to $k_iu_i(\partial \Sig)$ is contained in the
  compact set $S_\partial$. Denote
  $m_i:=f(e^{i\xi_i}gk_iu_i(x_i))$. We may assume $m_i > c_0$. Suppose
  $U$ is a contractible open neighborhood of $x_i$ on which $f \circ
  (e^{i\xi_i}gk_iu_i) \geq c_0$. We will now show that $f \circ
  (e^{i\xi_i}gk_iu_i)$ attains the value $m_i$ on all of $U$. There is
  a gauge transformation $k_i'\in W^{1,p}(\ol U,K)$ be so that
  $k_i'e^{i\xi_i}gk_iA_i$ is the trivial connection on $\ol U$. Then,
  $\delbar (k_i'e^{i\xi_i}gk_iu_i)=0$. Since $f$ is $K$-invariant
  $f\circ (e^{i\xi_i}gk_iu_i)=f\circ (k_i'e^{i\xi_i}gk_iu_i)$ on $\ol
  U$. By \eqref{eq:subharmonic}, $f\circ (k_i'e^{i\xi_i}gk_iu_i)$ is
  sub-harmonic.  By the mean value inequality from complex analysis
  (Proposition 7.7.4 in the book Greene-Krantz \cite{Krantz}), a
  subharmonic function attains its maximum value on the boundary of
  the domain, therefore $f \circ (e^{i\xi_i}gk_iu_i) \equiv m_i$ on
  $U$. Therefore the set $\{f \circ (e^{i\xi_i}gk_iu_i)= m_i\}$ is
  open and closed in $\Sig$. This proves a contradiction because $m_i
  \to \infty$, the map $f:X \to [0,\infty)$ is proper and
  $k_i'e^{i\xi_i}gk_iu_i(\partial \Sig)$ is contained in a compact set
  $S_\partial \subset X$.
\end{proof}

\begin{proof}[Proof of Theorem \ref{thm:vstarget}] Consider a gauged
  holomorphic map $(A_0,u_0)$, and denote $k:=E(A_0,u_0)$. By Lemmas
  \ref{lem:ubd} and \ref{lem:ubd_bdry}, there is a compact set $S(k)
  \subset X$, such that the image of any gauged holomorphic map with
  energy $\leq k$ is contained in $S(k)$.  The heat flow trajectory
  starting at $(A_0,u_0)$, if it exists, would have its image
  contained in $S(k)$, because by Proposition \ref{prop:endec}, energy
  $E(A_t,u_t)$ decreases with $t$. Therefore, all the results about
  existence and convergence of heat flow - Theorems
  \ref{thm:flowintro}, \ref{thm:convintro}, \ref{thm:bdryintro} and
  \ref{thm:closedintro} can be applied with the target $S(k)$ instead
  of $X$.
\end{proof}

\section{Sobolev spaces}\label{sec:sobolevspaces}
The goal of this Section is to define Sobolev completions of
time-dependent sections of vector bundles and prove uniform bounds on
certain operators. The results of this section are used in the proof
of existence of heat flow in Section \ref{subsec:flowexist}.  Section
\ref{subsec:spacesec} introduces Sobolev completions of the space of
sections of vector bundles, including the case when Sobolev exponents
are negative or non-integral. The Sobolev norm is dependent on a
choice of connection. If the connection satisfies a curvature bound,
then the relevant operators between Sobolev spaces of sections will be
uniformly bounded. This is proved in Section \ref{subsec:unifbd} using
Uhlenbeck compactness. Section \ref{subsec:timedep} describes time
dependent sections. Next, in Section \ref{subsec:heateqn}, we show
that in these spaces, the solution of the heat equation has uniformly
bounded norm. Finally, in Section \ref{subsec:interchange}, we define
the space of time-dependent sections that are in the class $H^r$ in
the time direction and $C^0$ in the space direction. All the results
in this section that are not proved or explicitly cited can be found
in Lions-Magenes \cite{LM1}.

\subsection{Sections of vector bundles}\label{subsec:spacesec}
\subsubsection{Definition and basic properties}
In this section, we define\linebreak Sobolev completions of vector bundles associated to a principal $K$-bundle $P$ on a compact Riemann surface $\Sig$, that is equipped with a metric. Suppose $K$ is embedded in $SO(n)$. We consider bundles of the type $E=\wedge^\ell T^*\Sig \tensor (P \times_K \R^n)$. A smooth connection $A$ on the principal bundle $P$ and the Levi-Civita connection on $T\Sig$ together determine a covariant derivative $\nabla_A$ on $E$. For a non-negative integer $s$, we recall that the space $H^s(\Sig, E)$ (also referred to as $H^s(E)$ or $H^s$ if  the other data is obvious) is the completion of $\Gamma(\Sig,E)$ under the norm
\begin{align}\label{eq:sobnorm}
\Mod \sigma^A_s := \left( \sum_{i=0}^s \Mod{\nabla^i_A \sigma}^2_{L^2}\right)^{1/2}, \quad  \sigma \in \Gamma(\Sig,E).
\end{align}
\begin{remark} The space $H^s(E)$ can alternately be defined as the equivalence classes of almost-everywhere defined sections $\sig$ that satisfy $\nabla^i_A \sigma \in L^2$ for $0 \leq i \leq s$. The derivatives $\nabla_A$ are taken in the distributional sense. The space of smooth sections is dense in $H^s(E)$. 
\end{remark}
The following properties are well known. For $s_2 < s_1$, the inclusion 
\begin{equation}\label{eq:includespace}
H^{s_1}(E) \hra H^{s_2}(E)
\end{equation}
is continuous. The operators
\begin{equation}\label{eq:Da}
\nabla_A:H^{s}(E) \lra H^{s-1}(E \tensor T^*X), \quad \nabla_A^*:H^{s}(E\tensor T^*X) \lra H^{s-1}(E)\hspace{-1ex}
\end{equation} 
are continuous by the definition \eqref{eq:sobnorm} of $\Mod{\cdot}_s$. We recall that $\nabla_A^*$ is the same as $\nabla_A$ followed by the contraction $T^*X \times T^*X \rightarrow \R$. 

\subsubsection{Interpolation}

Sobolev completions of non-integral indices are defined by interpolation.
\begin{definition}
The complex Banach spaces $X_0$ and $X_1$ form a {\em compatible pair} if they are subspaces of a Hausdorff topological vector space $\X$. In that case, $X_0+X_1$ and $X_0 \cap X_1$ are also Banach spaces. An {\em interpolation space} $X$ is a Banach space for which the inclusions $X_0 \cap X_1 \subset X \subset X_0+X_1$ are continuous and which satisfies the following: if $L:X_0+X_1 \to X_0+X_1$ is a linear operator for which $L|_{X_i}:X_i \to X_i$ is a bounded map for $i=0,1$, then $L|_X$
is a bounded map from $X$ to itself. It is an interpolation space of exponent $\theta$ if there exists constant $C$ such that $$\Mod{L}_X \leq C \Mod{L}_{X_0}^{1-\theta} \Mod{L}_{X_1}^{\theta} \quad \text{ for all such operators }L.$$ Further, if $C=1$, then $X$ is an exact interpolation space.
\end{definition}
The complex interpolation functor $I_\theta$ produces an exact interpolation space of exponent $\theta$ (see \cite{LM1}, \cite{triebel}). We describe this method of obtaining interpolation spaces. Let $S$ be the strip $\{z \in \C: 0<Re(z)<1\}$. Let $\H(X_0,X_1)$ denote the space of functions $f:\ol{S} \to X+Y$ with the following properties:
\begin{itemize}
\item $f$ is holomorphic on $S$,
\item $\eta \mapsto f(i \eta)$ is a bounded continuous function from $\R$ to $X$ and 
\item $\eta \mapsto f(1+i \eta)$ is a bounded continuous function from $\R$ to $Y$.
\end{itemize}
The space $\H(X_0,X_1)$ is equipped with the norm
$$\Mod{f}_{\H}:=\max(\sup_{\eta \in \R}\Mod{f(i \eta)}_X, \sup_{\eta \in \R}\Mod{f(1+i \eta)}_Y).$$
By the {\it three lines theorem}, $\H$ is a Banach space.
\begin{definition} [Complex Interpolation] \label{def:complexint} Let $X_0$, $X_1$ be a compatible pair of complex Banach spaces. For $0<\theta<1$,
$$[X_0,X_1]_\theta:=I_\theta(X_0,X_1):=\{ a | \exists f \in \H(X_0,X_1): f(\theta)=a \}$$ 
with norm $\Mod{a}_{I_\theta(X_0,X_1)}=\inf \{\Mod{f}_\H|f(\theta)=a \}$.
\end{definition} 
Complex interpolation is a functor. This means that given compatible pairs $(X_0,X_1)$ and $(Y_0,Y_1)$ and a linear map $L:X_0+X_1 \to Y_0+Y_1$ such that the restriction $L|_{X_i}$ is a bounded map from $X_i$ to $Y_i$ for $i=0,1$, the restriction $L|_{[X_0,X_1]_\theta}$ is a bounded map from $[X_0,X_1]_\theta$ to $[Y_0,Y_1]_\theta$ and it satisfies
$$\Mod{L}_{[X_0,X_1]_\theta,[Y_0,Y_1]_\theta}\leq \Mod{L}_{X_0,Y_0}^{1-\theta}\Mod{L}_{X_1,Y_1}^{\theta}.$$ 
Sobolev spaces with non-integral indices are defined by complex interpolation. 
\begin{definition} [Fractional Sobolev spaces]
For an integer $n$ and $0<\theta<1$, $H^{n+\theta}(E):=I_\theta(H^n(E), H^{n+1}(E))$.
\end{definition}
We remark that the Sobolev spaces $H^n(\Sig,E)$ are not complex. But, while applying complex interpolation, we can instead use the spaces\linebreak $H^n(\Sig,E \oplus iE)$, and after interpolating, take the real part of the result. 
For $s_1$, $s_2 \geq 0$, and $0 < \theta < 1$, the map
\begin{equation}\label{eq:interpolpos} 
I_\theta(H^{s_1},H^{s_2}) \to H^{\theta s_1+(1-\theta)s_2}
\end{equation}
is an isomorphism. The operators in \eqref{eq:includespace} and \eqref{eq:Da} are bounded for all $s>0$.

For $s>\dim \Sig/2$, there is an embedding 
\begin{align}\label{eq:embedspace}
H^s(\Sig,E) \hra C^0(\Sig,E).
\end{align}

\subsubsection{The spaces $H_0^s$, $H_\partial^s$} The boundary trace
map $H^s(\Sig,E) \to \linebreak H^{s-\hh}(\partial \Sig, E|_{\partial
  \Sig})$ is well-defined and continuous for $s>\hh$.  Let
$C^\infty_0(\Sig,E)$ denotes the space of smooth sections supported
away from the boundary of $\Sig$. For any $s>\hh$, we define
$H_\partial^s(\Sig,E)$ to be the subspace of $H^s(\Sig,E)$ consisting
of sections whose boundary trace vanishes.  For a non-negative integer
$m$, $H_0^m(\Sig, E)$ is defined as the closure of $C^\infty_0$ in
$H^m(\Sig, E)$. For non-integral exponents, the space $H_0^s$ is
defined by interpolation. That is, for $0<\theta<1$,
$H_0^{m+\theta}(\Sig,E):=[H_0^m,H_0^{m+1}]_\theta.$

\begin{remark}[Alternate characterization of $H^s_0$]\label{rem:lmspace}
  If $s \neq \mu + \hh$, where $\mu$ is an integer, the spaces $H_0^s$
  can be directly defined as the closure of $C^\infty_0(\Sig,E)$ in
  $H^s(E)$. These spaces can be alternately characterized as : $\sig
  \in H^s_0$ if and only if $\sig \in H^s$ and $\frac {\partial^j
    \sig}{\partial \nu^j}=0$ on $\partial \Sig$ for $j=0, \dots ,
  \lfloor s-\hh \rfloor$. So, for $0<s<\hh$, $H_0^s = H^s$.

  However, if $s=\mu+\hh$, $H_0^s(\Sig, E)$ is a strict subspace of
  the closure of $C^\infty_0$ in $H^s(\Sig, E)$, with a finer
  topology. The space $H_0^{\mu+1/2}$ is called the Lions-Magenes space and is
  not closed in $H^{\mu+1/2}$. We will talk about these spaces more in
  the 1-dimensional case in Section \ref{subsec:timedep}.  Our
  notation here is different from \cite{LM1}, where
  $H_0^s(\Sig, E)$ is defined as the closure of $C^\infty_0$ in $H^s(\Sig, E)$
  for all $s$. The space $[H^\mu,H^{\mu+1}]_{1/2}$ is called $H^{\mu+1/2}_{00}$ in \cite{LM1} (see Theorem 11.7, Chapter 1).
\end{remark}
The $H^s_0$ spaces are well-behaved in terms of interpolation. For
$s_1,s_2 \geq 0$ and $0<\theta<1$,
\begin{equation}\label{eq:interpolpos0}
  I_\theta(H_0^{s_1},H_0^{s_2}) \to H_0^{\theta s_1+(1-\theta)s_2}
\end{equation}
is an isomorphism.

\subsubsection{Defining $H^{-s}$ by duality}
\begin{definition} For any $s \geq 0$, $H^{-s}(E)$ is the dual $(H_0^s(E))^*$, i.e. $H^{-s}(E)$ is the completion of $\Gamma(\Sig,E)$ under the norm
\begin{equation}\label{def:dual}\Mod{\sigma}_{-s}:=\sup\left\{\int_\Sig (\sigma,\sigma') :\sigma' \in H_0^s(\Sig,E), \Mod{\sigma'}_s=1\right\}.
\end{equation} 
\end{definition}
Elements in $H^{-s}$ need not be sections that are defined almost everywhere, they are just distributions. 
\begin{notation}
We use the notation $H_*^s$ in statements that apply to both $H^s$ and $H_0^s$.
\end{notation}

 Using the above duality, we have
\begin{proposition}\label{prop:opboundspace}
The operators in \eqref{eq:includespace} and \eqref{eq:Da} are continuous for all $s$, $s_1$, $s_2$.
\end{proposition}
By duality, the expected interpolation results also hold for $H^{-s}$ spaces.

\begin{proposition}[Multiplication Theorem] The map 
\begin{equation}\label{eq:multspace}
H_*^{s_1}(E_1) \tensor H_*^{s_2}(E_2) \lra H_*^{s_3}(E_1 \tensor E_2)
\end{equation} 
is continuous if $s_1+s_2 \geq 0$, $s_3 < \min(s_1,s_2)$ and $s_3 \leq s_1+s_2-\frac{\dim \Sig}{2}$.
\end{proposition}
This is a Corollary of the corresponding result on $W^{m,p}$ spaces, Proposition \ref{prop:sobmult}.
\subsection{Uniform operator bounds}\label{subsec:unifbd}
So far, we have used a smooth connection $A$ to define spaces $H^s_*$. The spaces $H^s_*$ are still well-defined for $s \in [-2,2]$ if we use a $H^1$ connection instead. The next proposition shows that different choices of connection produce equivalent norms.
\begin{proposition}\label{prop:Anorm}
  Let $A \in H^1$ be a connection on $P$ (and hence $E$). We assume
  that $B$ is a smooth connection and that the spaces $H^s(E)$ are
  Sobolev completions under the norm $\Mod{\cdot}_s^B$.
  \begin{enumerate}
  \item For $s \in [-1,2]$, the operator $\nabla_A:H^s(E) \to
    H^{s-1}(E)$ is continuous.
  \item For $s \in [-2,2]$, $\Mod{\cdot}_A$ defines a norm and is
    equivalent to $\Mod{\cdot}_B$.
  \end{enumerate}
\end{proposition}
\begin{proof}
  Let $a:=A-B \in \Omega^1(X,P(\mathfrak{k}))_{H^{1}_B}$. For $s=1,2$, if $\sigma \in H^s_B(E)$, $\nabla_A \sigma = \nabla_B \sigma +
  [a,\sigma]$. By the multiplication theorem $\Mod{[a,\sigma]}_{s-1}^B
  \leq \Mod{a}_{1}^B \Mod{\sigma}_s^B$. This fact is used to prove that $\Mod{\cdot}^A_s \leq c\Mod{\cdot}^B_s$ for
  $s=0,1,2$. The result is trivial for $s=0$, since both norms are
  just the $L^2$-norms. Assuming the result for $s-1$,
  \begin{equation*}
    \begin{split}
      \Mod{\sigma}^A_s & \leq \Mod{\sigma}^A_{s-1}+\Mod{\nabla_A \sigma}^A_{s-1} \leq c(\Mod{\sigma}^B_{s-1}+ \Mod{\nabla_A \sigma}^B_{s-1}) \\
      & \leq c(\Mod{\sigma}^B_{s-1} + \Mod{\nabla_B \sigma}^B_{s-1} + \Mod{a}_{1}^B \Mod{\sigma}^B_s) \leq c\Mod{\sigma}^B_s.
    \end{split}
  \end{equation*}
  The other direction $\Mod{\cdot}^B_s \leq c\Mod{\cdot}^A_s$ can be
  proved similarly. The result extends to all $s \in [-2,2]$ by
  duality and interpolation. The boundedness of the operator $\nabla_A : H^s \to H^{s-1}$ follows in an obvious way by using the norm $\Mod{\cdot}_s^A$.
\end{proof}
  
Although the topology of the Hilbert spaces $H_*^s(\Sig,E)$ is independent of the choice of connection used to define the norm, the operator norms depend on the connection. However, if the connection
satisfies a curvature bound $\Mod{F(A)}_{L^2}<\kappa$, then the
operator norm bounds depend only on $\kappa$ and not on the choice of
connection. Constants that depend only on $\kappa$ will be denoted
$c_\kappa$. We will also use terms like {\it $c_\kappa$-bounded}, {\it
  $c_\kappa$-isomorphism} etc. to say that the relevant operator norms
are bounded by $c_\kappa$.
The operator norms will be shown to be uniformly bounded using an alternate definition of $H^s$ involving local
trivializations of the principal bundle.
\subsubsection{Local trivialization definition of
  $H^s$-spaces} \label{subsec:loctriv} It is possible to define the
spaces $H^s$ using a local trivialization of the bundle : roughly,
$\Mod{\sig}_s$ will be the sum of its $H^s$-norms in each co-ordinate
patch. Different choices of trivialization would produce equivalent
norms. We will pick a trivialization that would produce a norm that is
$c_\kappa$-equivalent to $\Mod \cdot_s$ using Uhlenbeck's local
theorem stated below.

\begin{lemma}\label{lem:uh}{\rm(Uhlenbeck Compactness, \cite[Lemma 3.5]{Uh:compactness}, \cite[Theorem B]{Weh:Uh})} Suppose $P \to \Sig$ be a principal $K$-bundle on a Riemann surface $\Sig$.
  Given a constant $\kappa>0$, there exists a finite cover
  $\{\U_\alpha\}_\alpha$ of $\Sig$ and constants $c_\kappa$ such that
  for a $H^1$ connection $A$ on $P$ satisfying
  $\Mod{F(A)}_{L^2}<\kappa$, there are local trivializations
  $\tau_\alpha:P|_{\U_\alpha} \to \U_\alpha \times K$ with transition
  functions $g_{\alpha\beta}:\U_\alpha \cap \U_\beta \to K$ such that
  if $(\tau_\alpha^{-1})^*A=d+a_\alpha$,
  $$\Mod{a_\alpha}_{H^{1}(\U_\alpha)} \leq c_\kappa \quad\text{and}\quad
  \Mod{g_{\alpha \beta}}_{H^2(\U_\alpha \cap \U_\beta)} \leq
  c_\kappa.$$
\end{lemma}

We fix a partition of unity $\eta_\alpha$ subordinate to the cover
produced by Uhlenbeck compactness above. Local trivializations of the
principal bundle $P$ induce local trivializations of the associated
vector bundle $E$.  Given a section $\sig \in \Gamma(\Sig,E)$, let
$\sigma_\alpha:= \phi_\alpha \circ \sigma:\U_\alpha \to \R^m$
represent $\sigma|_{\U_\alpha}$ under the above trivialization. For
any $s \in [-2,2]$, define another norm on $H^s(\Sig,E)$ as
\begin{align}\label{eq:trivnorm}
  |\sigma|_s:=\left( \sum_\alpha \Mod{\eta^{1/2}_\alpha\cdot
      \sigma_\alpha}^2_{H^s(\U_\alpha,\R^m)}\right)^{1/2}.
\end{align}
The $L^2$ product corresponding to the norms $|\cdot|$ and
$\Mod{\cdot}$ agree
\begin{equation}
  \label{eq:l2pair}
  \int_X(\sigma',\sigma)dV=\sum_\alpha \int_{\U_\alpha}\eta_\alpha(\sigma'_\alpha, \sigma_\alpha) dV.  
\end{equation}

\begin{remark} \label{rem:normdual} For any $s \geq 0$, the dual of
  the space $(H_0^s,|\cdot|_s)$ with respect to the $L^2$-pairing
  \eqref{eq:l2pair} is $c_\kappa$-isomorphic to
  $(H^{-s},|\cdot|_{-s})$. This is because the constants on the norm
  bounds depend only on the covering $\{\U_\alpha\}_\alpha$ and the
  partition of unity $\eta_\alpha$, both of which are determined by
  $\kappa$.
\end{remark}

\begin{proposition} \label{prop:localtrivnorm} Given $\kappa>0$, there
  are constants $c_\kappa$ so that : if $A$ is a connection on $P$
  satisfying $\Mod{F_A}_{L^2}<\kappa$, for any $s \in [-2,2]$, the
  norms $\Mod{\cdot}^A_s$ and $|\cdot|_s$ are $c_\kappa$-equivalent on
  $H^s(E)$. The norm $|\cdot|_s$ is defined by \eqref{eq:trivnorm} and
  is produced by the trivialization given by Lemma \ref{lem:uh}.
\end{proposition}
\begin{proof}
  We first prove the result for non-negative integers by
  induction. For $s=0$, $|\sigma|_{L^2} = \Mod{\sigma}_{L^2}$. We
  assume the estimate is true for $s-1$ and prove $|\sigma|_s \leq
  c_\kappa\Mod{\sigma}_s$
  \begin{align*}
    |\sigma|_{H^s(E)}^2&=|\sigma|_{L^2}^2 + \sum_\alpha \Mod{\nabla(\eta_\alpha^{1/2} \cdot \sigma_\alpha)}^2_{H^{s-1}(\U_\alpha,\R^m)} \leq \Mod \sigma_{L^2}^2 + \Mod{\nabla_A \sigma}^2_{H^{s-1}(E)}\\
&\quad + \sum_\alpha \Mod{a_\alpha \times \sigma_\alpha}^2_{H^{s-1}(\U_\alpha,\R^m)} + \sum_\alpha \Mod{(\nabla \eta_\alpha^{1/2}) \cdot \sigma_\alpha)}^2_{H^{s-1}(\U_\alpha,\R^m)}\\
   & \leq \Mod \sigma_{H^s(E)}^2 + c_\kappa|\sigma|^2_{H^{s-1}(E)}
    \leq c_\kappa \Mod \sigma_{H^s(E)}^2.
  \end{align*}
  The other direction i.e. $\Mod{\sigma}_s \leq c_\kappa |\sigma|_s$
  is similar to the proof of Proposition~\ref{prop:Anorm}.  We have
  exact interpolation isomorphisms for both norms $\Mod \cdot$ and
  $|\cdot|$, so the result extends to all positive $s$. By Remark
  \ref{rem:normdual}, it extends to negative $s$ by duality.
\end{proof}
In this norm defined using local trivializations, operator norms do
not depend on $A$. So, using the $c_\kappa$-equivalence, we get the
following result.
\begin{proposition}\label{prop:curvbound}
  Given $\kappa>0$, there exist constants denoted by $c_\kappa$ such
  that if $\Mod{F(A)}_{L^2}^A< \kappa$, then the multiplication
  operator \eqref{eq:multspace} and Sobolev embedding
  \eqref{eq:embedspace} have norm $\leq c_\kappa$ and interpolation
  operators \eqref{eq:interpolpos}, \eqref{eq:interpolpos0} are
  $c_\kappa$-isomorphisms for Sobolev indices in the range $[-2,2]$.
\end{proposition}
\begin{remark}
  In the proof of Proposition \ref{prop:curvbound}, there is an
  additional detail in the bound for the multiplication operator: for
  $(\sigma,\sigma') \mapsto \sigma \tensor \sigma'$, $(\sigma \tensor
  \sigma')_\alpha$ depends on $g_{\beta
    \alpha}(\sigma'_\beta|_{\V_\beta \cap \tau_\beta^{-1}\U_\alpha})$.
  These terms can be bounded using the bound on $g_{\beta \alpha}$.
\end{remark}

\subsection{Time dependent sections}\label{subsec:timedep}

\subsubsection{Sobolev spaces over time intervals} To define Sobolev completions of the space of time-dependent sections, we first define
Sobolev completions of functions from a time interval $[0,T]$ to a
Hilbert space $\H$. The definition is standard. The only new idea is
that we introduce a $T$-dependent scaling. This is for technical
reasons and its usefulness will be pointed out later.
\begin{definition} \label{def:normsum} For any $m \in \Z_{\geq 0}$, $H^m([0,T],\H)$ is the completion of $C^\infty([0,T],\H)$
  in the norm
$$\Mod{\sig}_m:= \left( \sum_{i=0}^m \Mod{T^{-(m-i)}\frac{d^i}{dt^i}f}^2_{L^2} \right)^{1/2}.$$
For non-integral indices, $H^r$ is defined by interpolating between
neighbouring integers. For negative indices, Sobolev spaces are defined as duals\linebreak $H^{-r}([0,T],\H):=(H_0^r([0,T],\H))^*$ with respect to the $L^2$-pairing, which we define as $\langle f,g \rangle_{L^2} \mapsto \int_0^{T}\lan f(t),g(T-t)\ran_{\H}dt$.
\end{definition}
Alternately, this norm can be defined by Fourier transform, also using
a $T$-scaling.
\begin{definition}\label{def:fourier}{\rm(Fourier transform definition
    of $H^s([0,T],\H)$)} For any $s \in \R$,
$$\Mod{f}_{H^s([0,T],\H)}:=\inf \Mod{(T^{-2}+\tau^{2})^{s/2}\hat{F}(\tau)}_{L^2},$$ where the infimum is taken over all smooth $F:\R \to \H$ that restrict to $f$ in $[0,T]$.
\end{definition}
We need another subspace here.
\begin{definition} [$H_P^s$] Let $C^\infty_P([0,T],\H)$ be the
  subspace of smooth functions that are supported away from $t=0$
  (i.e. all derivatives vanish at $t=0$). For a positive integer $m$,
  $H^m_P:=$ closure of $C^\infty_P$ in $H^s$. The definition is
  extended to non-integers by interpolation and to negative numbers by
  duality : $H^{-s}_P:=(H^s_P)^*$ under the pairing $\langle f,g
  \rangle \mapsto \int_0^{T}\lan f(t),g(T-t) \ran_{\H}dt$.
\end{definition}
\begin{remark} The spaces $H^s_0$, $H^s_P$ and $H^s$ coincide if $0 \leq s< \hh$. By duality,
  $H^s_P=H^s$ for $-\hh <s \leq 0$ also. For $s\leq -\hh$, $H^s_P$ is
  a formal space whose elements may not correspond to distributions.
\end{remark}
\begin{remark} Continuing Remark \ref{rem:lmspace}, the space
  $H^{\mu+1/2}_P([0,T])$ can alternately be defined as the subspace of $H^{\mu+1/2}([0,T])$ consisting of elements $f$ for which $t^{-1/2}f^{(\mu)}
  \in L^2$. The space $H^{\mu+1/2}_P([0,T])$ has the norm
$$\Mod{f}_{H^{\mu+1/2}_P}=\left (\Mod{f}_{H^{\mu+1/2}}^2+\Mod{t^{-1/2}f}_{L^2}^2 \right )^\hh,$$
see Theorem 11.7, Chapter 1 in \cite{LM1}.
The topology is finer than that of $H^{\mu+1/2}$, so it is not closed
in $H^{\mu+1/2}$. Similarly, the norm of $H^{\mu+1/2}_0([0,T])$ is
equivalent to
$$\Mod{f}_{H^{\mu+1/2}_0}=\left (\Mod{f}_{H^{\mu+1/2}}^2+\Mod{t^{-1/2}f}_{L^2}^2 + \Mod{(T-t)^{-1/2}f}_{L^2}^2 \right )^\hh.$$
\end{remark}

We now state some properties of the spaces $H^s([0,T],\H)$. For $s_1>s_2$, the inclusion
\begin{equation}\label{eq:includetime}
  H_*^{s_1}([0,T],\H)\lra H_*^{s_2}([0,T],\H)
\end{equation}
is compact and has norm $cT^{s_1-s_2}$. The advantage of
the scaling is that by choosing small $T$, we have a handle on how
small a perturbation this operator can cause.
A bounded linear map $L:\H \to \H'$ between two Hilbert spaces, induces the following continuous operator
\begin{equation}\label{eq:bundlemap}
  H_*^s([0,T],\H) \lra H_*^s([0,T],\H').
\end{equation}
Its norm is determined by $\Mod{L}$. For $r>\hh$, the Sobolev embedding
\begin{align}\label{eq:embedtime}
   H^s([0,T],\H) \hra C^0([0,T],\H)
\end{align}
is a compact operator with norm bounded by $cT^{s-\hh}$. The multiplication theorem follows from the multiplication theorem for
real valued functions (Proposition \ref{prop:sobmult}). The multiplication operation
\begin{equation}\label{eq:multiplicationtime}
  H_*^{s_1}([0,T],\H) \tensor H_*^{s_2}([0,T],\H') \lra H_*^{s_3}([0,T],\H \tensor \H')
\end{equation}
is continuous if $s_1+s_2 \geq 0$, $s_3 < s_1+s_2-1/2$ and $s_3 \leq
\min(s_1,s_2)$. It has norm $\leq cT^{s_1+s_2-s_3-1/2}$. 
As in Section \ref{subsec:spacesec}, there is an isomorphism of interpolation spaces 
\begin{equation}\label{eq:interpolationtime}
 H^{\theta s_0 + (1-\theta)s_1}([0,T],\H) \lra I_\theta(H^{s_0}([0,T],\H),H^{s_1}([0,T],\H)).
\end{equation} 

\begin{lemma}[Integration]\label{lem:integrationtime} The differentiation operator
  $\ddt:H_P^{s+1} \to H_P^s$ is invertible, the inverse is given by
  the integration operator $\int_0$.
\end{lemma}
\begin{proof} The integration operator $f \mapsto \int_0 f(t)dt$ defined
  on $C^\infty_P([0,T])$ extends to a bounded
  operator $\int_0:H_P^n \to H_P^{n+1}$ for integers $n \geq 0$, using
  the Definition \ref{def:normsum} of the norm. The result follows by
  interpolation and duality.
\end{proof}
\begin{remark} For $s>-\hh$, the integration operator $\int_0:H^s_P
  \to H^{s+1}_P$ corresponds to ``real integration''. Otherwise it is
  a formal operator. This ties in with the fact that for $f \in H^s$,
  one can evaluate $f(0)$ only if $s>\hh$.
\end{remark}
\subsubsection{Mixed spaces}
We can define the following mixed spaces to describe time-dependent
sections of vector bundles.
\begin{definition} For any real $r$ and $s$,
  \begin{align*}
    H^{r,s}(\Sig \times [0,T],E)&=H^r([0,T],H^s(\Sig,E))\\
    H_{0,0}^{r,s}(\Sig \times [0,T],E)&=H_0^r([0,T],H_0^s(\Sig,E))\\
    H_{P,}^{r,s}(\Sig \times [0,T],E)&=H_P^r([0,T],H^s(\Sig,E)), \quad etc.
  \end{align*}
\end{definition}
If the spaces $H^s_*(E)$ are defined using a connection $A \in H^1$, then, $H^{r,s}_*$ is well-defined for all $r$ and $s \in
[-2,2]$. If there is a curvature bound $\Mod{F(A)}_{L^2}^A \leq
\kappa$, the uniform $c_\kappa$-bounds on operator norms extend in expected ways to mixed Sobolev spaces.
For example, the
multiplication map
\begin{equation}\label{eq:multiplicationmixed}
  H_*^{r_1,s_1}([0,T],E_1) \tensor H_*^{r_2,s_2}([0,T],E_2) \lra H_*^{r_3,s_3}([0,T],E_1 \tensor E_2)
\end{equation}
is well-defined and continuous if $r_1+r_2$, $s_1+s_2 \geq 0$, $r_3
\leq \min(r_1,r_2,r_1+r_2-\hh)$ and $s_3 \leq
\min(s_1,s_2,s_1+s_2-1)$. It has norm $\leq c_\kappa
t^{r_1+r_2-r_3-1/2}$.

\subsection{Heat equation}\label{subsec:heateqn}

At the center of solving the flow problem, lies the problem of
uniformly bounding the solution of a parabolic differential equation
on the space of sections of a vector bundle $E$. Throughout this
section, we fix a unitary connection $A \in H^1$ on the bundle $E$. We
consider the Laplacian operator $\Delta_A=\nabla_A^*\nabla_A$ on the
sections of $E$ and the parabolic operator $\ddt +\Delta_A$ on the
space of time-dependent sections of $E$.  Assuming a curvature bound
$\Mod{F_A}_{L^2}<\kappa$, we prove certain uniform $c_\kappa$-bounds
on the solution of the heat equation \eqref{eq:heat} below.  The heat
equation is solved using standard techniques (see Evans
\cite{Evans}), but we present the details in order to prove uniform $c_\kappa$-bounds on the solution. We use the operator norms
$\Mod{\cdot}_s:=\Mod{\cdot}_s^A$ for the sections of $E$.

\subsubsection{Laplacian equation}
Suppose $f \in \Gamma(\Sig,E)$ is a section. The Dirichlet elliptic boundary value problem is
\begin{equation}\label{eq:heat}
  \begin{cases}
      (\Id + \Delta_A )\sigma=f  &\text{on $\Sig$}\\
      \sigma=0 & \text{on $\partial \Sig$},
    \end{cases} 
\end{equation}
where  $\sig \in \Gamma(\Sig,E)$.
Recall that $H^s_\partial := \{ \sigma \in H^s: \sigma=0$ on
$\partial \Sig \}$ for $s>\hh$, and $H^s_\partial:=H^s_0$ for
$\hh < s <\frac{3}{2}$.
\begin{proposition} \label{prop:ellipinv} The operator 
$\Id + \Delta_A :
  H_\partial^{s+1}(\Sig,P(\mathfrak{k})) \to H^{s-1}(\Sig,P(\mathfrak{k}))$ is invertible for $s \in (-\hh, 1]$.
\end{proposition}
\begin{proof} First, we consider $s\!=\!0$, i.e. $\Id \!+\! \Delta_A \!:\!
  H_\partial^1 \!\to\! H^{-1}$. For any $\sig, \sig' \in H^1_\partial$,
  \begin{align*}
\lan(1+ \Delta_A)\sigma, \sigma'\ran_{L^2(\Sig)}&=\lan \sigma,\sigma'\ran_{L^2}+\lan \nabla_A \sigma,\nabla_A \sigma'\ran_{L^2}\\
&=\lan \sigma,\sigma'\ran_{H^1} \leq \Mod{\sigma}_{H^1}\Mod{\sigma'}_{H^1}.
\end{align*}
Therefore, $\Mod{(1+ \Delta_A)\sigma}_{H^{-1}}=\Mod{\sigma}_{H^1}$. So,
the operator is injective. It is onto by the Riesz representation
theorem on $H_0^1$. For any $\tau \in H^{-1}$, there exists $\sigma
\in H^1$, so that $\lan \tau, \sigma'\ran_{L^2}=\lan \sigma,\sigma'\ran_{H^1}$ for all
$\sigma' \in H^1_\partial$. Then, $(1+ \Delta_A)\sigma =\tau$.

Injectivity of the operator for $s=0$ implies injectivity for $s=1$. We
know that, for any smooth connection $B$, the operator
$1+\Delta_B$ is onto. Let $a=B-A$, then
$$(1+\Delta_B)\sigma - (1+\Delta_A)\sigma = *[a \wedge *\nabla_A \sigma] + \nabla_A ^*[a, \sigma] + *[a \wedge *[a,\sigma]].$$ 
Using multiplication theorem, the right hand side is a compact
operator. So, $(1+\Delta_A)$ is Fredholm with index $0$, and
so it is onto.

The result extends to all $s \in [0,1]$ by interpolation. Dualizing
the\linebreak map gives $(1 + \Delta_A)^{-1} : (H_\partial^{s+1})^*
\to (H^{s-1})^*$. For $s \in (-\hh,\hh)$,
$(H_\partial^{s+1})^*=\linebreak (H_0^{s+1})^*=H^{-s-1}$ and
$(H^{s-1})^*=H_0^{-s+1}=H_\partial^{-s+1}$ and the result follows for
$s \in (-\hh,s_0]$.
\end{proof}

\begin{proposition} \label{prop:laplaceinv} In Proposition \ref{prop:ellipinv}, the inverse map $(\Id+\Delta_A)^{-1}:\linebreak H^{s-1} \to H_\partial^{s+1}$ has norm $\leq c_\kappa$.
\end{proposition}
For the proof, we use an elliptic regularity result in Euclidean space
(Ch.~2, Theorem 5.1 \cite{LM1}): Let $V \subseteq \R^n$ be a bounded
open set, and $L$ be a smooth elliptic operator on $V$. For $m \in \Z_{\geq 0}$,
\begin{equation}\label{eq:euclideanelliptic}
  \Mod{u}_{H^{m+2}(V)} \leq c(\Mod{Lu}_{H^m(V)}+ \Mod{ru}_{H^{m+3/2}(\partial V)} + \Mod u_{H^{m+1}(V)}),
\end{equation}
where $r$ denotes restriction of a function to the boundary $\partial
V$.
\begin{proof}[Proof of Proposition \ref{prop:laplaceinv}] We work with local trivializations described by Lemma
  \ref{lem:uh}. For a section $\sig:\Sig \to E$, on a chart
  $\U_\alpha$, 
$$(\Delta_A \sig)_\alpha = \Delta \sig_\alpha + [d \sig_\alpha, A_\alpha] + [\sig_\alpha, dA_\alpha] + [A_\alpha , [A_\alpha, \sig_\alpha]].$$
Assume, $\sig|_{\partial \Sig}=0$. Then, 
$(\eta_\alpha^{1/2} \sig)_\alpha$ vanishes on $\partial \U_\alpha$. 
Using \eqref{eq:euclideanelliptic}, on each $\U_\alpha$, we have
\begin{align*}
  \Mod{\eta_\alpha^{1/2} \sig_\alpha}_{H^s} & \leq c(\Mod{(\Id + \Delta) \eta_\alpha^{1/2} \sig_\alpha}_{H^{s-2}} + \Mod{\eta_\alpha^{1/2} \sig_\alpha}_{H^{s-1}})\\
  & \leq c\Mod{((\Id+\Delta_A) \eta_\alpha^{1/2} \sig)_\alpha}_{H^{s-2}} +
  c_\kappa\Mod{\eta_\alpha^{1/2} \sig_\alpha}_{H^{s-1}}.
\end{align*}
Since the norms 
$$\sig \mapsto \Mod \sig_{H^s(\Sig)},\quad \sig \mapsto
\sum_\alpha\Mod{\eta_\alpha^{1/2} \sigma_\alpha}_{H^s(\U_\alpha)}\quad \text{and}\quad \sig \mapsto
\sum_\alpha\Mod{\sigma_\alpha}_{H^s(\U_\alpha)}$$
are equivalent, we get : if
$\sig|_{\partial \Sig}=0$, then
$$\Mod \sig_{H^s(\Sig)} \leq c_\kappa( \Mod{(\Id+\Delta_A)\sig}_{H^{s-2}(\Sig)}+\Mod \sig_{H^{s-1}(\Sig)}).$$
The operator $(\Id+\Delta_A)^{-1}:H^{-1} \to H_\partial^1$
has norm 1. By induction, we get the result for all non-negative
integers $s$. As in the proof of Proposition \ref{prop:ellipinv}, we
get the result for all $s \in (-\hh, 2]$.
\end{proof}

The operator $(\Id+\Delta_A)$ is positive, self-adjoint and unbounded on $L^2$. We will now define its negative and fractional powers $(\Id + \Delta_A)^s$. Further, we define a family of spaces $\F^s \subseteq H^s$ such that $(\Id + \Delta_A)^s:\F^s \to L^2$ is a bounded isomorphism.
\begin{definition} For $-\thh<s<\hh$, $\F^s:=H^s$ and for $\hh<s \leq 2$, $\F^s:=H^s_\partial$.
\end{definition}
We know that the maps $(\Id+\Delta_A):\F^s \to \F^{s-2}$ are
isomorphisms. We next extend the result to fractional powers of
$(\Id+\Delta_A)$.
\begin{proposition}\label{prop:fraciso} For any $r,s \in \R$ such that $s,s+2r \in (-\thh,2]\bs \{\hh\}$, the map $(\Id+\Delta)^r:\F^{s+2r} \to \F^s$ is a
  $c_\kappa$-isomorphism.
\end{proposition}
To prove this result, we first need to show that $\{\F^s\}_s$ is a family
of interpolation spaces.
\begin{lemma}\label{lem:Finterpolate} For
  $\theta \in (0,1)$, $\F^{2\theta}$ is $c_\kappa$-isomorphic to $[L^2,H^2_\partial]_\theta$.
\end{lemma}
For the proof, we need the following result, which is Proposition 2.1 in Lions-Magenes \cite{LM1}.
\begin{lemma}\label{lem:dual_half}
  Let $H$ be a Hilbert space, whose dual is identified to itself via
  $\langle\cdot,\cdot\rangle_H$. Let $V \subseteq H$ be a dense
  subspace, and $V'$ be its dual via
  $\langle\cdot,\cdot\rangle_H$. Then, $V \subseteq H \subseteq V'$
  are dense inclusions and $[V,V']_{1/2}=H$.
\end{lemma}

\begin{proof}[Proof of Lemma \ref{lem:Finterpolate}]
  We apply Lemma \ref{lem:dual_half} with $V=H^2_\partial$, $H=H^1_0=\F^1$. The dual $V'$ is $c_\kappa$-equivalent to $L^2$. This is because $\langle
  u,(\Id+\Delta_A)v\rangle_{L^2}=\langle u,v\rangle_{H^1}$ for all $u,v
  \in H^2_\partial$ and so,
$$\Mod{u}_{V'}=\sup_{v \in H^2_\partial}\frac{\langle u,v\rangle_{H^1}}{\Mod{v}_{H^2_\partial}} = \sup_{v \in H^2_\partial} \frac{\langle u,(\Id+\Delta_A)v\rangle_{L^2}}{\Mod{(\Id+\Delta_A)v}_{L^2}}=\Mod{u}_{L^2},$$
where all the equalities mean $c_\kappa$-equivalences. Therefore,
$[H^2_\partial,L^2]_\hh=H^1_0$. The result follows by using the facts $[H^1_0,L^2]_\theta=H^{1-\theta}_0$ and
$[H^2_\partial,H^1_0]=H^{2-\theta}_\partial$, and by the reiteration
Theorem for interpolation.
\end{proof}
For positive self-adjoint operators, fractional powers of the operator
are well-defined and these behave well on interpolation spaces. The following result is a slight variation of Theorem 14.1, Chapter 1 in Lions-Magenes \cite{LM1}.
\begin{lemma} \label{lem:selfad} Let $i:X \hra Y$ be a compact inclusion of Hilbert spaces such that the image $i(X)$ is dense in $Y$. Suppose $\Lam:Y \to Y$ is  an unbounded
  positive self-adjoint operator on $Y$, whose restriction $\Lam:X \to Y$ is bounded and invertible. Then, for any $\theta \in [0,1]$, the map $\Lam^{1-\theta}:[X,Y]_\theta \to Y$ is bounded and invertible. The norms of $\Lam^{1-\theta}$ and $\Lam^{-1+\theta}$ are bounded by $c(\Mod{\Lam},\Mod{\Lam^{-1}})$.
\end{lemma}
\begin{proof} We first show that the operator $\Lam^z$ is well-defined
  for $z \in \C$. The composition $Y \xrightarrow{\Lam^{-1}} X
  \xrightarrow{i} Y$ is self-adjoint, positive and compact. Therefore,
  $Y$ has an orthonormal basis of eigen-sections of $\Lam$. Further,
  since all the eigen-values are positive, $\Lam^z$ is well-defined on
  $Y$.

We next show that $\Lam^{1-\theta}:[X,Y]_\theta \to Y$ is bounded.
For any $a \in [X,Y]_\theta$, consider a holomorphic function $f \in \H(X,Y)$ on the strip $\{z:\on{Re}(z) \in [0,1]\}$ such that $f(\theta)=a$
and $\Mod{f}_{\H(X,Y)} \leq 2\Mod{a}_{[X,Y]_\theta}$. Recall that $\eta \mapsto f(i\eta)$ and $\eta \mapsto f(1+i\eta)$ are bounded in $L^\infty(\R,X)$ and $L^\infty(\R,Y)$ respectively. Define $g(z):=\Lam^{-z}f(z)$ on the strip. By the boundedness of $\Lam^{-1}:X \to Y$, both $\eta \mapsto g(i\eta)$ and $\eta \mapsto g(1+i\eta)$ are bounded in $L^\infty(\R,X)$. Therefore $g$ is a bounded map from the strip to $X$ and
by the three-lines theorem, 
\begin{align*}
  \Mod{g}_{L^\infty(\theta + i\R,X)} &\leq \Mod{g}_{L^\infty(i\R,X)}^{1-\theta}\Mod{g}_{L^\infty(1 + i\R,X)}^\theta \\
&\leq (1-\theta)\Mod{g}_{L^\infty(i\R,X)} + \theta \Mod{g}_{L^\infty(1 + i\R,X)}.
\end{align*} 
Hence,
\begin{align*}
\Mod{\Lam^{-\theta}a}_X &\leq (1-\theta)\Mod{\Lam^{-1}}\cdot\Mod{f}_{L^\infty(i\R,Y)} + \theta \Mod{f}_{L^\infty(1 + i\R,X)}\\
& \leq c\Mod{f}_{\H(X,Y)} \leq 2c\Mod{a}_{[X,Y]_\theta}.  
\end{align*}
This proves that $\Lam^{-\theta}:[X,Y]_\theta \to X$ is bounded and so, the same is true for $\Lam^{1-\theta}=\Lam \circ \Lam^{-\theta}:[X,Y]_\theta \to Y$.

For the inverse map, consider $y \in Y$ and define $f \in \H(X,Y)$ as $f(z):=\Lam^{-z}y$. Then,
$$\Mod{\Lam^{-1+\theta}y}_{[X,Y]_\theta} \leq \Mod{f}_{\H(X,Y)}=\Mod{y}_Y + \Mod{\Lam^{-1}y}_X \leq c\Mod{y}_Y.\vspace{-2em}$$
\end{proof}

\begin{proof}[Proof of Proposition \ref{prop:fraciso}] 
We show that $(\Id+\Delta_A)^\theta:\F^{2\theta} \to L^2$ is a
  $c_\kappa$-\linebreak isomorphism for $0<\theta<1$. The other cases of the
  Theorem can be proved by compositions. This result follows by applying Lemma \ref{lem:selfad} with $X:=H^2_\partial(\Sig,P(\mathfrak{k}))$, $Y:=L^2(\Sig,P(\mathfrak{k}))$ and $\Lam:=\Id+\Delta_A$. The conclusion of the\linebreak Lemma is that for any $\theta \in [0,1]$, $(\Id+\Delta_A)^\theta:[H^2_\partial,L^2]_\theta \to L^2$ is a $c_\kappa$-\linebreak isomorphism. By Lemma \ref{lem:Finterpolate}, $[H^2_\partial,L^2]_\theta=\F^{2\theta}$, for $\theta \in [0,1]\bs \{\hh\}$, which proves the Proposition.
\end{proof}

The following is a consequence of Proposition \ref{prop:fraciso}.
\begin{proposition}\label{prop:eignorm} {\rm(A norm via eigen-sections of the Laplacian)}
  The eigen-sections $\{\sigma_i\}_{i \in I}$ of the Laplacian
  $\Delta_A$ form an orthonormal basis of\linebreak $L^2(\Sig,P(\mathfrak{k}))$. The
  eigen-sections $\sig_i$ are in $H^2_\partial$ and satisfy $\Delta_A
  e_i=\lambda_i e_i$.  For $s \in (-\thh,2]$ and $\sigma \in \F^s$,
  $\sig \mapsto \left(\sum_{i \in I} (1+\lambda_i)^s
    (\sigma,e_i)^2_{L^2} \right) ^{1/2}$ is a norm on $\F^s$, which is
  $c_\kappa$-equivalent to $\Mod{\cdot}_s^A$.
\end{proposition}
\begin{proof}
  The map $(\Id+\Delta_A)^{-1} : H^{-1} \to H_0^1$ is well-defined and
  bounded, and the inclusions $L^2 \hra H^{-1}$ and $H^1_0 \hra L^2$
  are compact. Therefore $(\Id+\Delta_A)^{-1}:L^2 \to L^2$ is a
  compact self-adjoint positive operator. So, it has a complete
  orthonormal system $\{e_i\}_{i \in I}$ of eigensections. These are
  eigen-sections for $\Delta_A$ also. By elliptic regularity, $e_i \in
  H^{2}_\partial$. By
 Proposition \ref{prop:fraciso}, for $s \in (-\thh,2]$ and $\sigma \in \F^s$,
  $c_\kappa^{-1}\Mod{\sigma}_s \leq \Mod{(I+\Delta_A)^{s/2}
    \sigma}_{L^2} \leq c_\kappa \Mod{\sigma}_s$. Therefore, 
  $\sig \mapsto \Mod{(I+\Delta_A)^{s/2} \sigma}_{L^2}$, which is same as $\sig \mapsto \left( \sum_{i \in I}
    (1+\lambda_i)^s (\sigma,e_i)^2_{L^2} \right)^{1/2}$, is a norm on~$\F^s$. 
\end{proof}
\begin{corollary}\label{cor:eignormtime}{\rm(Eigen-section norm for time-dependent sections)} A time-dependent section $\sig \in
  H^r([0,T],\F^s)$ can be written as $\sig=\sum_{i \in
    I}\sig_i(t) e_i$,\linebreak where $\sig_i \in H^r([0,T])$ and
  \begin{align}\label{eq:rseqnorm}
    c_\kappa^{-1}\Mod{\sig}_{r,s} \leq \left( \sum_{i \in
        I}(1+\lambda_i)^s\Mod{\sig_i}^2_{H^r([0,T])} \right)^{1/2} \leq
    c_\kappa \Mod{\sig}_{r,s}.
  \end{align}
\end{corollary}
\begin{proof} The components $\sig_i$ are given by
  $\sig_i:=(\sig,e_i)_{L^2(\Sig)}$. The operator\linebreak $\F^s(E) \to \R$
  mapping $\eta \mapsto (1+\lambda_i)^s(\eta,e_i)_{L^2(\Sig)}$ is
  bounded. By \eqref{eq:bundlemap}, it induces a bounded operator
  between $H^r_*(\F^s) \to H^r_*$ as well. Therefore, $\sig_i$ is in
  $H^r_*([0,T])$. The norm bound \eqref{eq:rseqnorm} follows from the
  norm bound in Proposition \ref{prop:fraciso} and
  \eqref{eq:bundlemap}.
\end{proof}
\subsubsection{Parabolic equation} Now, we consider the equation
\begin{equation}\label{eq:heateqn}
   \begin{cases}
      (\ddt + \Delta_A )\sigma =f   &\text{on $[0,T] \times \Sig$}\\
      \sigma =0  &\text{on $[0,T] \times \partial \Sig$}\\
      \sigma(0) = g & \text{on $\Sig$}.
    \end{cases}
\end{equation}
Here $\sig:[0,T] \times \Sig \to E$ is a time-dependent section. The Laplacian $\Delta_A$ is given by the connection $A
\in H^1$ on $P \to \Sig$, that satisfies the curvature bound
$\Mod{F_A}_{L^2(\Sig)}<c_\kappa$. We use standard methods, but get a
$c_\kappa$-bound on the solution.  The spaces $L^2(H^{2s}) \cap
H^s(L^2)$ are very natural to solve the heat equation, since the time
derivative is order 1 and space derivative is order 2.

\begin{lemma} \label{lem:heateqnbaseg} Let $s,s-2r \in (-3/2,
  2]$. Given $g \in \F^s$ and $f=0$, we can find a unique solution
  $\sig \in H^{\hh+r,s-2r}$ for \eqref{eq:heateqn}, with bound
  $\Mod{\sig}_{\hh+r,s-2r} \leq c_\kappa t^{-r}\Mod g_s$.
\end{lemma}
\begin{proof}
  The solution $\sig$ is calculated using the eigen-section basis from
  Proposition \ref{prop:eignorm}. By that result, $g$ can be written as
  $g=\sum_{i \in I}g_i e_i$, where $g_i \in \R$. Then, $\sig_i(t):=g_i
  e^{-\lambda_i t}$ is a solution of the differential equation
  $\frac{d \sig_i}{dt} + \lambda_i \sig_i = 0$ with initial condition
  $\sig_i(0)=g_i$. Hence, $\sigma:=\sum_{i \in I} \sigma_i(t) e_i$ is
  a solution of the heat equation \eqref{eq:heateqn}.  To bound the
  norm of the solution $\sig$, we use the eigen-value norm in
  Corollary \ref{cor:eignormtime}. It is enough to show that for each
  $i$,
    $$\Mod{(1+\lambda_i)^{-r} e^{-\lambda_i t}}_{H^r([0,T])} \leq cT^{-\hh} |(1+\lambda_i)|,$$
  which holds using $\Mod{e^{-\lambda_it}\chi_{[0,T]}}_{H^r([0,T])}
  \leq c(\lambda_i + T^{-1})^{r-\hh}$ and assuming\linebreak $T\leq 1$.
\end{proof}
\begin{remark}\label{rem:C0heat} 
  The operator $g \mapsto \sig$ is well-defined between the spaces
  $H^s \to C^0(H^s)$ and has norm $\leq c$. The proof is similar and
  follows from $\Mod{e^{-\lambda_i t}}_{C^0([0,T])} \leq 1$.
\end{remark}
\begin{lemma} \label{lem:heateqnbasef} Let $-\thh<s<\hh$. Given $f \in
  H_P^{r,s}$ and $g=0$, \eqref{eq:heateqn} can be solved uniquely for
  $\sig \in H_P^{r+1,s} \cap H_P^{r,s+2}$, with bound $\Mod
  \sigma_{H^{r+1,s} \cap H^{r,s+2}} \leq\linebreak c_\kappa\Mod f_{r,s}$.
\end{lemma}
\begin{proof}
  Similar to the proof of Lemma \ref{lem:heateqnbaseg}, the solution
  $\sig$ is calculated using the eigen-section basis from Proposition
  \ref{prop:eignorm}.  We write $f$ as $f=\sum_{i \in I}f_i(t) e_i$,
  where $f_i \in H_P^r([0,T],\R)$. The solution of the heat equation
  can be written as $\sig=\sum_{i \in I}\sig_i(t)$, where $\sig_i$ is
  a solution of the ODE $\frac{d \sig_i}{dt} + \lambda_i \sig_i = f_i$
  and $\sig_i(0)=0$. Therefore, $\sig_i$ is given by
  $\sig_i(t):=\int_0^t e^{-\lambda_i(t-s)}f_i(s) ds$. To bound $\sig$,
  we need to show
  \begin{equation}\label{eq:eigcoeff}
    |\sig_i|_{H^{r+1}_P([0,T])} + (1+\lambda_i)|\sig_i|_{H^r_P([0,T])} \leq c|f_i|_{H^r_P([0,T])}
  \end{equation}
for each
  $i \in I$.
  First assume $r\geq 0$. It is enough to prove the statement for $f_i
  \in C^\infty_P$. We prove it using the Fourier-transform definition
  of the norm of $H^r([0,T])$ (see definition \ref{def:fourier}). By
  this, there exists $F_i \in C^\infty_0(\R)$ that restricts to $f_i$
  on $[0,T]$, vanishes for $t<0$ and $\Mod{F}_{H^r(\R)} \leq
  2\Mod{f}_{H^r([0,T])}$. Let $S_i=F_i * e^{-\lambda_it}
  \chi_{[0,T]}$. Then $S_i$ vanishes for $t<0$ and restricts to
  $\sig_i$ on $[0,T]$ and
  $\Mod{(T^{-2}+\tau^2)^{r/2}\hat{S}_i}_{H^r(\R),T} \leq 2
  \Mod{\sig_i}_{H^r([0,T])}$. So, we need to prove
  $$\Mod{(T^{-2}+\tau^2)^{1/2}\hat{S}_i(\tau)}_{L^2(\R)} + |1+\lambda_i|\cdot \Mod{\hat{S}_i(\tau)}_{L^2(\R)} \leq c \Mod{\hat{F}_i(\tau)}_{L^2(\R)},$$
  which follows from observing that $\hat{S}_i(\tau)=\hat{F}_i(\tau)
  \widehat{e^{-\lambda_it}\chi_{[0,T]}}$ and
  $\widehat{e^{-\lambda_it}\chi_{[0,T]}} \leq (T^{-2}+\tau^2)^{-\hh}$,
  $\widehat{e^{-\lambda_it}\chi_{[0,T]}} \leq (1+\lambda_i)^{-1}$.  We
  have proved that the operators $f_i \mapsto \int_0^t
  e^{-\lambda_i(t-s)}f_i(s) ds$ and $f_i \mapsto (1+\lambda_i)\int_0^t
  e^{-\lambda_i(t-s)}f_i(s) ds$ are uniformly bounded between $H^r_P \to
  H^{r+1}_P$ and $H^r_P \to H^r_P$ respectively
  for $r\geq 0$. These operators are self-adjoint under the
  $L^2$ pairing. So the statement holds for negative $r$
  by duality.
\end{proof}

Finally, we prove an intermediate derivatives result, which allows us to say that if a time-dependent section is in $L^2(H^s) \cap H^{s/2}(L^2)$, then, it is in $H^{\theta s/2}(H^{(1-\theta)s})$ for any $\theta \in (0,1)$.
\begin{lemma}
  \label{lem:intder} Let $X$, $Y$ be Hilbert spaces as in Lemma
  \ref{lem:selfad} and $[X,Y]_\theta$ be a family of interpolation
  spaces for $\theta \in [0,1]$. Then for any real numbers $r_1<r_2$
  and $r_3:=(1-\theta)r_1 + \theta r_2$, there is an isomorphism 
  $$I_\theta(H^{r_1}([0,T],X),H^{r_2}_*([0,T],Y)) \to H^{r_3}([X,Y]_\theta).$$
\end{lemma}
\begin{proof} A map $f$ in the space $\H(H^{r_1}(X),H^{r_2}(Y))$ is a holomorphic map from the strip $\{z:\on{Re}(z) \in
  [0,1]\}$ to the Banach space $H^{r_1}(Y)$ such that the restrictions
  $f|_{i\R}$ and $f|_{1+i\R}$ are continuous and bounded maps from
  $\R$ to $H^{r_1}([0,T],X)$ and $H^{r_2}([0,T],Y)$
  respectively. Define another holomorphic map
  $g(z):=\Lambda^{-z}f(z)$ on the strip. The map is well-defined,
  because by Lemma \ref{lem:selfad} and \eqref{eq:bundlemap}, the map
  $\Lambda^{-z}:H^{r_2}(Y) \to H^{r_2}(Y)$ is well-defined for any $z$ on the strip. Now, we
  examine the boundaries of the strip. The maps
  \begin{equation}
    \label{eq:stripg}
    H^{r_2}(Y) \xrightarrow{\Lam^{-1-i\tau}} H^{r_2}(X), \quad  H^{r_1}(X) \xrightarrow{\Lam^{-i\tau}} H^{r_1}(X)  
  \end{equation}  
  are continuous for all $\tau \in \R$, with norms bounded by
  $\Mod{\Lam^{-1}}$ and $1$ respectively. Since the boundaries of the
  strip map continuously and boundedly to $H^{r_1}(X)$, the same is
  true of the interior. In particular, \eqref{eq:stripg} implies that $g
  \in \H(H^{r_1}(X),H^{r_2}(X))$. Then by the interpolation relation \eqref{eq:interpolationtime}, we can say that
  $g(\theta) \in H^{r_3}(X)$. Further, by the
  continuous map $\Lambda^\theta:X \to [X,Y]_\theta$ and~\eqref{eq:bundlemap}, we have $f(\theta) \in H^{r_3}([X,Y]_\theta).$
  The norm bound follows naturally from the above calculations in a
  similar way to Lemma \ref{lem:selfad}.

For the inverse map, consider 
$a \in H^{r_3}([X,Y]_\theta)$. Then, there is a holomorphic map $f \in \H(H^{r_1}([X,Y]_\theta), H^{r_2}([X,Y]_\theta))$ such that $f(\theta)=a$ and $\Mod{f} \leq 2\Mod{a}$. The map $g(z):=\Lambda^{-\theta+z}f(z)$ can be shown to lie in $\H(H^{r_1}(X), H^{r_2}(Y))$, which proves the result.
 \end{proof}

\begin{corollary}\label{cor:interpolmixed}
Suppose $\theta \in [0,1]$, $s_1, s_2, s_3 \in (-\thh,2]\bs \{\hh\}$, $r_1$, $r_2$, $r_3 \in \R$ are such that $r_1<r_2$, $r_3=\theta r_1 + (1-\theta)r_2$, $s_1>s_2$ and $s_3:=\theta s_1 + (1-\theta)s_2$. Then, the interpolation space  $I_\theta(H^{r_1}_*(\F^{s_1},H^{r_2}_*(\F^{s_2}))$ is $c_\kappa$-isometric to $H^{r_3}_*(\F^{s_3})$. Here $\F^s:=H^s(\Sig,E)$ for $s \in (-\thh,\hh)$ and $\F^s=H^s_\partial(\Sig,E)$ for  $s \in (\hh,2]$.
\end{corollary}
\subsection{Interchanging order of
  coordinates} \label{subsec:interchange} In this section, we define
spaces of sections with $r$ derivatives in the time co-ordinate and
continuous in the space co-ordinate. The spaces are denoted by
$H^r(C^0)$. Since $C^0$ does not have a good dual space, it is not
possible to define the spaces for negative $r$ in a natural way. To
circumvent this problem, we show that the space $H^{r,s}$ can be
defined with the order of co-ordinates $r$, $s$ reversed, as
$H^s(\Sig,H^r([0,T],E))$. Then $H^r(C^0)$ can be defined as $C^0(\Sig,
H^r([0,T],E))$ and this space has relevant properties like Sobolev
embedding $H^{r,1+\eps} \hra H^r(C^0)$. In this section, the spaces
with reversed co-ordinates will be denoted by $\ol H^{r,s}$, but this
notation will not be used once it is proved equivalent to $H^{r,s}$.

A {\em Hilbert bundle} $\pi:\H \to \Sig$ is a bundle on $\Sig$ with
fiber-wise inner product, whose fibres are isomorphic to a Hilbert
space $H$.  The bundle $\H$ is described by the following data: a
cover $\sqcup_\alpha U_\alpha$ of $\Sig$ and smooth transition
functions on intersections $g_{\beta \alpha}:U_\alpha \cap U_\beta \to
\Aut(H,H)$, where $\Aut(H,H)$ is the space of linear isomorphisms from
$H$ to $H$ that preserve the inner product. The Hilbert bundle is
defined as the quotient
\begin{align}\label{eq:hilbbundle}
  \H:=\bigsqcup_\alpha (U_\alpha \times H)/\sim,
\end{align}
where the equivalence $\sim$ is given by $U_\alpha \times H \ni (x,h)
\sim (x,g_{\beta \alpha}(x)h) \in U_\beta \times H$ for all $x \in
U_\alpha \cap U_\beta$.  A {\em section} of the Hilbert bundle
$\sig:\Sig \to \H$ is given by local sections $\sig_\alpha:U_\alpha
\to H$ that agree on intersections, i.e. $\sig_\beta=g_{\beta
  \alpha}\sig_\alpha$ on $U_\alpha \cap U_\beta$. For any $s \geq 0$,
a section $\sig$ is in $H^s(\Sig,\H)$ if for each coordinate chart
$\alpha$, $\sig_\alpha$ is in $H^s_0(U_\alpha,H)$. The space $H^s(\Sig,\H)$ has a norm
\begin{equation}\label{eq:hilbnorm}
  \Mod{\sigma}^2_{H^s(\Sig,\H)}:=\sum_\alpha \Mod{\eta^{1/2}_\alpha\sigma_\alpha}_{H^s_0(U_\alpha,H)}^2, \quad \sig \in H^s(\Sig,\H),
\end{equation}
that makes it a Hilbert space.  Sobolev sections for some indices can
still be defined if the transition functions $g_{\beta \alpha}$ are
not smooth. In our application, the structure group will be a finite dimensional subgroup of
$\Aut(H,H)$. So, it makes sense to consider transition
functions $g_{\beta\alpha}:U_\alpha \cap U_\beta \to \Aut(H,H)$ that
are in $H^2(U_\alpha \cap U_\beta)$. Then the spaces $H^s(\Sig,\H)$
are well-defined for $s \in [0,2]$ with norm given by
\eqref{eq:hilbnorm}.

We are now ready to define the spaces $H^s(\Sig, H^r([0,T],E))$ with
reversed coordinates. First, we recall the set-up of the preceding
sections, where $E \to \Sig$ is a vector bundle with a compact
structure group $K$ and a unitary $H^1$ connection $A$. The connection
has a curvature bound $\Mod{F(A)}_{L^2}<\kappa$. There is a
trivialization of the bundle $E$ over a cover $\sqcup_\alpha U_\alpha$
of $\Sig$, such that the connection matrices and transition functions
of $E$ are $c_\kappa$ bounded (see Lemma \ref{lem:uh} on Uhlenbeck
compactness). For any $r \in \R$, we define a Hilbert bundle
$H^r([0,T],E) \to \Sig$ whose fiber over $z \in \Sig$ is
$H^r_*([0,T],E_z)$. To express the bundle $H^r([0,T],E)$ in the form
\eqref{eq:hilbbundle}, we use the above trivialization of the bundle
$E|_{U_\alpha}$ corresponding to the connection $A$.  The transition
functions of the bundle $E$ induce transition functions $g_{\alpha
  \beta}$ of the Hilbert bundle. That is, for any $z \in U_\alpha \cap
U_\beta$, $g_{\alpha \beta}(z)$ is an element of $K$ and is
independent of time $t \in [0,T]$. We define the space
$\ol{H}_*^{r,s}(E)$ as the space of $H^s$-sections of the Hilbert
bundle $H^r([0,T],E) \to \Sig$. Since the transition functions
$g_{\beta\alpha}$ are in $H^2(U_\alpha \cap U_\beta)$, the spaces
$\ol{H}_*^{r,s}(E)$ are well-defined for $s \in [-2,2]$ and $r \in
\R$.

\begin{proposition} 
  Let $s \in [0,2]$ and $r \in \R$. The differentiation operator
  $\ddt:H_P^r([0,T],\allowbreak \R^m)\to H_P^{r-1}([0,T],\R^m)$
  induces an invertible operator $\ddt^\Sig:\ol H_P^{r,s} \to \ol
  H_P^{r-1,s}$. The inverse is induced by $\int_0$ on the fibres.
\end{proposition}
\begin{proof} The bundle $H^r_P([0,T],E) \to \Sig$ is trivializable on
  the open sets $U_\alpha \subset \Sig$. The fiberwise operator $\ddt$
  induces the map
  \begin{equation}\label{eq:Umap}
    \ddt^{U_\alpha}:H^s(U_\alpha,H_P^r([0,T],\R^m)) \to H^s(U_\alpha,H_P^{r-1}([0,T],\R^m)).
  \end{equation}
  On the intersection $U_\alpha \cap U_\beta$, $\ddt^{U_\alpha}$ and
  $\ddt^{U_\beta}$ agree, i.e. $g_{\alpha \beta}^{-1}(z)
  \ddt^{U_\alpha}g_{\alpha \beta}(z)= \ddt^{U_\beta}$ for all $z \in U_\alpha \cap U_\beta$.  Hence, they
  patch up to yield $\ddt^\Sig$ defined on $H^s(\Sig,H^r_P([0,T],E))$.
  The operator \eqref{eq:Umap} is linear on the fibres with norm $\leq
  c$ (see Lemma \ref{lem:integrationtime}). It is also identical on
  every fibre. So, the norm of~\eqref{eq:Umap} is $\leq c$, and hence
  the same is true of $\ddt^\Sig$.  Since $\int_0$ is the inverse of $\ddt$
  fibre-wise, the result follows for $s \geq 0$.
\end{proof}
\begin{proposition}For $r \in \R$ and $s\in [0,2]$, the identity map
  \begin{align}\label{eq:reviso}
    \ol H^{r,s} \to H^{r,s}
  \end{align}
  is a $c_\kappa$-isomorphism.
\end{proposition}
\begin{proof}
  First, we consider the case when $r$ and $s$ are non-negative
  integers. In the proof of Proposition \ref{prop:curvbound}, we
  showed for $\sigma \in H^s(E)$,
$$c_\kappa^{-1} \Mod{\sigma}_s \leq |\sigma|_s \leq \Mod{\sigma}_s.$$
So, it is enough to show that
\begin{align}\label{eq:eucrev}
  H^s_0(U_\alpha,H_*^r([0,T],\R^m)) \simeq
  H_*^r([0,T],H^s_0(U_\alpha,\R^m))
\end{align}
with constants independent of $T$. These spaces are in fact identical
when $r$ and $s$ are non-negative integers since both are completions
of $C^\infty_{0,P}(U_\alpha \times [0,T],\R^m)$ under the same norm
$$\left(\sum_{i=0}^r \sum_{0\leq |\lambda| \leq s} \Mod{T^{-(r-i)}\frac{d^i}{dt^i}\frac{d^\lambda}{dx^\lambda} \sigma}^2_{L^2(U_\alpha \times [0,T])}\right)^{1/2}.$$
The spaces in \eqref{eq:eucrev} are equivalent for non-integers $r
\geq 0$ and $s \in [0,2]$ by interpolation.

Next, we prove the equivalence of $H_P^{-r,s}$ and $\ol H_P^{-r,s}$,
for non-negative $r$ and $0 \leq s \leq 2$ by induction on $r$. The
result is true for $-1 < r \leq 0$, which forms the base case of the induction. We get an isomorphism between
$H_P^{-r,s}$ and $\ol H_P^{r,s}$ by
$$H_P^{-r,s} \xrightarrow{\int_0} H_P^{-r+1,s} \xrightarrow {\simeq} \ol H_P^{-r+1,s} \xrightarrow {\ddt} \ol H_P^{-r,s}.$$
Here each of the arrows is an isomorphism with constants $\leq
c_\kappa$, and the middle arrow comes from the induction hypothesis.
\end{proof}
Reversing the order of space and time co-ordinates lets us define the
space $H^r(C^0)$ for any $r$.
\begin{definition}
  For any $r$, $H^r(C^0):=C^0(\Sig,H^r([0,T],E))$. It is the space of
  continuous sections of $H^r([0,T],E)$. Its norm is given by
$$\Mod{\sigma}_{r,C^0}:=\sup_{z \in \Sig}\Mod{\sigma(z)}_{H^r([0,T],E)}.$$
\end{definition}
This space satisfies the following properties: for any $r$, there is
an inclusion
\begin{align*}
  H^r(C^0) \hra H^r(L^2).
\end{align*}
Differentiation $\ddt$ is an invertible operator with inverse $\int_0$
between the following spaces
\begin{align*}
  H^r_P(C^0) & \xrightarrow{\ddt} H^{r-1}_P(C^0).
\end{align*}
There is a multiplication operator, for $r_3 \leq
min(r_1,r_2,r_1+r_2-\hh)$
\begin{align}\label{eq:c0mult}
  H_P^{r_1}(C^0)\tensor H_P^{r_2}(C^0) & \lra H_P^{r_3}(C^0).
\end{align}
For any $r$, there is an inclusion
\begin{align*}
  H_P^{r,1+\eps} & \hra H_P^r(C^0).
\end{align*}
The following result follows by the definition of $H^r(C^0)$.
\begin{proposition}
  If $\Mod{F_{A}}_{L^2}<\kappa$, all the above operators have norms
  bounded by $c_\kappa$.
\end{proposition}
\section{Composition of functions}\label{sec:compfunc}
\subsection{Composition of functions in Sobolev spaces}\label{subsec:compfunc}
Left composition by a smooth function induces a smooth map between
Sobolev spaces of functions. In this section, we discuss some
variations of this result, including one result that deals with
fractional Sobolev spaces. In all of this section, $U$ is a compact
connected subset of $\R^n$ with smooth boundary. The results of this section are used in relation to the following
operators defined in Section \ref{sec:heatflow}:
\begin{align}\label{eq:compeg}
  \nonumber  \Theta_1:\Gamma(\Sig,P(\mathfrak{k})) &\to \Gamma(\Sig,P(\End \mathfrak{k}))\\
  \xi &\mapsto (F \mapsto (\exp_{u_0}\xi)^*{d}\Phi (JF_{\exp_{u_0} \xi})-u_0^*{d}\Phi (JF_{u_0})),\\
  \nonumber  \Theta_2:\Gamma(\Sig,P(\mathfrak{k})) &\to \Gamma(\Sig,P(\End \mathfrak{k}))\\
  \nonumber \xi &\mapsto (F \mapsto ({d}\exp\xi)^{-1}
  (JF_{\exp_{u_0}\xi}) - JF_{u_0}).
\end{align}
\begin{proposition}[Composition of functions]\label{prop:compint} Let
  $l \in \Z_{\geq 0}$ and $\Psi:\R \to \R$ be a $C^l$ function
  satisfying $\Psi(0)=0$. For any integer $k\in [0, l]$ and $p>1$, the
  map
$$\F_\Psi:W^{k,p}(U) \cap C^0(U) \to W^{k,p}(U), \quad f \mapsto \Psi \circ f$$
is continuous and satisfies
$$\Mod{\Psi \circ f}_{W^{k,p}} \leq c\Mod \Psi_{C^k} \Mod{f}_{W^{k,p}}(1+\Mod{f}^{k-1}_{L^\infty}).$$
Further, if $kp>n$, $\F_\Psi$ is a $C^{l-k}$-map
between Banach spaces. If $l \geq k+1$, its derivative satisfies
$$\Mod{{d}\F_\Psi(f)} \leq c\Mod{\Psi}_{C^{k+1}}(1+\Mod{f}_{W^{k,p}})(1+\Mod f_{L^\infty}^{k-1}).$$
\end{proposition}
The differentiability of the operator $\Psi$ is in terms of
Fr\'echet-derivative. A map $L:V \to W$ between Banach spaces is {\em
  Fr\'echet-differentiable} at a point $x \in V$ if there is a linear
bounded function $dL_x:V \to W$ such that $\lim_{h \to
  0}\frac{\Mod{L(x+h)-L(x)-dL_x(h)}_W}{\Mod{h}_V}=0$. In the above
case, the derivative of $\F_\Psi$ at a point $f \in W^{k,p}(U) \cap
C^0(U)$ is the composition $\frac {d\Psi}{df} \circ f$. If $l \geq
k+1$, $\F_\Psi$ is a $C^{l-k}$ map if $\F_{d\Psi/df}$ is a
$C^{l-k-1}$-map.

Proposition \ref{prop:compint} is a slight variation of Proposition
B.1.20 in McDuff-Salamon. We use the space $W^{k,p}(U) \cap C^0(U)$
instead of placing the restriction $kp>n$. The proof is skipped,
because it is similar to a more general result Proposition
\ref{prop:nem} below. The result of Proposition \ref{prop:compint} can
be extended to fractional Sobolev indices.
\begin{proposition}[Composition of functions, fractional Sobolev
  indices]\label{prop:compfrac} Let $U \subset \R$. Let $\frac 1 p<s<1$ and $\Psi:\R \to \R$ be a $C^2$ function that satisfies
  $\Psi(0)=0$. Then,
$$\F_\Psi:W^{s,p}(U) \to W^{s,p}(U), \quad f \mapsto \Psi \circ f$$
is a $C^1$-map between Banach spaces that satisfies
$$\Mod{\Psi \circ f}_{W^{s,p}} \leq c\Mod \Psi_{C^1} \Mod{f}_{W^{s,p}}, \quad \Mod{{d}\F_\Psi(f)} \leq c\Mod{\Psi}_{C^2}(1+\Mod{f}_{W^{s,p}}).$$
\end{proposition}
\begin{proof}
  The proof uses the following equivalent norm for fractional Sobolev
  spaces $W^{s,p}(\R^n)$ (see Remark 4, p189 in Triebel
  \cite{triebel}): let $s=k+\sig$, $k$ is an integer and $0<\sig<1$:
$$\Mod{f}^p_{s,p} := \Mod{f}_{L^p}^p + \Mod{D^kf}_{L^p}^p + \int_{\R^n}\int_{\R^n}\frac{|D^kf(x)-D^kf(y)|^p}{|x-y|^{n+\sig p}}dx dy.$$
By substituting $k=0$ and $\sig=s$, and using Proposition
\ref{prop:compint}, we get
\begin{align*}
  \Mod{\Psi \circ f}^p_{W^{s,p}} &= \Mod{\Psi \circ f}^p_{L^p} + \int_{\R} \int_{\R} \frac {|\Psi(f(x))-\Psi(f(y))|^p}{|x-y|^{1+sp}}{d} x {d} y\\
  & \leq \bigMod{\frac {d\Psi} {df}}^p_{C^0}
  \left(\Mod{f}_{L^p}^p+\int_{\R}\int_{\R}\frac
    {|f(x)-f(y)|^p}{|x-y|^{1+sp}}{d} x {d} y\right) \\
&\leq c\bigMod{\frac
    {d\Psi} {df}}^p_{C^0}\Mod{f}^p_{W^{s,p}}.
\end{align*}
The bound on the derivative is a consequence of a similar bound on
$\frac {d\Psi} {df}$.
\end{proof}

The next result involving time-dependent sections in the Sobolev class $H^r(C^0)$ is used in the proof of the existence of heat flow
in Section \ref{subsec:flowexist}.
\begin{corollary} \label{cor:comptbd} Let $\hh<r<1$. Let $\Sig$, $X$
  be as in Section \ref{subsec:flowexist} and let $T>0$ be a
  constant. The bundle maps $\Theta_1$ and $\Theta_2$ in
  \eqref{eq:compeg} induce $C^1$-maps
  $$\F_{\Theta_i}:C^0(\Sig, H^r([0,T],P(\mathfrak{k}))) \to C^0(\Sig, H^r([0,T],P(\mathfrak{k})),$$
  for $i=1,2$. There is a constant $c(\Sig,X,\Phi)$ independent of $T$
  and $u_0$ such that
$$\Mod{\F_{\Theta_i}(\xi)}_{r,C^0} \leq c\Mod{\xi}_{r,C^0},\quad \Mod{{d}\F_{\Theta_i}(\xi)}_{r,C^0} \leq c(1+\Mod{\xi}_{r,C^0}).$$
\end{corollary}
\begin{proof} For any point $z \in \Sig$, the map $\Theta_i(z):\mathfrak{k}
  \to \End(\mathfrak{k})$ is a smooth map. By the compactness of $X$, we get a
  uniform bound for all $z$: $\Mod{\F_{\Theta_i(z)}}_{C^2} \leq
  c(X)$. By Proposition \ref{prop:compfrac}, it induces
$$\F_{\Theta_i(x)}:H^r([0,T],\mathfrak{k}) \to H^r([0,T],\End(\mathfrak{k})).$$
The norm of the operator and its derivative are uniformly bounded for
all $z \in \Sig$, from which the Corollary follows.
\end{proof}
\subsection{Sobolev extensions of smooth bundle maps}
A smooth bundle map between vector bundles induces maps between
Sobolev completions of sections of the vector bundles.  Suppose
$\pi_{E_1}:E_1 \to \Sig$, and $\pi_{E_2}:E_2 \to \Sig$ are vector
bundles and $\Psi:E_1 \to E_2$ is a smooth bundle map, i.e. it
satisfies $\Psi \circ \pi_{E_2}=\pi_{E_1}$ and the zero section is
mapped to the zero section. The bundle map $\Psi$ induces a map of
sections
$$\F_\Psi:\Gamma(\Sig,E_1) \to \Gamma(\Sig,E_2), \quad \xi \mapsto \F_\Psi \xi:=z \mapsto \Psi(\xi(z)), \quad z \in \Sig.$$
The maps $\Theta_1$ and $\Theta_2$ in \eqref{eq:compeg} are examples
of such maps between sections. We remark that $\F_\Psi$ can not be
viewed as a `composition of functions' operation as in Section
\ref{subsec:compfunc}.

Results on Sobolev completions of sections are obtained by working on
local trivializations.  Locally, the domain and target vector bundles,
$E_1$ and $E_2$ are the trivial bundles $U \times R^{m_1}$ and $U
\times \R^{m_2}$ respectively, where $m_1$, $m_2$ are positive
integers and $U \subset \R^n$ is a compact connected set of $\R^n$ and
has smooth boundary. The bundle map $\Psi$ can be locally written as 
$\Psi:U \times \R^{m_1} \to \R^{m_2}$, that satisfies
$\Psi(\cdot,0)=0$. Given a section $\xi:U \to \R^{m_1}$, we have
$\F_\Psi(\xi)(x):=\Psi(x,\xi(x))$.

\begin{proposition}[Local result for Sobolev extension of maps of
  sections]\label{prop:nem} Let $k \in \Z_{\geq
    0}$ and $p>1$. Suppose $\Psi:U \times \R^{m_1} \to \R^{m_2}$
  satisfies $\Psi(\cdot,0)=0$ and is in $C^l$, where $l \geq k$. For
  any smooth section $f \in \Gamma(U,\R^{m_1})$, suppose $\F_\Psi f
  \in \Gamma(U,\R^{m_2})$ be the section defined as $\F_\Psi
  f(x):=\Psi(x,f(x))$ for $x \in U$. Then $\F_\Psi$ extends to a
  continuous map
  \begin{equation}
    \label{eq:fsobmap}
    \F_\Psi:(W^{k,p} \cap C^0)(U,\R^{m_1})\to W^{k,p}(U,\R^{m_2}), 
  \end{equation}
  which satisfies
  \begin{equation}\label{eq:nemcont}
    \Mod{\F_\Psi(f)}_{W^{k,p}} \leq c\Mod \Psi_{C^k} \Mod f_{W^{k,p}}(1+\Mod{f}^{k-1}_{L^\infty}).
  \end{equation}
  Further, if $kp>n$, $\F_\Psi$ in \eqref{eq:fsobmap} is a
  $C^{l-k}$-map of Banach spaces. If $l \geq k+1$, its derivative
  satisfies
  \begin{equation}\label{eq:nemdiff}
    \Mod{d\F_\Psi(f)} \leq c\Mod{\Psi}_{C^{k+1}}(1+\Mod{f}_{W^{k,p}})(1+\Mod f_{L^\infty}^{k-1}).
  \end{equation}
  The constant $c$ is independent of $\Psi$ and $f$.
\end{proposition}
\begin{proof} [Proof of Proposition \ref{prop:nem}] It is enough to
  prove the result for $m_2=1$. We first show that $\F_\Psi$ is
  continuous at $f=0$. To bound $\Mod{\Psi(f)}_{W^{k,p}}$, we need to
  get an $L^p$-bound on terms of the form $\frac{\partial^I}{\partial
    x^I}\F_\Psi(f)$ where $I$ is a multi-index with $|I| \leq k$. For
  an index of length 1, we have $\frac{\partial}{\partial
    x_i}\F_\Psi(f)=\frac{\partial \Psi}{\partial f} \cdot
  \frac{\partial f}{\partial x_i}+\frac{\partial \Psi}{\partial x_i}.$
  So, $\frac{\partial^I}{\partial x^I}\F_\Psi(f)$ is a sum of terms of
  the form
  \begin{align}
    \label{eq:diffexp}
 &   \frac{\partial^{J+j} \Psi}{\partial x^{J} \partial f^j}\cdot
\left( \frac{\partial^{L_1} f}{\partial x^{L_1}}\right)\cdots
\left( \frac{\partial^{L_N} f}{\partial x^{L_N}}\right), \\
\notag &\text{where } |J|+j \leq k,|L_1|+\dots+|L_N| = j.
  \end{align}
Suppose
$j>0$. Denote $\ell_i:=|L_i|$ and let $p_i=jp/\ell_i$. Then,
\begin{align*}
  &\quad\ \left\lVert \frac{\partial^{J+j} \Psi}{\partial x^J \partial
      f^j}\right. \cdot \left( \frac{\partial^{L_1} f}{\partial
      x^{L_1}}\right)\cdots \left. \left( \frac{\partial^{L_N}
        f}{\partial x^{L_N}}\right) \right\rVert_{L^p}\\
  &\leq \Mod{\Psi}_{C^k} \Pi_{i=1}^N\bigMod{\frac{\partial^{L_i}f}{\partial x^{L_i}}}_{L^{p_i}}
 \leq \Mod{\Psi}_{C^k}\Pi_{i=1}^N \Mod f_{W^{\ell_i,p_i}}\\
 &\leq c \Mod{\Psi}_{C^k}\Pi_{i=1}^N \Mod f^{\ell_i/k}_{W^{k,p}} \cdot
  \Mod f_{L^\infty}^{1-\ell_i/k} \leq c \Mod{\Psi}_{C^k}\Mod{f}_{k,p}(1+\Mod
  f_{L^\infty}^{k-1}).
\end{align*}
The second bound is by H\"older's inequality and the third one is by
the Gagliardo-Nirenberg inequality (see Proposition B.1.18,
\cite{MS}). Suppose $j=0$. The term $\frac{\partial^J \Psi}{\partial
  x^J}$ vanishes for $f=0$. So, we have
\begin{align*}
  \bigMod{\frac{\partial^J \Psi}{\partial x^J}(x,f(x))}_{L^p} \leq
  \bigMod{\frac{\partial}{\partial f} \frac{\partial^J \Psi}{\partial
      x^J}}_{L^\infty}\cdot \Mod f_{L^2}=c\Mod f_{L^p}.
\end{align*}
This proves the inequality \eqref{eq:nemcont} and continuity of
$\F_\Psi$ at $f=0$. It is continuous at any $f \in W^{k,p}$ because
the operator $\Delta f \mapsto \F_\Psi(f+\Delta f)-\F_\Psi(f)$ is
continuous at $\Delta f=0$.

Now, we prove differentiability. We assume $kp>n$ and $l \geq
k+1$. For any $f \in W^{k,p}$, we claim that the derivative
${d}\F_\Psi(f)$ is given by $\frac{\partial \Psi}{\partial f} \circ
f$. That is,
the map
\begin{equation}
  \label{eq:dermult}
  {d}\F_\Psi(f):W^{k,p}(U) \to W^{k,p}(U), \quad   \Delta f \mapsto \Delta f \cdot \frac{\partial \Psi}{\partial f} \circ f
\end{equation}
is just multiplication by $\frac{\partial \Psi}{\partial f} \circ
f$. By the continuity result above, $\frac{\partial \Psi}{\partial f}
\circ f$ is in $W^{k,p}(U)$. Since $kp>n$, by Sobolev multiplication
(Proposition \ref{prop:sobmult}), \eqref{eq:dermult} is a bounded linear
operator that is bounded by
\begin{align}\label{eq:derbd}
  \Mod{{d}\F_\Psi(f)} \leq c\bigMod{\frac {\partial \Psi}{\partial f}
    \circ f}_{W^{k,p}} \leq
  c\Mod{\Psi}_{C^{k+1}}(1+\Mod{f}_{W^{k,p}})(1+\Mod{f}^{k-1}_{L^\infty}).
\end{align}
To obtain the last equality, we use the fact that $f \mapsto \frac
{\partial \Psi}{\partial f} \circ f$ is an operator similar to
\eqref{eq:fsobmap}. The only difference is that $\frac {\partial
  \Psi}{\partial f}$ does not vanish for $f=0$. But, we can apply the
continuity bound \eqref{eq:nemcont} on the map $f \mapsto \frac
{\partial \Psi}{\partial f}(f)-\frac {\partial \Psi}{\partial f}(0)$,
and the bound in \eqref{eq:derbd} now follows. Lastly, we show that
\eqref{eq:dermult} is indeed the Fr\'echet-derivative. This is because
\begin{align*}
 &\quad\ (\Psi(f+\Delta f)(x) - \Psi(f)(x) -d\F_\Psi(f)\Delta f)(x)\\
  &=\Delta f(x)^2 \int_0^1(1-t) \frac {\partial^2 \Psi}{\partial
    f^2}(x,(f+t\Delta f)(x))dt
\end{align*}
and $\Mod{\int_0^1(1-t) \frac {\partial^2 \Psi}{\partial
    f^2}(x,(f+t\Delta f)(x))dt}_{L^2} < c\Mod \Psi_{C^2}$.
\end{proof}

\begin{corollary}\label{cor:bundsob} Suppose $E_1$, $E_2 \to \Sig$ are vector bundles over a smooth manifold $\Sig$, possibly with boundary. Suppose $\Psi:E_1 \to E_2$ is a smooth bundle map. Suppose $k \in \Z_{\geq
    0}$ and $p>1$. The map $\Psi$ induces a continuous map of
  Sobolev completions
  \begin{equation}
    \label{eq:bundsob}
    \F_\Psi:\Gamma(\Sig,E_1)_{W^{k,p} \cap C^0} \to \Gamma(\Sig,E_2)_{W^{k,p}}.    
  \end{equation}
  Further, if $kp>\dim(\Sig)$, then $\F_\Psi$ is a smooth map.
\end{corollary}
\begin{proof}
After choosing local trivializations of $E_1$,
$E_2$ over a finite cover of $\Sig$, the result follows by using Proposition
\ref{prop:nem} on each element of the cover of~$\Sig$.
\end{proof}

\begin{corollary}\label{cor:bundsob_time} Assume the setting of Corollary
  \ref{cor:bundsob}. Suppose $k \in \Z_{\geq 0}$ and $p>1$.  The
  bundle map $\Psi$ induces bounded maps of time-dependent
  sections
 \begin{align}
  \label{eq:bund1}
  \quad\  \F_\Psi:L^p([0,T],W^{k,p}(\Sig,E_1)) &\cap C^0([0,T] \times \Sig, E)\\
  \notag &\to L^p([0,T],W^{k,p}(\Sig,E_2)). \\
\label{eq:bund2}    
\quad\ \F_\Psi:W^{k,p}([0,T],L^p(\Sig,E_1)) &\cap C^0([0,T] \times \Sig, E)\\
\notag &\to W^{k,p}([0,T],L^p(\Sig,E_2)).
  \end{align}
\end{corollary}
\begin{proof} The continuity of the operator \eqref{eq:bund1} follows by Corollary \ref{cor:bundsob} and \eqref{eq:bundlemap}. The operator \eqref{eq:bund2} is handled in an identical way to the operator in Corollary \ref{cor:comptbd}.
For any $z \in \Sig$, the operator $\F_{\Psi(z)}:W^{k,p}([0,T],(E_1)_z) \cap C^0 \to W^{k,p}([0,T],(E_2)_z)$ is a `composition of functions' operator, whose $L^\infty$ norm can be uniformly bounded for all $z$ using Proposition \ref{prop:compint}. The uniform $L^\infty(\Sig)$ bound implies a $L^p(\Sig)$ bound, proving the boundedness of \eqref{eq:bund2}.
\end{proof}

\appendix
\section{Some analytic results}
In this section, we collect some analytic results used at various places in the paper.
The following is 
Proposition A.3.4 in \cite{MS}.
\begin{proposition} \label{prop:impfnMS}{\rm(Implicit function theorem)} Let $Y_1$, $Y_2$ be Banach spaces, and $\S \subset Y_1$ be an open set containing the origin.  Let $\F: \S \to Y_2$ be a differentiable map. Suppose $D\F(0)$ is invertible and $\Mod{D\F(0)^{-1}} \leq C$. Let
  $\delta>0$ be a constant such that $B_\delta \subset \S$ and for all $x \in B_\delta$,
  $\Mod{D\F(x)-D\F(0)}<\frac 1 {2C}$. If $\Mod{\F(0)}<\frac \delta {4C}$,
  there is a unique $x \in B_\delta$ for which $\F(x)=0$.
\end{proposition}
The next result is a small addition to the above implicit function Theorem.
\begin{lemma}\label{lem:inj_convex} Let $Y_1$, $Y_2$, $\S$ and $\F$ be as in Proposition \ref{prop:impfnMS}. Suppose $\Mod{D\F(0)^{-1}} \leq C$, $\S$ is convex and $\Mod{DF(x)-DF(0)}<\frac 1 {2C}$ for all $x \in \S$, then, $\F$ is injective on $\S$.
\end{lemma}
\begin{proof} Let $\F_1:=D\F(0):Y_1 \to Y_2$ be a linear map and $\F_2:=\F-\F_1$ on $\S$. Then, we have $\Mod{D\F_2(x)}<\frac 1 {2C}$ for all $x \in \S$.
For any $x_1$, $x_2 \in \S$, the line segment joining $x_1$, $x_2$ is contained in $\S$. Then,
  \begin{align*}
    \Mod{\F_1(x_2) - \F_1(x_1)} &= \Mod{\F_1(x_2-x_1)} \geq \frac 1 C \Mod{x_1 -x_0}\\
    \Mod{\F_2(x_2) - \F_2(x_1)} & \leq \frac 1 {2C} \Mod{x_2 -x_1}.\\
    \implies \Mod{\F(x_2) - \F(x_1)} & \geq \frac 1 {2C} \Mod{x_2 -x_1}.
  \end{align*}
  which proves the result.
\end{proof}

\begin{proposition}{\rm (Sobolev multiplication)}\label{prop:sobmult} 
Let $\Sig$ be an  $n$-dimensional compact Riemannian manifold possibly with a smooth boundary. 
\begin{enumerate}
\item {\rm(\cite[Theorem 4.39]{Adams:sobspace})} Given $k \in \Z$ and $p>1$ be such that $kp>n$. Then, $W^{k,p}(\Sig)$ is a Banach algebra with respect to pointwise multiplication. There is a constant $c$ such that for any $f$, $g \in W^{k,p}(\Sig)$
  \begin{equation*}
    \Mod{fg}_{W^{k,p}} \leq c \Mod f_{W^{k,p}}\Mod g_{W^{k,p}}.
  \end{equation*}
\item {\rm(\cite[Theorem 9.5 (3)]{Palais})} Suppose $l \in \Z$, $k_i \in \Z_{\geq 0}$ and $q$, $p_i>1$, where $i=1, 2$. Suppose at least for one $i$, $k_ip_i<n$. Let $l \leq k_1, k_2$ and $l - \frac n q < \sum_{i=1,2}k_i - \frac n {p_i}$. Further, if $l<0$, then we assume $\sum_{i:k_ip_i<n}(\frac n {p_i}-k_i)<n$. Then, there is a constant $c$ such that for any $f \in W^{k_1,p_1}(\Sig)$, $g \in W^{k_2,p_2}(\Sig)$
  \begin{equation*}
    \Mod{fg}_{W^{l,q}} \leq c \Mod f_{W^{k_1,p_1}}\Mod g_{W^{k_2,p_2}}.
  \end{equation*}
\end{enumerate}
\end{proposition}
The next result is a slight extension of Uhlenbeck's compactness theorem (\cite{Uh:compactness}, \cite{Weh:Uh}).
\begin{proposition}\label{prop:uhhigh} {\rm(Uhlenbeck compactness for higher regularity connections)} Let $M$ be a $2$-dimensional compact Riemannian manifold, possibly with a smooth boundary. Let
  $p>1$. Further,
  let $P \to M$ be a principal $K$-bundle and $\{A_i\}_i$ be a sequence
  of $W^{k,p}$-connections on $P$ whose curvature satisfies a uniform bound $\Mod{F(A_i)}_{W^{k-1,p}(M)} < c$.
Then, there exists a sequence of gauge transformations $k_i \in
\K^{k+1,p}(P)$ and a connection $A_\infty \in \A^{k,p}(P)$
such that $k_iA_i$ weakly converges in $W^{k,p}$ to $A_\infty$.
\end{proposition}
\begin{proof}
The proof of Uhlenbeck compactness consists of a local theorem
followed by patching
arguments. We first recall the local result for $k=1$. Choose $1  < q < p$. The local theorem (theorem B in 
\cite{Weh:Uh}) says that there is a constant $\eps_{Uh}$ such that any point $m \in M$ has a neighborhood $U$ with smooth boundary satisfying the following:  
if a connection $A$ satisfies $\Mod{F_A}_{L^q(U)}$, then there is a gauge transformation $k$ that puts $A$ in Coulomb gauge, i.e. if $k(A)={d} + a$ then,
\begin{equation}
  \label{eq:coul}
  d^*a=0,\quad (*a)|_{\partial U_\alpha}=0, \quad
\Mod{a}_{W^{1,p}} \leq c\Mod{F_{A}}_{L^p}.
\end{equation} 
By a dilation
argument, it can be shown that given $\kappa>0$, there is a cover
$\Sig=\cup_\alpha U_\alpha$ such that for any connection $A$
satisfying a curvature bound $\Mod{F_A}_{L^p(\Sig)}<\kappa$, on the
sets $U_\alpha$, the $L^q$ bound of the curvature is smaller than
$\eps_{Uh}$, ensuring that the local theorem stated earlier is
applicable. 

To prove a corresponding result for $k>1$, we strengthen the above local result. We will prove that given a connection $A={d} + a$ in Coulomb gauge on an open $U$, i.e. $a$ satisfies \eqref{eq:coul}, then, 
\begin{equation}
  \label{eq:ihuh}
\forall K\, \exists c(K): \quad \Mod{F(A)}_{W^{k-1,p}} \leq K \implies \Mod{a}_{W^{k,p}}  \leq c(K).  
\end{equation}
We use an inductive argument to prove \eqref{eq:ihuh}. The statement is already true for $k=1$. So, we assume $k \geq
2$ and that \eqref{eq:ihuh} is true when $k$ is replaced by
$k-1$. We also assume $\Mod{F(A)}_{W^{k-1,p}} \leq K$. By the induction
hypothesis, we have $\Mod{a}_{W^{k-1,p}} \leq c(K)$. We first focus
on the case $(k-1)p > 2$. In that case, by Sobolev multiplication,
$[a \wedge a]$ has a $W^{k-1,p}$-bound. Then, by the formula
$F(A)={d} a + [a \wedge a]$, we have a $W^{k-1,p}$-bound on $da$. Since $a$ is in Coulomb gauge, by elliptic regularity (the proof of theorem
5.1 in \cite{Weh:Uh} carries over to the higher regularity case),
$$\Mod{a}_{W^{k,p}} \leq \Mod{{d} a}_{W^{k-1,p}} \leq c(K).$$
It remains to prove \eqref{eq:ihuh} when $(k-1)p \leq n$. In that case
 $k=2$ and $p \leq 2$. In fact, the only part of
the proof that remains is to show a $W^{1,p}$ bound on $[a \wedge
a]$ assuming a $W^{1,p}$-bound on $a$, a $W^{1,p}$-bound on ${d} a +
[a \wedge a]$ and the fact that $a$ is in Coulomb gauge. This is
done by a bootstrapping argument. There exists a number $\ell \geq 1$
and a sequence $p=q_0 < q_1 < \dots < q_{\ell-1} \leq 2 < q_\ell$ such
that for any $i\geq 1$, $q_i < q_{i-1}/(2-q_{i-1})$. Then, if
$a \in W^{1,q_{i-1}}$, by Sobolev embedding $a \in L^{2q_i}$. By
H\"older's theorem, $[a \wedge a] \in L^{q_i}$ and therefore ${d}
a$ is also in $L^{q_i}$. By the Coulomb gauge condition, elliptic
regularity yields $a \in W^{1,q_i}$. Starting with $a \in
W^{1,p}$ and applying these steps repeatedly, we get $a \in
W^{1,q_\ell}$, where $q_\ell>2$. Then, $[a \wedge a]$ is also in
$W^{1,q_\ell}$, and hence in $W^{1,p}$.

The arguments for patching gauge transformations on the open sets $U_\alpha$ are
identical for the higher regularity case - see Lemma
3.2, 3.3 in \cite{Uh:compactness}. The weak convergence then follows in a similar way to the case when $k=0$ which is proved in \cite{Uh:compactness} and \cite{Weh:Uh}.
\end{proof}

\end{document}